\documentclass[final]{siamltex}

\usepackage{amsmath}
\usepackage{amsfonts}
\usepackage{amssymb}
\usepackage{array}
\usepackage{epsfig}
\usepackage{graphics}
\usepackage{color}

\newcommand {\N} {\mathbb{N}}
\newcommand {\R} {\mathbb{R}}

\newcommand {\C} {\mathbb{C}}

\newcommand {\Z} {\mathbb{Z}}

\newcommand {\tr} {\text{tr }}

\newcommand{\rb}[1]{{\unskip\nobreak\hfil\penalty 50%
\hspace{1em}\hbox{}\nobreak\hfil\hbox{#1}%
\parfillskip=0pt\finalhyphendemerits=0\par}}
\DeclareRobustCommand{\qed}{\quad\ensuremath{\square}}

\title{A Lie-Group Approach to Rigid Image Registration}

\author{Martin Schr\"oter\footnotemark[1] 
        \and Uwe Helmke\footnotemark[1]
\and Otto Sauer\footnotemark[2]}

\begin{document}

\maketitle

\renewcommand{\thefootnote}{\fnsymbol{footnote}}
\footnotetext[1]{Mathematisches Institut, Universit\"at W\"urzburg, Am Hubland, 97074 W\"urzburg, Germany ({\tt helmke, schroeter@mathematik.uni-wuerzburg.de}).}
\footnotetext[2]{Klinik und Poliklinik f\"ur Strahlentherapie, Josef-Schneider-Str. 11,
97080 W\"urzburg, Germany ({\tt sauer\_o@klinik.uni-wuerzburg.de}).}

\begin{abstract}
The task of image restration is to find the spatial correspondence of two or more given images. In this paper we assume that the correspondence is given either by an Euclidean, or by an affine volume-preserving transformation. Since the registration problem can be seen as an optimization problem on a finite dimensional Lie group, we use a recently developed framework of approximate-Newton methods on manifolds, which leads to locally quadratically convergent algorithms. To reduce numerical costs, we present two strategies: One makes use of the quasi Monte Carlo Method and the other ends up with an algorithm acting on spline function spaces. An extension for multi-modal image registration is given as well.
\end{abstract}


\pagestyle{myheadings}
\thispagestyle{plain}
\markboth{M. SCHR\"OTER, U. HELMKE, O. SAUER}{A Lie-Group Approach to Rigid Image Registration}

\section{Introduction}
In this paper we study the task of rigid image registration as an optimization problem on a Lie group. Although standard formulations of the problem focus on two- or three-dimensional images, the subsequent mathematical analysis goes through in any dimension. Any gray-scale image is thus identified with its associated intensity function $f:\R^n\rightarrow\R$, which we assume to be a
three times continuously differentiable function with compact support. The rigid image registration task for two such functions $f,g$ then amounts to find the Euclidean transformation $(A,t)$ that minimizes the $L^2$-distance  
\begin{equation}
\label{leastsquares}
\int_{\R^n}\big(f(Ax+t)-g(x)\big)^2dx. 
\end{equation}
Since the set of rigid body transformations $ \rho_{A,t}:\R^n\rightarrow\R^n, x\mapsto Ax+t,$ forms a Lie group, the Euclidean transformation group $SE(n)$, we obtain a least squares optimization on $SE(n)$, which is solved here using appropriate-Newton algorithms. 

%%%%%%%%%%%%%%%%%%
Image registration is a fundamental task in image processing, with applications in various fields, including e.g. robotics \cite{Computer} and geophysics \cite{Geophysiks}. For medical image applications, registration is used for image-based treatment planning and image-guided treatment delivery, see e.g. \cite{Moders.-Uerbersicht}, \cite{Viergever-Uerbersicht}, \cite{Braun-Uerbersicht} and the references therein. Although an overwhelming number of publications focuses on non-rigid registration, where the task is to find a non-linear diffeomorphism, rigid registration is still of considerable interest and may lead to good starting point for a subsequent non-rigid registration phase; see e.g. \cite{AnwMed1}, \cite{AnwMed2}, \cite{AnwMed3}, \cite{AnwMed4}, and \cite{AnwMed5}. Often, as in e.g. \cite{Modersitzki}, \cite{Levenberg}, rigid image registration algorithms are designed, that employ fixed local coordinates of the Euclidean group via Euler-angles and then apply standard optimization algorithms on the affine parameter space. This simple local coordinate chart approach however has its drawbacks and may lead to ill-conditioned algorithms at the boundary of the parameter space. For example, in \cite{640-2} it is observed that the singularities inherent in local Euler angle coordinates may reduce the speed of convergence, compared to algorithms acting on the Lie group. Therefore Lee et al propose in \cite{640-2} a linearly convergent Nelder-Mead algorithm on the Lie group $SE(n)$ for multi-modal image registration. In order to achieve faster, local quadratic convergence rates we introduce a new type of approximate-Newton methods on $SE(n)$ that avoids singularities of Euler-angle coordinates. 

%%%%%%%%%%%%%%

Of course, the task of studying Newton's method on Lie groups is not new and has been already applied e.g. to robotics and computer vision problems; see e.g Park~\cite{Park}, H\"uper et al~\cite{EssentialMatrix} and Sastry~\cite{3DVision} for background material. However, such prior work suffers from a number of shortcomings that limit the applicability to image registration problems. The Riemannian Gauss-Newton method \cite{Absil} performs a Newton step via a line search along a geodesic. For the Euclidean rotation group $SO(n)$, such geodesics require the computation of the matrix exponential of a skew-symmetric matrix; which may be a formidable numerical task in high dimensions. Moreover, since the Euclidean group $SE(n)$ carries no natural bi-invariant Riemannian metric, such geodesics are described by solutions to nonlinear second order differential equations which are hard to compute. Even in the case of non-rigid registration, the Lie group structure of the set of diffeomorphisms get a key position in the performance of the algorithms. In \cite{Beg} it is mentioned that incorporating geodesics leads to an increasing of the accuracy. However, calculating the geodesics, which is equivalent to solve a time-varying ODE, is a time consuming procedure and several approaches are known in literature to avoid this step: For example, in \cite{Younes} vector fields are used which fulfill the momentum conservation equation, in \cite{log-Euclid} the authors uses one-parameter subgroups to approximate the geodesics and propose a fast method for calculating the vector field exponential. We refer to \cite{diffeos} or \cite{pennecII} and the reference therein for a further study of the task of non-parametric image registration.

%%%%%%%%%%%%%%

In this paper, we present a new class of approximate Newton algorithms that are taylored to the least squares optimization problem (\ref{leastsquares}) and avoid the above mentioned difficulties.  Following earlier work on Newton's method on manifolds by Helmke and Moore~\cite{Helmke&Moore}, Shub~\cite{Shub:86}, Manton~\cite{Manton}, H\"uper and Trumpf~\cite{640-3}, Helmke, H\"uper and Trumpf~\cite{grassmann} and also Absil~\cite{Absil}, in this paper, we use very simple local parameterizations of the Euclidean transformation group to compute an approximated version of the Hessian and to perform the Newton-step. These previous works mainly deal with minimizing trace-functions on $SO(n)$ or on its homogeneous spaces. In comparison, we study the action of the Euclidean transformation Group on an infinite dimensional vector space and give an extension to the Special Affine Group $SA(n)$. We will show local quadratic convergence of the algorithms under suitable genericity conditions, which also arise in the classical theory of Newton methods on vector spaces. Our algorithm seems to be new even for minimizing standard trace functions. We are not awear of similar algorithms on $SA(n)$.

A bottleneck in implementing such algorithms lies in the difficulty of effectively evaluating the higher dimensional integrals, which may suffer from the curse of dimensionality. In this paper we will present and discuss two different strategies to circumvent this problem. First, we use Quasi Monte--Carlo methods to approximate the integrals by taking samplings of the functions at suitable random points. This leads to an easily implementable algorithm in the case of image registration. 
Second, we used $B$-spline approximations of the images to get an approximation of the integral. The second method actually works better in practice as will be shown by examining concrete medical imaging tasks. Extensions to rigid registrations using volume preserving transformations are given, too.

\section{Local Parameterizations}
To construct a Newton-like algorithm on a Lie Group, we use local parameterizations and coordinate charts. The difference to previous work as e.g. \cite{Modersitzki}, \cite{Levenberg} is that we do not use fixed local coordinates, as our parameterizations change with the iteration points. This offers considerable advantages in designing the algorithm. Recall, that a local parameterization on an $n$-dimensional manifold $M$ is a family  
$\left\lbrace
\mu_p\right\rbrace _{p\in M}$ of smooth maps $\mu_p:\R^n\rightarrow M$
that satisfies $\mu_p(0)=p,~p\in M,$ and defines a local diffeomorphism around $0$. Such local parameterizations exist on every manifold and provide coordinate charts around each point of the manifold. In contrast to usual coordinate charts, local parameterizations vary with each point of the manifold; a trivial fact, that actually helps to simplify the construction of numerical algorithms considerably. On a Riemannian manifold, a standard set of local parameterizations are given by the so-called Riemannian normal coordinates that are defined via the Riemannian exponential map. In the sequel, our local parameterizations have the advantage of being more easily computable than the exponential map, although they may not allow such immediate Riemannian geometry interpretations.

\subsection{Parameterization of Volume Preserving Transformations}
For necessary background on Lie groups and Lie algebras we refer to \cite{Hilgert}. Any affine transformation of $\R^n$ is of the form  $x\mapsto Ax+t$ for a given transformation matrix $A\in \R^{n\times n}$ and a translation vector $t\in\R^n$. Let $SL(n)$ denotes the special linear group of matrices $A\in \R^{n\times n}$ with determinant $1$. Its Lie algebra then is $sl(n)$, the set of $n\times n$-matrices with trace zero. The group of affine, volume preserving transformations than can be identified with the Lie group of all $(n+1)\times (n+1)$-matrices of the form
\begin{eqnarray}
M=\left(\begin{matrix}
 A & t\\
0 & 1
\end{matrix}
\right),\label{Form_affine}
\end{eqnarray}
which define the  ``special affine group'' $SA(n)$. 
By identifying a vector $x\in \R^n$ with its homogenous coordinates 
\begin{eqnarray}
 \bar{x}=\left(
\begin{array}{c}
 x\\
 1
\end{array}
\right),\nonumber
\end{eqnarray}
the affine transformation $\rho_M:x\mapsto Ax+t$ then becomes the linear map 
\begin{eqnarray}
 \bar{x}\mapsto M \bar{x}=\left(
\begin{array}{cc}
 A & t\\
0 & 1
\end{array}
\right) \bar{x}.\nonumber
\end{eqnarray}

Using standard terminology from group theory, this just says that $SA(n)$ is the semidirect product $SL(n)\ltimes\R^n$. The associated Lie algebra $sa(n)$ of $SA(n)$ consists of all matrices
of the form 
\begin{eqnarray}
\left( 
\begin{array}{cc}
 \Omega & v\\
0 & 0
\end{array}
\right) \label{Form_Algebra}
\end{eqnarray}
with the condition that $\tr \Omega =0$ holds. Local parameterizations of $SA(n)$ are then given as 
\begin{align}
\begin{split}
&\mu_M:sl(n)\times \R^n \rightarrow SA(n)\\
&\mu_M(\Omega,v):=M\exp \left( 
\begin{array}{cc}
 \Omega & v\\
0 & 0
\end{array}
\right).
\end{split}
\label{Exp_SAn}
\end{align}

In order to construct a computationally more feasible local parameterization, we consider the first order approximation of this map. Thus we decompose each Lie algebra element of $sl(n)$ as 
\begin{eqnarray}
 X=X_l+X_d+X_u\nonumber
\end{eqnarray}
where $X_d$ is a diagonal matrix and $X_u, X_l$ are strictly upper and lower triangular matrices, respectively. Let $A_Q$ denote the $Q$-factor in the $QR$-decomposition of a matrix $A$. One can easily show that
\begin{eqnarray}
 \theta :sl(n)\rightarrow SL(n),~~~X\mapsto (I+X_l-X_l^{\top})_Q\left[ \exp(X_d)+X_u+X_l^{\top}
\right] \nonumber
\end{eqnarray}
is a first order approximation of $\exp|_{sl(n)}$. This leads to the
system of local parameterization for $SA(n)$
\begin{align}\begin{split}
&\nu_M^{QR}:sl(n)\times\R^n\rightarrow SA(n)\\
 &\nu_M^{QR} 
 (\Omega, v)
:=M\left(
\begin{array}{cc}
 \theta(\Omega) & (I+\frac{1}{2}\Omega)v\\
0 & 1
\end{array}
\right). \end{split} \label{Karte-SA(n)}
\end{align}

An important property of this parameterization which will prove useful in the sequel is that the derivative of $\nu_M^{QR}$ at the origin is the identity map.

%%%%%%%%%%%%%

\subsection{Parameterization of the Euclidean group}
Every rigid body transformation of $\R^n$ can be decomposed in a rotation around the origin and a translation $x\mapsto Ax+t$. Thus the Euclidean group parameterizes all affine transformations  $\rho_{A,t}(x)= Ax+t$, where $A\in SO(n)$ is a rotation matrix. Here $SO(n)$ denotes the compact Lie group of $n\times n$ real matrices $A$ satisfying $AA^{\top}=A^{\top}A=I$ and $\det A=1$. Using homogenous coordinates, this ``Euclidean transformation group'' becomes identified with the subgroup of $SA(n)$, consisting of all $(n+1)\times (n+1)$ matrices of the form (\ref{Form_affine}) with $A\in SO(n)$. Similarly, the associated Lie Algebra $se(n)$ consists of all matrices of the form (\ref{Form_Algebra}) in which $v\in\R^n$ and $\Omega$ is a skew-symmetric $n\times n$ matrix.

The matrix exponential map $\exp: se(n)\rightarrow SE(n)$ then provides us with a canonical map between the Lie algebra and the Lie group. This leads to the local parameterization around any $M\in SE(n)$ of the form (\ref{Form_affine}) as
\begin{align}\begin{split}
&\mu_M:so(n)\times \R^n \rightarrow SE(n)\\
&\mu_M(\Omega,v):=M\exp \left( 
\begin{array}{cc}
 \Omega & v\\
0 & 0
\end{array}
\right).\end{split}\label{Exp_SEn}
\end{align}

Note, that the computation of the matrix exponential is expensive for large scale matrices  but in the special cases $n=2,3$ explicit formulas such as that by Rodriguez (cf. e.g. \cite{3DVision} p.27) are available. Nevertheless, such explicit formulas are not very useful in an optimization algorithm and the simplified formulas we use is more numerical efficient  even in the low dimensional case. In the sequel, we  approximate $\exp \Omega$ by the the orthogonal part of the QR-factorization of $I+\Omega$. Note, that $I+\Omega$ is invertible for every skew-symmetric matrix $\Omega$, thus, $(I+\Omega)_Q$ is always well defined. This leads the local parameterization of the $SE(n)$ as:
\begin{align}\begin{split}
&\nu_M^{QR}:so(n)\times \R^n\rightarrow SE(n)\\
& \nu_M^{QR} (\Omega,v):=M\left( 
\begin{array}{cc}
 (I+\Omega)_Q & (I+\frac{1}{2}\Omega)v\\
0 & 1
\end{array}
\right). \end{split}\label{Karte-SE(n)}
\end{align} 

Note, that this map coincides with the previous map for $SA(n)$, when restricted to $so(n)\times \R^n$. In particular, $D\nu_M^{QR}(0)=id$ holds, i.e. $\nu_M^{QR}$ is locally diffeomorphic around $0$ and a valid first order approximation of the exponential map.

%%%%%%%%%%%%%%%%%%%%%%%%%%%%%%%%%%%%%%%%%%%%%%%%%%%%%%%%%%%%

\section{Quasi-Newton Method}
In this section, we propose a novel approximate-Newton algorithm for image registration that is based on the above local parameterizations. Our construction differs essentially from the well-known Riemannian Newton algorithm, which is based on knowledge of the geodesics to calculate the Hessian. Since the geodesics in  $SA(n)$ are available only implicitly via the solutions of complicated nonlinear second-order differential differential equation, we prefer to avoid the Riemannian Newton method. Instead, we adapt a version of the approximate-Newton method as developed by Shub \cite{Shub:86} and H\"uper and Trumpf \cite{640-3}, which has been already successfully applied for several optimization problems (see e.g.~\cite{grassmann}).

Let $\left\lbrace  \mu_M\right\rbrace _{M\in G}$ be a set of local parameterizations of a Lie Group $G$. Thus, $\mu_M$ is defined on an open neighborhood $U\subset\R^n$ of $0\in U$ such that $\mu_M:U\rightarrow G$ is diffeomorphic with $\mu_M(0)=M$. Additionally, we assume that $\mu(M,x):=\mu_M(x)$ is a smooth map. Let $\left\lbrace \nu_M\right\rbrace _{M\in G}$ be another set of local parameterizations, subject to the same conditions. The $(\mu,\nu)$ Newton-iteration on $G$ for a smooth objective function $\Phi:G\rightarrow \R$ then is defined as
\begin{equation}
 M_{k+1}=\nu_{M_k}\left( -\left(
     \mbox{Hess}_{\Phi\circ\mu_{M_k}}\left( 0\right) \right) ^{-1}\nabla
   ({\Phi\circ\mu_{M_k}}) \left( 0\right) \right),~~~~~~M_0\in G, \label{newton-G}
\end{equation}
where $\nabla h(0)$ and $\mbox{Hess}_h(0)$ denote the standard gradient and Hesse-operator of a smooth function $h:\R^n\to\R$, respectively. The iteration in (\ref{newton-G}) can be decomposed into the calculation of the Newton-step $d=-( \mbox{Hess}_{\Phi\circ\mu_{M_k}}(0)) ^{-1}\nabla ({\Phi\circ\mu_{M_k}})(0)$ and subsequent application of the local parameterization $M_{k+1}=\nu_{M_k}(d)$. The local parameterizations $\left\lbrace  \mu_M\right\rbrace _{M\in G}$ are used in (\ref{newton-G}) to calculate a classical Newton-step in Euclidean coordinates of $\R^n$. The second parameterization $\left\lbrace \nu_M\right\rbrace _{M\in G}$ acts as a retraction of the tangent space onto the manifold, to carry out the actual Newton step. A practical choice of $\mu$ would be via the Riemannian exponential map, while for the retraction $\nu$ any first order approximation of the exponential map would be sufficient. 
% With this combination, we overcome the computational expensive procedure of evaluating the exponential map in each iteration step. 
Local quadratic convergence of this method has been recently established; see \cite{640-3}, \cite{grassmann}.

In the case of the image-registration problem, $G$ is either the Euclidean transformation group $SE(n)$ or the Special Affine Group $SA(n)$. Moreover, $\left\lbrace  \mu_M\right\rbrace$ is chosen as the exponential map, while (\ref{Karte-SE(n)}), (\ref{Karte-SA(n)}) are chosen for the retraction map $\left\lbrace  \nu_M\right\rbrace_{M\in G}$ in $SA(n)$ and $SE(n)$, respectively. In this combination, neither the Riemannian nor the matrix exponential map have to be evaluated in a point different from zero. In order to reduce the numerical costs, we use the QR-factorizations in (\ref{Karte-SE(n)}) and (\ref{Karte-SA(n)}) for the update part of the iteration. The cost function to maximize on the Euclidean motion groups $G=SA(n),SE(n)$ is $\Phi:G\rightarrow \R$
\begin{equation}
\Phi(A,t):=\int_{\R^n}^{}f(Ax+t)g(x)dx.\label{Zielfunktion}
\end{equation}
By invariance of the integral under volume preserving maps, this function differs from the least squares index (\ref{leastsquares}) by a constant. In particularly, maximization of $\Phi$ is equivalent to minimization of (\ref{leastsquares}).

% \textcolor{red}{In spite of this equivalence, an optimization method for (\ref{leastsquares}) and (\ref{Zielfunktion}) can lead to quite different algorithms. We choose here the correlation measure (\ref{Zielfunktion}), since we will make use of the specific structure in section four. But the following framework can be applied for (\ref{leastsquares}) as well.
% One has also the choice among various optimization algorithms for solving (\ref{leastsquares}) or (\ref{Zielfunktion}) (see e.g. \cite{Absil} for an overview). For example, it is a quite classical approach to use a Gauss-Newton algorithms in order to exploit the special structure of (\ref{leastsquares}) (see e.g. \cite{haber}) and in \cite{Malis} an optimization technique is proposed which converges local quadratic in the special case when $f(x)$ is a rotated and translated version of $g(x)$. However, both  methods converge only linear in non-academic examples and a lot of additional optimization steps would be necessary to arrive at the accuracy of the Newton method, presented in this section.}

%%%%%%%%%%%%%%%%%%%%%%%%%%%%%%%%%%%%%%%%%%%%%%%%%%%%%%%%%%%%%%%%%%%%%%%%%

\subsection{Calculation of Gradient and Hessian}
The aim of this section is to compute the Newton-iteration (\ref{newton-G}) for the objective function~(\ref{Zielfunktion}). In each iteration-step, we have to consider the function $\Phi\circ\mu_M$ and its first and second derivatives.

\begin{lemma} Let $G$ denote either the Euclidean and affine Lie group $SE(n)=SO(n)\ltimes \R^n$, $SA(n)=SL(n)\ltimes\R^n$, respectively. Let $\mu_M$ denote the local parameterizations (\ref{Exp_SAn}), (\ref{Exp_SEn}) and let $\Phi$ denote the objective function (\ref{Zielfunktion}). Endow the Lie algebras $\frak{g}=sa(n),se(n)$ with their standard Euclidean inner product. For a fixed $M\in G$, we get:  \\
\begin{itemize}
 \item [(a)] The gradient of $\Phi\circ\mu$ in $0$ is $\nabla(\Phi\circ\mu_M)(0,0)=(\tilde{\Omega},\tilde{v})$ with 
\begin{align}
 \tilde{\Omega}=\int\limits_{\R^n}g_M(z)\pi_{\frak{k}}(\nabla f(z)z^{\top})dx & & \tilde{v}=\int\limits_{\R^n}g_M(z)\nabla f(z)dz.\label{Gradient}
\end{align}
Here $g_M:=g\circ\rho_{M^{-1}}$ and $\pi_{\frak{k}}$ denotes the projection from $gl(n)$ to the Lie algebra $\frak{k}=so(n)$ for $G=SE(n)$ and $\frak{k}=sl(n)$ for $G=SA(n)$, respectively:
\begin{align}
\pi_{so(n)}(X):=\frac{1}{2}(X-X^{\top}) & & \pi_{sl(n)}(X):=X-\frac{\tr X}{n}I_n.\nonumber 
\end{align}

\item [(b)] The Hessian operator $\mbox{Hess}_{\Phi\circ\mu_M}(0):\frak{g}\rightarrow \frak{g} $ of $\Phi\circ\mu_M$ at a critical point $M\in G$ is $\mbox{Hess}_{\Phi\circ\mu_M}(0)(\Omega,v)=(\hat{\Omega},\hat{v}) $ with
\begin{align}\begin{split}
&\hat{\Omega}
=\pi_{\frak{k}}
\left(  \frac{1}{2}\Omega^{\top}\int\limits_{\R^n}\nabla f(z)z^{\top}g_M(z)dz +\frac{1}{2}\int\limits_{\R^n}\nabla f(z)z^{\top}g_M(z)dz\Omega^{\top}\right.\\
&\left. +\int\limits_{\R^n}\mbox{H}_f(z)\Omega zz^{\top}g_M(z)dz
+\int\limits_{\R^n}\mbox{H}_f(z)vz^{\top}g_M(z)dz\right)\\
&\hat{v}=\int\limits_{\R^n} \mbox{H}_f(z)g_M(z)(\Omega z+v)dz 
\end{split}\label{SimpleHessian}
\end{align}
where $\mbox{H}_f$ denotes the matrix representation of the Hessian of $f$.   
\end{itemize}
\end{lemma}

\begin{proof}
We calculate the directional derivative of $\Phi\circ\mu_M$ as:
\begin{align}
 \frac{d}{d\tau}\Phi\circ\mu_M(\tau \Omega,\tau v)&=\int\limits_{\R^n}\nabla 
f\left( P\exp(\tau
\Omega_0) 
M \bar{x}
\right)^{\top} 
P \Omega_0
\exp\left( \tau \Omega_0\right)
M\bar{x}
g(x)dx\label{one}\\
 \mbox{with }  \Omega_0=\left( 
\begin{array}{cc}
 \Omega & v\\
0 & 0
\end{array}
\right)
&,~~M=
\left( 
\begin{array}{cc}
 A & t\\
0 & 0
\end{array}
\right)
\mbox{ and }
P=\left( I_n~0\right) \in\R^{n\times (n+1)} 
\nonumber\\
\frac{d}{d\tau}\Big|_{\tau=0}\Phi\circ\mu_M(\tau\Omega,\tau v)&=\int\limits_{\R^n}\nabla f(Ax+t)^{\top}\left( \Omega(Ax+t)+v\right) g(x)dx \label{two}\\
\frac{d^2}{d\tau^2}\Big|_{\tau=0}\Phi\circ\mu_M(\tau\Omega,\tau v)&=\int\limits_{\R^n}\left( \Omega(Ax+t)+v\right)^{\top}\mbox{H}_f(Ax+t) \left( \Omega(Ax+t)+v\right)g(x)dx\nonumber\\
 &+\int\limits_{\R^n}\nabla f(Ax+t)^{\top}\left( \Omega^2(Ax+t)+\Omega v\right) g(x)dx \label{three}
\end{align}
After substituting $z=Ax+t$ we get
\begin{align}
 \frac{d}{d\tau}\Big|_{\tau=0}\Phi\circ\mu_M(\tau\Omega,\tau v)&=\tr\left[ \int\limits_{\R^n}z\nabla f(z)^{\top}g_M(z)dz\Omega\right]+\left\langle \int\limits_{\R^n}\nabla f(z)g_M(z)dz,v\right\rangle _{\R^n}.\nonumber
\end{align}
Since the gradient ($\tilde{\Omega},\tilde{v})$ is the unique vector of the tangent space with 
\begin {align}
 \frac{d}{d\tau}\Big|_{\tau=0}\Phi\circ\mu_M(\tau\Omega,\tau v)=\tr(\Omega^{\top}{}\tilde{\Omega})+v^{\top}\tilde{v} \nonumber
\end{align}
we have proved (\ref{Gradient}).

To calculate the Hessian of $\Phi\circ\mu_M$ in zero, we again substitude $z=Ax+t$ in formula (\ref{three}) and get 
\begin{align}
 \frac{d^2}{d\tau^2}\Big|_{\tau=0}\Phi\circ\mu_M(\tau\Omega,\tau v)&=\int\limits_{\R^n}(\Omega z+v)^{\top}\mbox{H}_f(z)(\Omega z+v)g_M(z)dz\nonumber\\
&+\int\limits_{\R^n}\nabla f(z)^{\top}\Omega^2zg_M(z)dz
+\int\limits_{\R^n}\nabla f(z)^{\top}g_M(z)dz\Omega v.\nonumber
\end{align}
Since the last summand is equal to $\tilde{v}\Omega v$ it vanishes in a critical point. Therefore, we optain the Hessian $\mathcal{H}$ by polarizing the two first summands 
\begin{eqnarray}
 \mathcal{H}_{\Phi\circ \mu_M(0)}(\Omega,v)(\hat{\Omega},\hat{v})&=&\int\limits_{\R^n}(\Omega z+v)^{\top}\mbox{H}_f(z)(\hat{\Omega} z+\hat{v})g_M(z)dz\nonumber\\
&+&\frac{1}{2}\tr\left[
  \left(\int\limits_{\R^n}z\nabla f(z)^{\top}g(A^{-1}(z-t))dz\right)(\Omega\hat{\Omega}+\hat{\Omega}\Omega)\right] \nonumber
\end{eqnarray}
which proves (\ref{SimpleHessian}).
\end{proof}

Note, that (\ref{SimpleHessian}) yields the Hessian of $\Phi\circ\mu_M$ in $0$ only at a critical point $M\in G$. In the sequel, we will use the same formula at an arbitrary point $M\in G$ and thus obtain a modified Newton algorithm for $\Phi$. 
Thus, the modified Newton-step in (\ref{newton-G}) requires to solve the following system of linear equations:
\begin{align}
 \int\limits_{\R^n}\mbox{H}_f(z)g_M(z)(\Omega z+v)dz
=-\int\limits_{\R^n}g_M(z)\nabla f(z)dz\label{Newton-pura}\\
\intertext{and}
\pi_{\frak{k}} 
\left(\frac{1}{2}\Omega^{\top}\int\limits_{\R^n}\nabla f(z)z^{\top}g_M(z)dz +\frac{1}{2}\int\limits_{\R^n}\nabla f(z)z^{\top}g_M(z)dz\Omega^{\top}\right.\nonumber\\
+\left. \int\limits_{\R^n}\mbox{H}_f(z)\Omega zz^{\top}g_M(z)dz
+\int\limits_{\R^n}\mbox{H}_f(z)vz^{\top}g_M(z)dz)\right)\label{Newton-purb}\\
=-\int\limits_{\R^n}g_M(z)\pi_{\frak{k}}(\nabla f(z)z^{\top})dz\nonumber
\end{align}
with the unknowns $v\in\R$ and $\Omega\in\mathfrak{g}$.

In order to rewrite (\ref{Newton-pura}) and (\ref{Newton-purb}) in a linear equation in the components $v_i$ and $\Omega_{i,j},$ we first focus on the Euclidean transformation group. Here, with $\Omega \in so(n)$ we obtain:
\begin{align}
 & \int\limits_{\R^n}\mbox{H}_f(z)g_M(z)(\Omega z+v)dz
=-\int\limits_{\R^n}g_M(z)\nabla f(z)dz\label{nra}\\
\intertext{and}
 & \frac{1}{2}\int\limits_{\R^n}\left(-\Omega \nabla f(z)z^{\top} - z\nabla f(z)^{\top}\Omega\right) g_M(z)dz \nonumber\\
&\frac{1}{2}\int\limits_{\R^n}\left( -\nabla f(z)z^{\top}\Omega -\Omega z\nabla f(z)^{\top}\right) g_M(z)dz  \nonumber\\
&+\int\limits_{\R^n}\left( \mbox{H}_f(z)vz^{\top} -zv^{\top}\mbox{H}_f(z)\right) g_M(z)dz\label{nrb}\\
&+\int\limits_{\R^n}\left( \mbox{H}_f(z)\Omega zz^{\top}+ zz^{\top}\Omega\mbox{H}_f(z)\right) g_M(z)dz\nonumber\\
&~~~~~~~~~~~~~~~~=-\int\limits_{\R^n}g_M(z)(\nabla f(z)z^{\top}-z\nabla f(z)^{\top})dz.\nonumber
\end{align}

To evaluate the components of this system of linear equations, we use the abbreviations:
\begin{align}
&\alpha_i=\int\limits_{\R^n}g(x)\frac{\partial f}{\partial x_i}(Ax+t)dx, ~~~~~~~~~~~~~~~~~~~~~~~\beta_{i,j}=\int\limits_{\R^n}g(x) (Ax+t)_j\frac{\partial f}{\partial x_i}(Ax+t)dx, \nonumber\\
& \gamma_{i,j,k}=\int\limits_{\R^n}g(x)(Ax+t)_i\frac{\partial^2
  f}{\partial x_j\partial x_k}(Ax+t)dx,~~~\epsilon_{i,j}=\int\limits_{\R^n}g(x)\frac{\partial^2
  f}{\partial x_i\partial x_j}(Ax+t)dx,\label{coefs} \\
& \delta_{i,j,k,l}=\int\limits_{\R^n}g(x) (Ax+t)_i(Ax+t)_j\frac{\partial^2
  f}{\partial x_k\partial x_l}(Ax+t)dx\nonumber
\end{align}
We obtain:

\begin{lemma}\label{th:QMC_SE(N)}
Let $(\Omega,v)\in so(n)\times \R^n$ be the modified Newton-direction for the objective function (\ref{Zielfunktion}) in a certain point $M\in SE(n)$. Then the components $\Omega_{k,l},$ $1\leqslant k,l\leqslant n$ of $\Omega$ and $v_k$ $1\leqslant k\leqslant n$ of $v$ satisfy
 \begin{align}
 \sum\limits_{k>l}(\gamma_{l,k,i}-\gamma_{k,l,i})\Omega_{k,l}+\sum\limits_k\epsilon_{i,k}v_k=-\alpha_i\label{begin_newton_SE(n)}
\end{align}
for $1\leqslant i\leqslant n$ and
\begin{align}
 \frac{1}{2}\sum\limits_{k>j}(\beta_{i,k}+\beta_{k,i})\Omega_{k,j}-\frac{1}{2}\sum\limits_{k<j}(\beta_{i,k}+\beta_{k,i})\Omega_{j,k}
-\frac{1}{2}\sum\limits_{k>i}(\beta_{j,k}+\beta_{k,j})\Omega_{k,i}\nonumber\\
+\frac{1}{2}\sum\limits_{k<i}(\beta_{j,k}+\beta_{k,j})\Omega_{i,k}-\sum\limits_{k>l}(\delta_{i,k,l,j}-\delta_{j,l,k,i}+\delta_{i,l,k,j}-\delta_{i,k,l,j})\Omega_{k,l}\label{end_newton_SE(n)}\\
-\sum\limits_{k}(\gamma_{j,k,i}-\gamma_{i,k,j})v_k
=\beta_{i,j}-\beta_{j,i}\nonumber
\end{align}
for $1\leqslant i<j\leqslant n$.
\end{lemma}

Note that the unknowns of this system are $v_i$ and $\Omega_{i,j}$ for $i>j$. Therefore, a unique solution of the linear system corresponds to an unique element of the $so(n)$.

Let us return to the case of volume-preserving transformations $G=SA(n)$. The modified Newton-equation (\ref{Newton-pura}) and (\ref{Newton-purb}) has now the form:
\begin{align}
& \int\limits_{\R^n}\mbox{Hess}_f(z)g_M(z)(\Omega z+v)dz
=-\int\limits_{\R^n}g_M(z)\nabla f(z)dz\label{nbc}\\
\intertext{and}
& \frac{1}{2}\Omega^{\top}\int\limits_{\R^n}\nabla f(z)z^{\top}g_M(z)dz +\frac{1}{2}\int\limits_{\R^n}\nabla f(z)z^{\top}g_M(z)dz\Omega^{\top}\nonumber \\
+ &\int\limits_{\R^n}\mbox{Hess}_f(z)\Omega zz^{\top}g_M(z)dz
+\int\limits_{\R^n}\mbox{Hess}_f(z)vz^{\top}g_M(z)dz\nonumber\\
- &\frac{1}{n}I\int\limits_{\R^n}g_M(z)\left( z^{\top}\Omega^{\top}\nabla f(z)+z^{\top}\Omega\mbox{Hess}_f(z)z+z^{\top}\mbox{Hess}_f(z)v\right)dz\label{nbd}\\
 & ~~~~~~~=-\int\limits_{\R^n}g_M(z)\nabla f(z)z^{\top}dz+\frac{1}{n}I\int\limits_{\R^n}g_M(z)z^{\top}\nabla f(z)dz.\nonumber
\end{align}
Again, we can calculate the components of this system using the coefficients in (\ref{coefs}). We end up with an analog version of lemma \ref{th:QMC_SE(N)} in the case of a volume preserving transformations.

\begin{lemma}
Let $(\Omega,v)\in sl(n)\times \R^n$ be the modified Newton-direction for the objective function (\ref{Zielfunktion}) in a certain point $M\in SA(n)$. Then the components $\Omega_{k,l},$ $ 1\leqslant k,l\leqslant n,~(k,l)\neq (n,n)$ of $\Omega$ and $v_k$ of $v$ satisfy for each $1\leqslant i\leqslant n$ 
\begin{align}
 \sum\limits_{k\neq l}\gamma_{l,k,i}\Omega_{k,l}+\sum\limits_{k\neq n}(\gamma_{k,k,i}-\gamma_{n,n,i})\Omega_{k,k}+\sum\limits_{k}\epsilon_{i,k}v_k=-\alpha_i\label{f0}
\end{align}
 and for all $1\leqslant i,j\leqslant n$, $(i,j)\neq(n,n)$ the following equations: 
\begin{align}\label{f1}\begin{split}
 \frac{1}{2}\sum\limits_k\beta_{i,k}\Omega_{j,k}+\frac{1}{2}\sum\limits_{k}\beta_{k,j}\Omega_{k,i}
+\sum\limits_{(k,l)\neq (n,n)}\delta_{j,l,k,i}\Omega_{k,l} \\
-\delta_{j,n,n,i}\sum\limits_{k\neq n}\Omega_{k,k} 
+\sum\limits_k\gamma_{j,k,i}v_k
=-\beta_{i,j}\end{split}
\end{align}
for $i\neq n,~j\neq n,~i\neq j$,
\begin{align}\begin{split}
 \frac{1}{2}\sum\limits_k\beta_{n,k}\Omega_{j,k}+\frac{1}{2}\sum\limits_{k\neq n}\beta_{k,j}\Omega_{k,n}
+\sum\limits_{(k,l)\neq (n,n)}\delta_{j,l,k,n}\Omega_{k,l} \\
-(\delta_{j,n,n,i}+\frac{1}{2}\beta_{n,j})\sum\limits_{k\neq n}\Omega_{k,k} 
+\sum\limits_k\gamma_{j,k,n}v_k
=-\beta_{n,j}\end{split}
\end{align}
for $i=n,~j\neq n$,
\begin{align}\begin{split}
 \frac{1}{2}\sum\limits_{k\neq n}\beta_{i,k}\Omega_{n,k}+\frac{1}{2}\sum\limits_{k}\beta_{k,n}\Omega_{k,i}
+\sum\limits_{(k,l)\neq (n,n)}\delta_{n,l,k,i}\Omega_{k,l} \\
-(\delta_{n,n,n,i}+\frac{1}{2}\beta_{i,n})\sum\limits_{k\neq n}\Omega_{k,k} 
+\sum\limits_k\gamma_{n,k,i}v_k
=-\beta_{i,n}\end{split}
\end{align}
for $j=n,~i\neq n$ and
\begin{align}
 \frac{1}{2}\sum\limits_k\beta_{i,k}\Omega_{j,k}+\frac{1}{2}\sum\limits_{k}\beta_{k,j}\Omega_{k,i}
+\!\!\sum\limits_{(k,l)\neq (n,n)}\!\!\left( \delta_{j,l,k,i}-\frac{1}{n}\left( 
\beta_{k,l}+\sum\limits_{m}\delta_{l,m,k,m}
\right) \right) \Omega_{k,l}\nonumber \\
-\left( \delta_{j,n,n,i}
-\frac{1}{n}\left( 
\beta_{n,n}+\sum\limits_{m}\delta_{n,m,n,m}
\right)
\right)\sum\limits_{k\neq n} \Omega_{k,k} 
+\sum\limits_k\left( \gamma_{j,k,i}
-\frac{1}{n}\sum\limits_l \gamma_{l,l,k}
\right) v_k\nonumber\\
=-\beta_{i,j}+\frac{1}{n}\sum\limits_{k}\beta_{k,k}\label{f6}
\end{align}
for $i=j,~i,j\neq n$.
\end{lemma}

We want to point out that the presented Newton step uses an approximated version of the Hessian. However, at each critical point of the cost function we have an exact evaluation of the Hessian. Approximations of the Hessian appear also in a couple of classical algorithms like the Gauss-Newton method, the Levenberg Marquard algorithm (see e.g. \cite{Absil} for an overview) or the optimization technique presented in \cite{Malis}. In contrast to our method, such algorithms do not perform a Newton step at a critical point and are not necessarily local quadratic convergent.

%  Hence, we will show in the next subsection that the convergence behavior of a classical Newton algorithm is preserved.

%%%%%%%%%%%%%%%%%%%%%%%%%%%%%%%%%%%%%%%%%%%%%%%%%%%%%%%%%%%%%%%%%%%%%%%%%%%%%%%%%%%%%%%%%%%%%%

\subsection{The Quasi-Monte-Carlo Newton Algorithm}
% The linear equation to derived above, define the approximate-Newton direction, using the local parameterization defined in (\ref{Karte-SE(n)}) and (\ref{Karte-SA(n)}). 
% 
% \noindent
% 
% Instead of writing down these equation in full detail, we note that the most time-consuming part of the algorithms is the calculation of the integrals appearing in the coefficients $\alpha,\ldots,\epsilon$ in (\ref{coefs}). 
Even though the previous subsection defines itself an iterative algorithm, a few specifications have to be made in order to apply it to the image registration task.\\
\indent
First, we note that the most time-consuming part of the algorithms is the calculation of the integrals appearing in the coefficients $\alpha,\ldots,\epsilon$ in (\ref{coefs}) -- if they can be calculated at all.
One way out is to approximate the integrals via Quasi Monte-Carlo methods. (See for example \cite{Niederreiter} for an introduction.) Thus, we replace in (\ref{coefs}) the integrals by the average of sampled function-values via

\begin{eqnarray}
 \int\limits_{Q\subset\R^n}f(x)dx & \approx & \frac{1}{N}\sum\limits_{i=1}^{N}f(x_i).\label{QM}
\end{eqnarray}

Explicitly, we use

\begin{align}
&\alpha_i\!=\!\frac{1}{N}\sum\limits_{r=1}^{N}g(x_r)\frac{\partial f}{\partial
  x_i}(Ax_r+t), 
& &\!\!\!\!\!\!\!\!\!\!\!\!\!\!\!\!\!\!\!\!\!\!\!\!\!\!\!\!\!\!\!\!\!\!\!\!\!\!\!\!\!\!\!\!\!\!\!\!\! \beta_{i,j}\!=\!\frac{1}{N}\sum\limits_{r=1}^N g(x_r) \frac{\partial f}{\partial x_i}(Ax_r+t)(Ax_r+t)_j, \nonumber\\
& \gamma_{i,j,k}\!=\!\frac{1}{N}\sum\limits_{r=1}^N g(x_r)(Ax_r+t)_i\frac{\partial^2
  f}{\partial x_j\partial x_k}(Ax_r+t), 
& &\!\!\!\!\!\!\!\!\!\!\!\!\!\!\!\!\!\!\!\!\!\!\!\!\!\!\!\!\!\! \epsilon_{i,j}\!=\!\frac{1}{N}\sum\limits_{r=1}^Ng(x_r)\frac{\partial^2
  f}{\partial x_i\partial x_j}(Ax_r+t),\label{disccoefs}\\[-0.4cm]
& \delta_{i,j,k,l}\!=\!\frac{1}{N}\sum\limits_{r=1}^Ng(x_r) (Ax_r+t)_i(Ax_r+t)_j\frac{\partial^2
  f}{\partial x_k\partial x_l}(Ax_r+t).
 & &\nonumber
\end{align}
The final registration algorithms for $G=SE(n)$ and $G=SA(n)$ are summarized in the Tables 3.1 and 3.2.

\begin{figure}[t]
\begin{minipage}[m]{0.3\textwidth}
\fbox{
\parbox{12.5cm}{
\textbf{Table 3.1: QMC-Newton Registration-Algorithm on SE(n)}\\[0.5cm]
\small
Step 1.\\
Pick an initial guess $M_0\in SE(n)$ and set $m=0$.\\[0.5cm]
Step 2.\\
Calculate $\alpha_i,~\beta_{i,j},~\gamma_{i,j,k},~\delta_{i,j,k,l}$ and $ \epsilon_{i,j}$ for all $1\leqslant i,j,k,l\leqslant n$ as defined in equation~(\ref{disccoefs}).\\[0.5cm]
Step 3.\\
Solve the linear system consisting of the equations
\begin{align}
 \sum\limits_{k>l}(\gamma_{l,k,i}-\gamma_{k,l,i})\Omega_{k,l}+\sum\limits_k\epsilon_{i,k}v_k=-\alpha_i~~~\mbox{for all}~1\leqslant i\leqslant n\nonumber
\intertext{and}
 \sum\limits_{k>j}\frac{1}{2}(\beta_{i,k}+\beta_{k,i})\Omega_{k,j}-\sum\limits_{k<j}\frac{1}{2}(\beta_{i,k}+\beta_{k,i})\Omega_{j,k}
-\sum\limits_{k>i}\frac{1}{2}(\beta_{j,k}+\beta_{k,j})\Omega_{k,i}\nonumber\\
+\sum\limits_{k<i}\frac{1}{2}(\beta_{j,k}+\beta_{k,j})\Omega_{i,k}-\sum\limits_{k>l}(\delta_{i,k,l,j}-\delta_{j,l,k,i}+\delta_{i,l,k,j}-\delta_{i,k,l,j})\Omega_{k,l}\nonumber\\
-\sum\limits_{k}(\gamma_{j,k,i}-\gamma_{i,k,j})v_k=\beta_{i,j}-\beta_{j,i}~~~\mbox{for all}~1\leqslant i<j\leqslant n\nonumber
\end{align}

with the unknowns $v_i$ and $\Omega_{i,j}$, $j>i$.\\[0.5cm]

Step 4.\\
Construct the $n\times n$ matrix $\Omega$ with the entries $\Omega_{i,j}$. In the case of $j>i$ use the solution of step 3. Else, set
\begin{eqnarray}
      \Omega_{i,j}=\left\lbrace 
\begin{array}{cc}
 -\Omega_{j,i} & \mbox{for~} j<i\\
 0 & \mbox{for~} j=i.
\end{array}
\right. \nonumber
    \end{eqnarray}
Compute 
\begin{eqnarray}
 M_{m+1}:=\nu_{M_m}^{QR}(\Omega,v),\nonumber
\end{eqnarray}
where $\nu^{QR}$ is defined in (\ref{Karte-SE(n)}).\\[0.5cm]

Step 5.\\
Set $m=m+1$ and goto Step 2.\\
~
}}

\end{minipage}
\end{figure}

In contrast to a classical Monte-Carlo method, the sampling-points $x_i\in Q$ are chosen to be uniform distributed in $Q$. The so-called Halton sequence~\cite{Halton} is a good example of such uniformly distributed sampling points. It is shown in the literature, that for such well chosen sampling points, the approximation error of (\ref{QM}) is bounded by $O((\log N)^n/N)$. (In contrast, the error of the Monte Carlo method tends to zero with the order $O(1/\sqrt{N})$.) The final algorithms are presented in Table 3.1 on $SE(n)$ and in Table 3.2 on $SA(n)$ respectively.

Before stating our main convergence result, a few remarks are in order.
We use standard Gauss elimination to solve the linear system in step 3 of each algorithm. If this system is not solvable, a standard approach in optimization theory is to search for the least squares solution. Moreover, as is the case for all Newton methods, convergence of the algorithm is not guaranteed for an arbitrary initial conditions $M_0\in G$. Even if the algorithm converges, the limiting point need not be a local maximum. To overcome this, one can adapt a Gauss-Newton step, or first test if the Newton-direction $(\Omega,v)$ is an ascent-direction. Alternatively, we can take the gradient of the objective function instead. Furthermore, one can make a line-search in the ascent-direction, e.g. by using the Amijo-rule. We skip the straightforward details.\indent

\begin{theorem}
Suppose $g\in C(\R^n,\R)$ and $f\in C^3(\R^n,\R)$. Then the QMC-Newton algorithms described in Table 3.1 and Table 3.2 being applied to the two following cost functions 
\begin{itemize}
 \item [(a)] $\Phi:G\to\R$ in (\ref{Zielfunktion}) with coefficients (\ref{coefs}) and for the 
\item [(b)] $\Psi:G\to\R$ definied by
\begin{equation}
\Psi(A,t):= \frac{1}{N}\sum_{r=1}^{N}f(Ax_r+t)g(x_r).\label{appsum}
\end{equation}
 with coefficients (\ref{disccoefs}),
\end{itemize}
are locally quadratically convergent around each nondegenerate critical point. 
\end{theorem}

\begin{proof} 
% Since $g\in C(\R^n,\R)$ and $f\in C^3(\R^n,\R)$, the cost functions $\Phi$ and $\Psi$ are $C^3$.
% Let  $M^{\ast}$ be a nondegenerate critical point 
% %of the smooth functions 
% $\Phi,\Psi$, respectively,  and let $\mu_M,\nu_M$ denote a family of smooth local parameterizations of $G$ around $M^{\ast}$, satisfying
% \begin{align}
%  D\mu_{M^{\ast}}(0)=D\nu_{M^{\ast}}(0).\label{cond_mu_nu}
% \end{align}
% From \cite{640-3}, \cite{grassmann} we conclude that there exists an open neighborhood $V\subset G$ of $M^{\ast}$ such that the point sequence $\left\lbrace M_k\right\rbrace_{k\in\N_0}$ generated by (\ref{newton-G}) converges locally quadratically fast to $M^{\ast}$, which proves (a).
To begin with, let us first focus on the cost function $\Phi$. We observe that the algorithms in Table 3.1 describes a $(\mu,\nu)-$Newton algorithm on $G=SE(n)$, where $\mu$ and $\nu$ are defined in (\ref{Exp_SEn}) and (\ref{Karte-SE(n)}). Respectively, Table 3.2 applies a $(\mu,\nu)-$Newton algorithm on $G=SA(n)$ where $\mu$ and $\nu$ are defined in (\ref{Exp_SAn}) and (\ref{Karte-SA(n)}). These parametrizations satisfy
\begin{align}
 D\mu_{M}(0)=D\nu_{M}(0),\label{cond_mu_nu}
\end{align}
for all $M\in G$. Moreover, the cost function $\Phi$ is in $C^3(G,\R)$ since $g\in C(\R^n,\R)$ and $f\in C^3(\R^n,\R)$. Thus, we can apply the local quadratic convergence theorem from \cite{640-3} or \cite{grassmann}: There exists an open neighborhood $V\subset G$ of each critical point $M^{\ast}\in G$ such that the point sequence $\left\lbrace M_k\right\rbrace_{k\in\N_0}$ generated by the algorithms in Table 3.1 or Table 3.2, converges quadratically to $M^{\ast}$, if $M_0\in V$. This proves (a).

To prove (b), we use the fact that the process of differentiation with respect to $M$ and the process of sampling in (\ref{QM}) commute. Hence, we obtain the gradient of $\Psi$ if we apply the approximation (\ref{QM}) to the right hand side of (\ref{nrb})-(\ref{nra}) for $G=SE(n)$ and of (\ref{nbc})-(\ref{nbd}) for $G=SA(n)$ respectively. In the same manner, we get the Hessian in a critical point by approximating the left side of these equations. Therefore, using (\ref{disccoefs}) in (\ref{begin_newton_SE(n)})-(\ref{end_newton_SE(n)}) and (\ref{f0})-(\ref{f6}), respectively, yield again a $(\mu,\nu)-$Newton step and the local quadratic convergence is obtained by the same argument as for the function~$\Phi$.
\end{proof}

A few remarks are in order: Firstly, the assumption that $f$ is three times continious differentiable is only needed to ensure the local quadratic convergence behavior. The QMC-Newton algorithms need only evaluations of the first and second derivative of $f$. Moreover, since the raw image data are usually given in a discreticed form, a previous interpolation step is needed to extend the function continously. If the interpolation order is high enough, then $f\in C^3(\R^n,\R)$ can always be guaranteed (see section 4.1).

In \cite{640-3}, sufficient conditions for local quadratic convergence are: A $C^3$ cost function, condition (\ref{cond_mu_nu}) and the non-degenericity of the critical points. We will show in the appendix that the last condition is fulfilled for a generic choice of the images.

Note, that the cost function (\ref{appsum}) is only an approximation of (\ref{Zielfunktion}); a more sensible choice would be the discretization of the least squares as
\begin{equation}  
\sum_{r=1}^{N}(f(Ax_r+t)-g(x_r))^2.\label{ssd-diskret}
\end{equation}
The above local quadratic convergence result would hold as well for suitably adapted choices of coefficients.

\begin{figure}[t]
\begin{minipage}[m]{0.3\textwidth}
\fbox{
\parbox{12.5cm}{
\textbf{Table 3.2: QMC-Newton Registration-Algorithm on SA(n)}\\[0.5cm]
\small
Step 1.\\
Pick an initial guess $M_0\in SA(n)$ and set $m=0$.\\[0.5cm]
Step 2.\\
Calculate $\alpha_i,~\beta_{i,j},~\gamma_{i,j,k},~\delta_{i,j,k,l}$ and $ \epsilon_{i,j}$ for all $1\leqslant i,j,k,l\leqslant n$ as defined in equation~(\ref{disccoefs}).\\[0.5cm]
Step 3.\\
Solve the linear system described in (\ref{f0}) - (\ref{f6}) with the unknowns $v_i$ and $\Omega_{i,j}$, $(i,j)\neq (n,n)$.\\[0.5cm]
Step 4.\\
Construct the $n\times n$ matrix $\Omega$ with the entries $\Omega_{i,j}$. In the case of $(j,i)\neq (n,n)$ use the solution of step 3. Else, set
\begin{eqnarray}
\Omega_{n,n}=-\sum\limits_{k \neq n} \Omega_{k,k}\nonumber
\end{eqnarray}
and compute 
\begin{eqnarray}
 M_{m+1}:=\nu_{M_m}^{QR}(\Omega,v),\nonumber
\end{eqnarray}
where $\nu^{QR}$ is defined in (\ref{Karte-SA(n)}).\\[0.5cm]
Step 5.\\
Set $k=k+1$ and goto Step 2.
}}\\[0.5cm]
\end{minipage}
\end{figure}

The local quadratic convergence of the algorithms will be numerically confirmed in section 5. However,  this property  has more advantages the bigger the region in which the algorithms converge quadratically is.
In Fig.~\ref{glaetten} it is demonstrated that this region can be very small if $f$ and $g$ are interpolations of raw medical images. (The region of quadratic convergence is a subset of the region in which the cost function is convex.) A way out is to smooth the image up to a certain level and to registrate the smoothed image data instead. The result can then be taken as the initial guess for the algorithms applied to  less smoothed images etc. Throughout this paper, we perform the smoothing by projecting the images to a finite dimensional function-space, namely a spline-basis (see section 4 for details). Fig.~\ref{glaetten} shows the change of the objective function by varying the degree of smoothness. In all the examples we considered, the smoothing of the images yields an increase of the domain of quadratic convergence.

%  So far the described algorithms act on sufficient smooth functions and not on the raw data of two given images. A common approach in image processing is to interpolate the pixels to a smooth function (cf.~\cite{3DVision} p. 99). However , if we proceed like this, it turns out that, for most of the relevant medical images, the region in which the algorithms converge quadratically will be smaller than the distance of two pixels, which incapacitates the algorithms. A loophole smoothes the given image data up to a certain level  and registrates the smoothed images  instead. The result can then be taken as the initial guess for the algorithms applied to  less smoothed images etc. 
% %\textcolor{red}{Throughout this paper, we smooth the images by projecting the data to a finite-dimensional spline function space. } 
% Fig.~\ref{glaetten} shows the change of the objective function by varying the degree of smoothness. The region around the maximum in which the second derivative is strictly negative grows by increasing the smoothing degree. In this region a generalized Newton would apply the Newton-direction instead of the gradient. In all examples presented here, we perform the smoothing by projecting the images to a finite dimensional function-space, namely a spline-basis (see section 4 for details). 
% % The smoothing factor $\lambda$ equates the reciprocal value of the number of used basis-elements.

\begin{figure}[t]
\begin{minipage}[l]{0.3\textwidth}
\centering
\includegraphics[height=3.5cm]{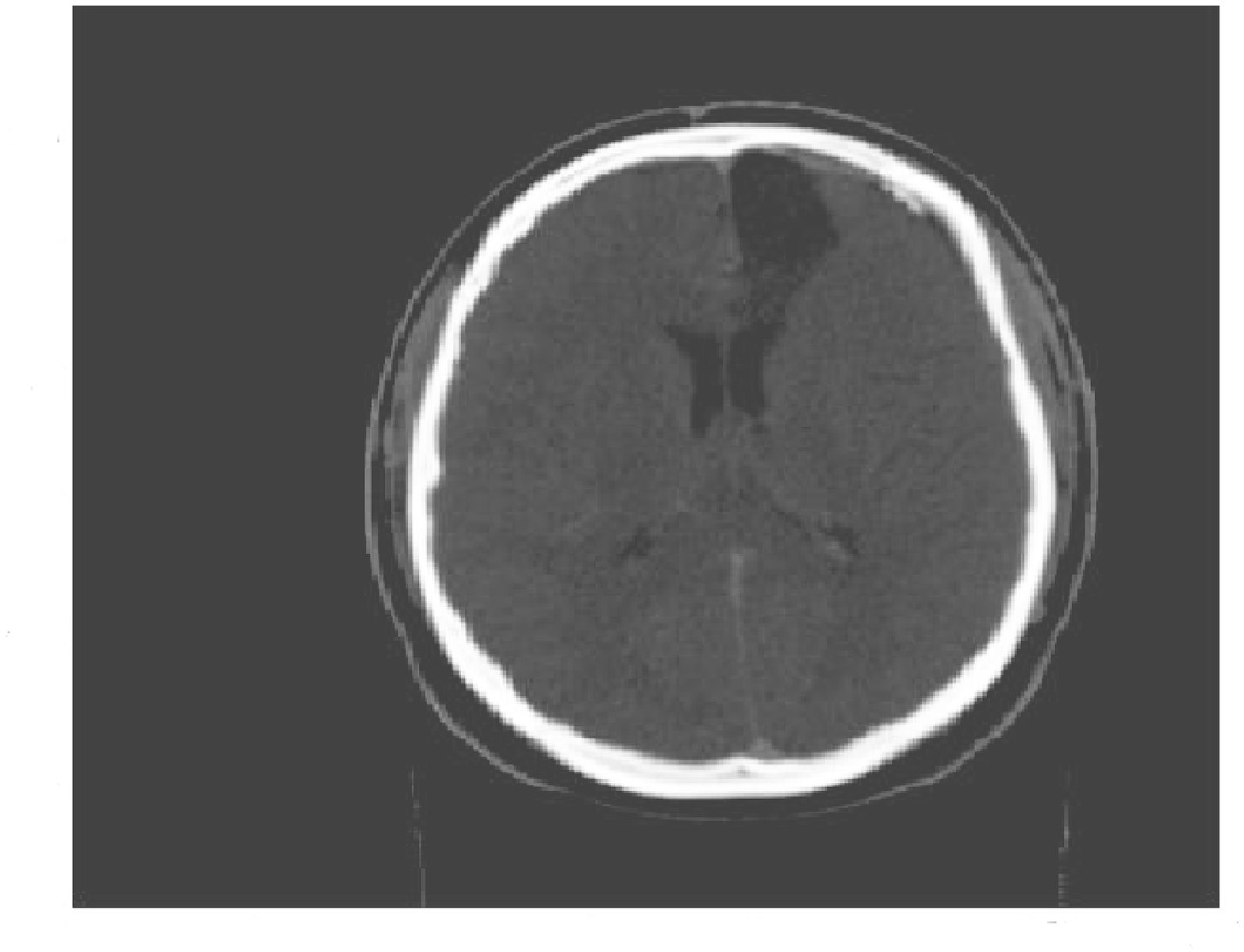}
\end{minipage}
\begin{minipage}[m]{0.3\textwidth}
\centering
~~\includegraphics[height=3.5cm]{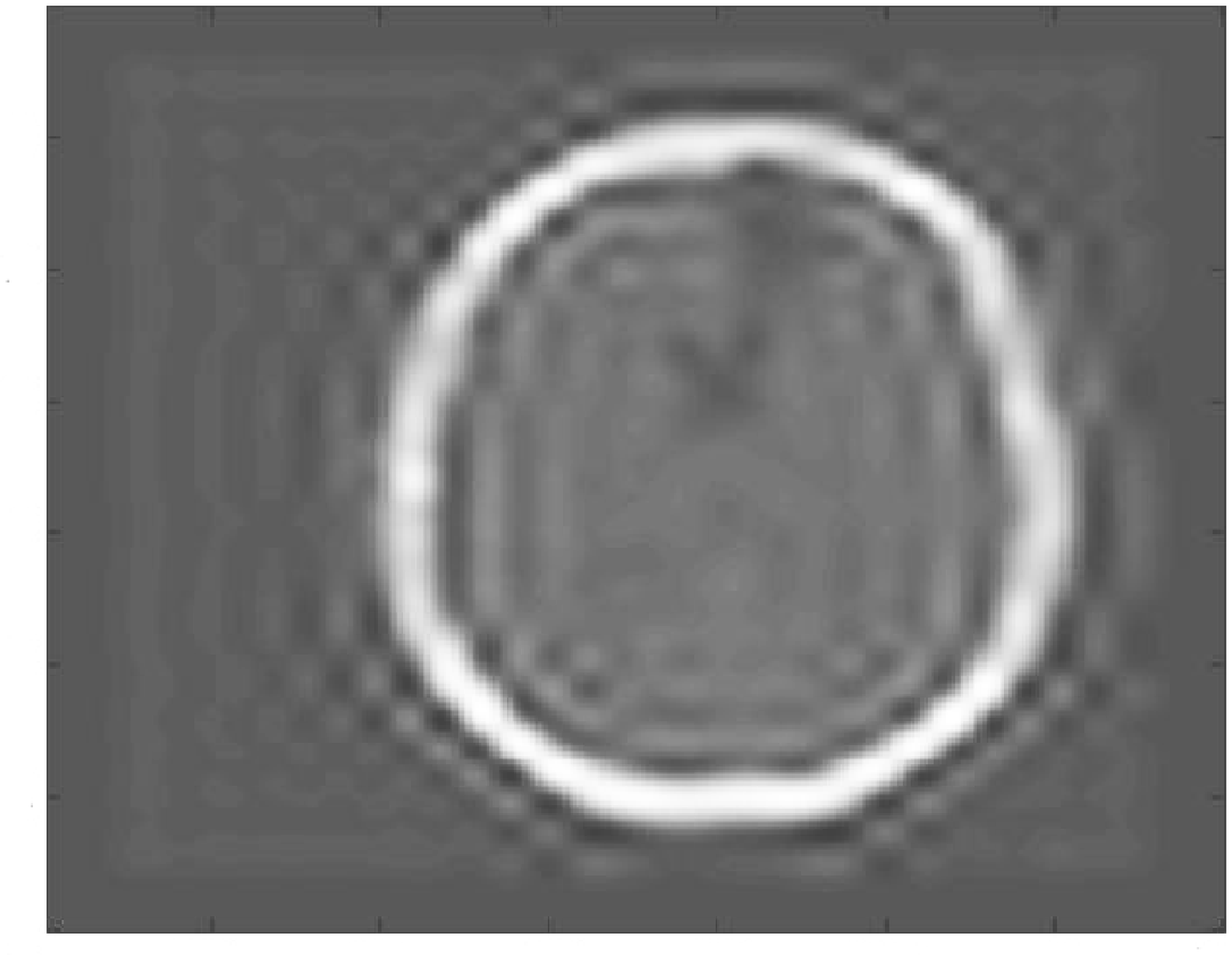}
\end{minipage}
\begin{minipage}[r]{0.3\textwidth}
\centering
~~~~~~\includegraphics[width=4.5cm,height=3.5cm]{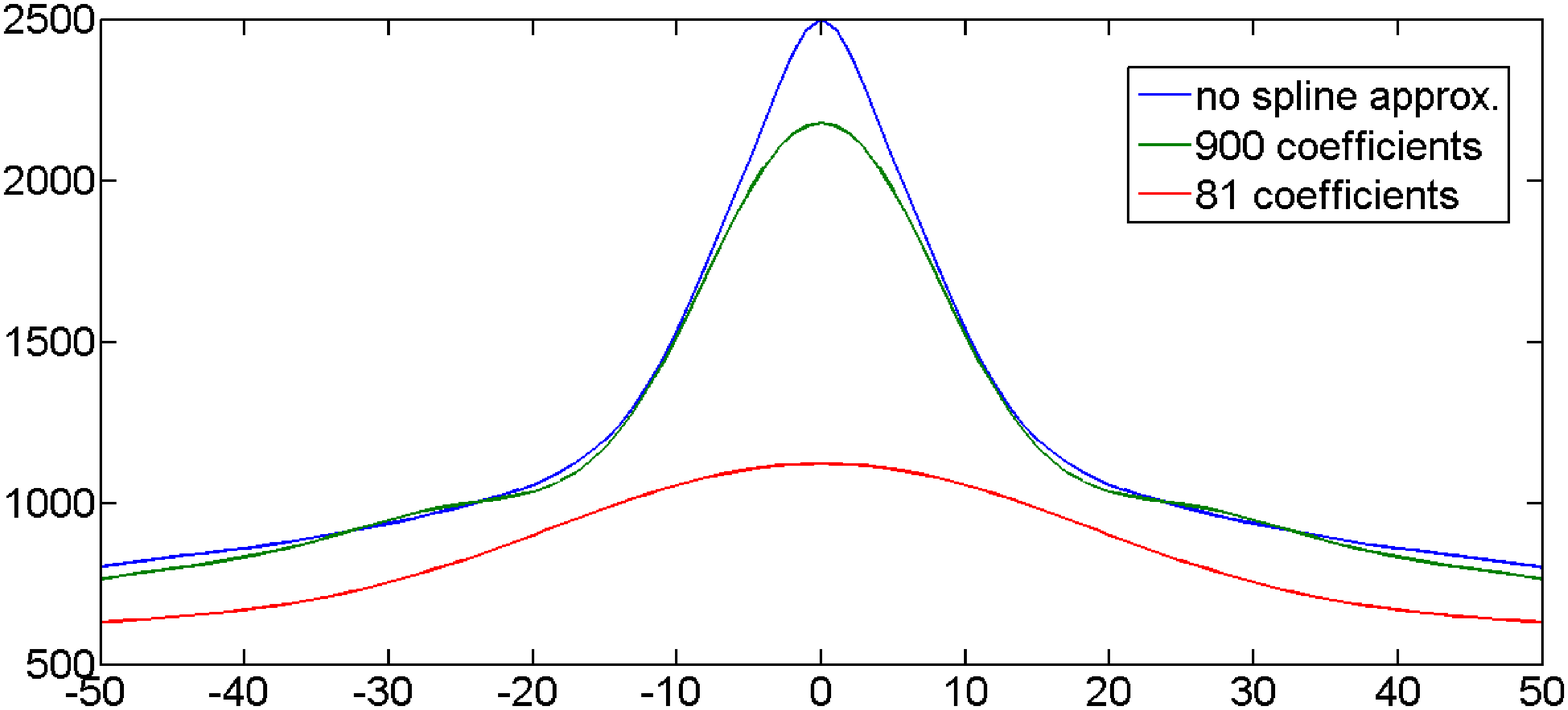}
\end{minipage}
\caption{\footnotesize The first image shows a 2D-slide of a CT data-set of the head with $350\times 350$ pixel. The second image is a spline-approximation with 900 basis-functions. The appearing artefacts may be caused by the Gibbs phenomena. The blue graph of the third image describes the cost function (\ref{Zielfunktion}) by translating the image on the left against itself. For the green and red graph we replace the image by its spline-approximation with 900 and respectively 81 basis-functions.} \label{glaetten}
\end{figure}

%%%%%%%%%%%%%%%%%%%%%%%%%%%%%%%%%%%%%%%%%%%%%%%%%%%%%%%%%%%%%%%%%%%%%%%%%

\section{Registration on Spline Function Spaces}
If we consider an image as a function on the space of pixels, smoothing can be regarded as a projection to a suitable space of smooth functions, such as e.g. spaces of splines. We will show that this interpretation, compared to standard smoothing methods like Gaussian filters, has the advantage that we can exploit the reduction of information in the previous discussed Newton method while preserving a high degree of accuracy. A difficulty with this approach though is that the space of splines is usually not invariant under rotations, as happens for tensor product splines. In this section we work out the details in this setting and indicate a way how to avoid such a difficulty.

%%%%%%%%%%%%%%%%%%%%%%%%%%%%%%%%%%%%%%%%%%%%%%%%%%%%%%%%%%%%%%%%%%%%%%%%%%

\subsection{Spline-Approximation of the Images}
In this subsection we will describe the projection of a row image data to a spline function space. A medical image $f$ can simply be seen as a mapping of vertices of a grid to gray values. 
\begin{eqnarray}
 f:\Z^n\rightarrow\R.\nonumber
\end{eqnarray}
Of course, in all applications, $f$ will only be defined on a finite set, but without any restriction, we can continue it by setting all other values equal zero. Furthermore, the Nyquist Sampling Theorem, (see for example \cite{Digital},) allows us to extend $f$ to a unique band-limited function on $\R^n$
\begin{eqnarray}
 f:\R^n\rightarrow\R.\nonumber
\end{eqnarray}
Of course, using this function in image processing exercises leeds to enormous numerical costs. Therefore we use a projection of $f$ into a finite dimensional function space~$\mathcal{S}$. Defining this space, we employ the so called B-splines
\begin{eqnarray}
 B^0(x):=\left\lbrace 
\begin{array}{cc}
 1 & \mbox{if}~-\frac{1}{2}\leqslant x\leqslant \frac{1}{2}\\
0 & \mbox{else}
\end{array},~~~~B^k(x):=\int\limits_{\R}B^{k-1}(s)B^0(s-x)ds,\nonumber
\right. \nonumber
\end{eqnarray}
extended in $n$-dimension by their tensor-products
\begin{eqnarray}
 B^k(x_1,\dots,x_n):=B^k(x_1)\cdot \ldots \cdot B^k(x_n).\label{spline}
\end{eqnarray}
Fig. \ref{Splines} shows the graphs of the splines in first and second order in one dimension. The most important fact of a B-splines for us is that it becomes smoother when the order is growing. The corresponding spline function space can now be defined as the set generated by all integer-translations of (\ref{spline}) (cf.\cite{splines})
\begin{eqnarray}
 \mathcal{S}_{\lambda}^k:=\left\lbrace \sum\limits_{\lambda\cdot r \in\Z^n}c_rB^k\left(\frac{x_1}{\lambda}-r_1,\dots,\frac{x_n}{\lambda}-r_n \right)~\Big|~c\in l_2  \right\rbrace,~~~~~~\lambda\in\N.\nonumber 
\end{eqnarray}
The condition $c\in l_2$ ensures that the $L_2$-Norm of all functions in $\mathcal{S}_{\lambda}^k$ exist. Note that we will only work with images that have a bounded support, which is the reason why we only consider finitely many linear combination of tensor product functions $B^k$. Therefore, the condition is always fullfilled. The additional parameter $\lambda\in\N$ is used in image processing to neglect high frequency informations of the images (cf.~\cite{640-4}). That is, if $\lambda$ is increasing, the projection $f_{R_{\lambda}}\in\mathcal{S}_{\lambda}^k$  of an image $R$ becomes smoother (cf. Fig. \ref{glaetten}).

%%%%%%%%%%%%%%%%%%%%%%%%%%%%%%%%%%%%%%%%%%%%%%%%%%%%%%%%%%%%%%

\subsection{Registration of Spline Coefficients}
We begin with a reformulation of the registration problem. For two given images $f\in \mathcal{S}_{\lambda}^2$ and $g\in \mathcal{S}_{\lambda}^1$ with coefficients $c^f,c^g\in l_2$ the objective function (\ref{Zielfunktion}) becomes
\begin{align}
 \Phi(A,t)=\sum\limits_{r,s\in \frac{1}{\lambda}\Z^n}c_s^fc_r^g \int\limits_{\R^n}B_{s,\lambda}^2(Ax+t)B_{r,\lambda}^1(x)dx,\nonumber
\end{align}
with $B_{r,\lambda}^k(x):=B^k\left(\frac{x}{\lambda}-\lambda r \right)$. Since $B_{r,\lambda}^k(x)$ is a translation of $B_{0,\lambda}^k$, we obtain
\begin{align}
 \Phi(M)&=\sum\limits_{s,r\in \frac{1}{\lambda}\Z^n}c_s^fc_r^g\int\limits_{\R^n}B_{0,\lambda}^1(x)B_{0,\lambda}^2(A(x+s)+t-r)dx\nonumber\\
&\approx \sum\limits_{s,r\in \frac{1}{\lambda}\Z^n}c_s^fc_r^g\int\limits_{\R^n}B_{0,\lambda}^1(x)B_{0,\lambda}^2(x-A^{-1}(r-t)+s)dx.\label{approx}
\end{align}

The approximation in the last line needs some explanations: Of course, the tensor products (\ref{spline}) are not invariant under rotation. Assuming they are, we obtain approximations, by rotating the argument of the second spline around the barycentre $t-r$ in such a way that the argument becomes a simple translation in the direction $A^{-1}(r-t)-s$. In the cases where $A$ is close to the identity (which is true for most of the medical image problems) the approximation error tends to zero. Our examples show that this simplification does hardly influence the solution of the  optimization problem.

Therefore, the convolution of two splines 
\begin{align}\begin{split}
F(s)&:=\int\limits_{\R^n}  B_{0,\lambda}^1(x)B_{0,\lambda}^2(x-s)dx, \\
&= B_{0,\lambda}^4(s_1)\cdot\ldots\cdot B^4_{0,\lambda}(s_n)\end{split}\label{delta}
\end{align}
is related to the optimization problem. Actually, $F(s)$ is, by definition,  the tensor product of B-splines of order four that is well studied in literature~\cite{640-4}. For us, the most important fact of such B-splines is that they are three times continuous differentiable, which makes (\ref{approx}) adaptive to the Newton-algorithm (\ref{newton-G}). Furthermore, $F(s)$ is piecewise polynomial with degree four, so the values and the derivatives of the function can be calculated very quickly and without additional numerical approximations. To sum up, the modified image registration problem becomes
\begin{eqnarray}
 \max\limits_{M\in G}\sum\limits_{s,r\in \frac{1}{\lambda}\Z^n}c_s^fc_r^g F(P M^{-1} \bar{r}-s).\nonumber
\end{eqnarray}

\begin{figure}[t]
\begin{minipage}[r]{0.5\textwidth}
\centering
\includegraphics[width=6cm]{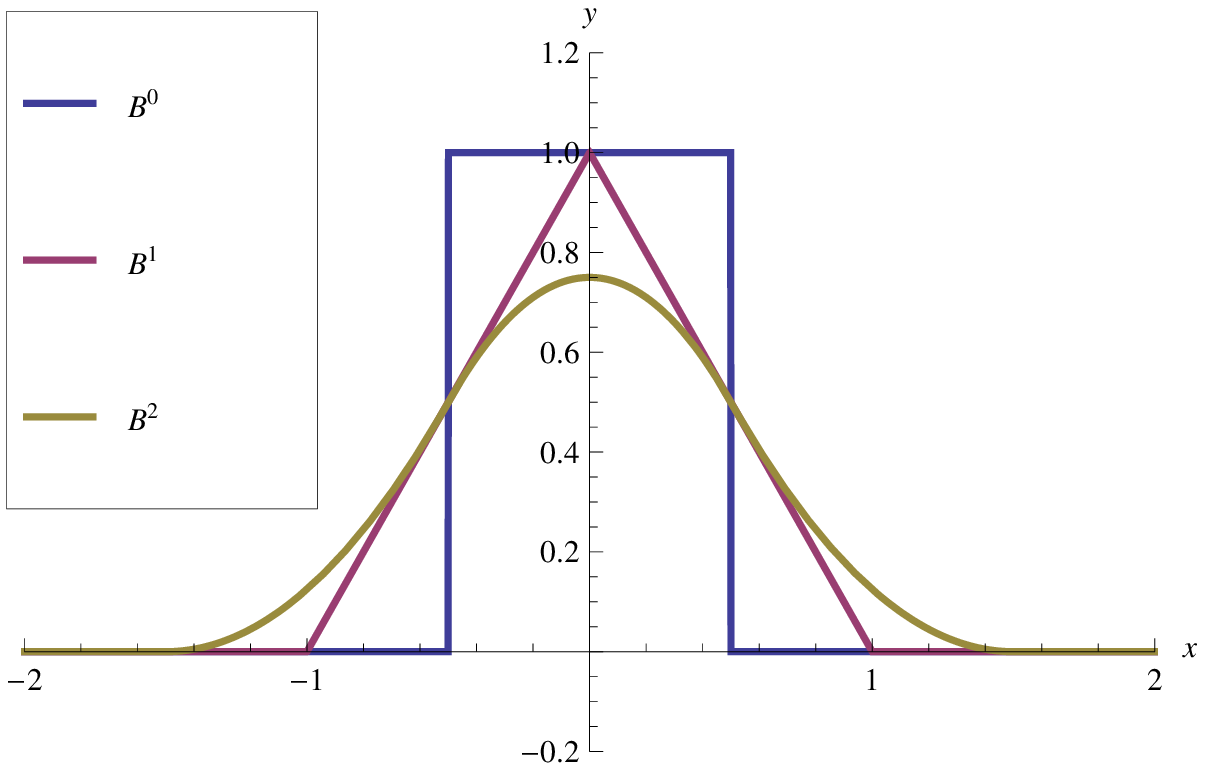}
\end{minipage}%
\begin{minipage}[l]{0.5\textwidth}
\centering
\includegraphics[width=6cm]{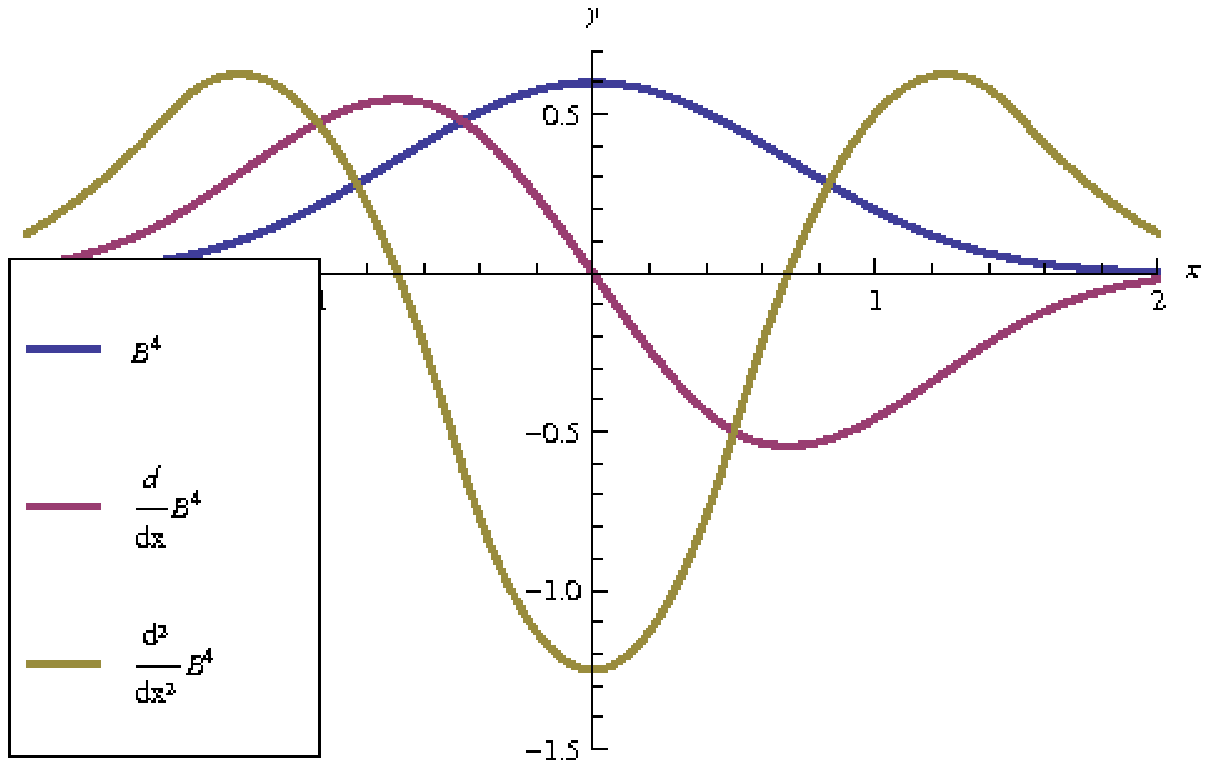}
\end{minipage}

\begin{eqnarray}
B^1(x)&=&\left\lbrace \begin{array}{lcc}
       1+x &\mbox{for}& -1 < x \leqslant 0\\
1-x &\mbox{for}& 0 < x \leqslant 1\\
0 &\mbox{else}&
                        \end{array}
\right. \nonumber\\
B^2(x)&=&\left\lbrace \begin{array}{lcc}
      \frac{1}{8} (9 + 12 x + 4 x^2) &\mbox{for}& -\frac{3}{2} < x \leqslant -\frac{1}{2}\\
\frac{1}{4}(3 - 4 x^2) &\mbox{for}& -\frac{1}{2} < x \leqslant +\frac{1}{2}\\
\frac{1}{8}(9 - 12 x + 4 x^2) &\mbox{for}& \frac{1}{2} < x \leqslant \frac{3}{2}\\
0 &\mbox{else}&
                        \end{array}
\right. \nonumber\\
B^3(x)&=&\left\lbrace \begin{array}{lcc}
      \frac{1}{6} (8 + 12 x + 6 x^2 + x^3) &\mbox{for}& -2 < x \leqslant -1\\
\frac{1}{6}(4 - 6 x^2 - 3 x^3) &\mbox{for}& -1 < x \leqslant 0\\
\frac{1}{6}(4 - 6 x^2 + 3 x^3) &\mbox{for}& 0 < x \leqslant 1\\
\frac{1}{6}(8 - 12 x + 6 x^2 - x^3) &\mbox{for}& 1 < x \leqslant 2\\
0 &\mbox{else}&
                        \end{array}
\right. \nonumber\\
 B^4(x)&=&\left\lbrace \begin{array}{lcc}
      \frac{1}{384} (5 + 2 x)^4 &\mbox{for}& -\frac{5}{2} < x \leqslant -\frac{3}{2}\\
\frac{1}{96}(55 - 20 x - 120x^2 - 80x^3 - 16x^4) &\mbox{for}& -\frac{3}{2} < x \leqslant -\frac{1}{2}\\
\frac{1}{192}(115 -120 x^2 +48x^4) &\mbox{for}& -\frac{1}{2} < x \leqslant \frac{1}{2}\\
\frac{1}{96}(55 + 20 x - 120x^2 + 80x^3 - 16x^4) &\mbox{for}& \frac{1}{2} < x \leqslant \frac{3}{2}\\
\frac{1}{384} (-5 + 2 x)^4 &\mbox{for}& \frac{3}{2} < x \leqslant \frac{5}{2}\\
0 &\mbox{else}&
                        \end{array}
\right. \nonumber
\end{eqnarray}

\caption{\footnotesize The first image shows the B-splines of order zero, one and two. The second image shows the B-spline of order four, with its first and second derivative.} \label{Splines}
\end{figure}

To calculate the Newton-step for this optimization problem, we apply the setting in section 3.2 for the objective function
\begin{eqnarray}
\Phi(M)=\sum\limits_{s,r\in \frac{1}{\lambda}\Z^n}c_s^fc_r^g F(P M \bar{r}-s)\label{Fkt-splin}
\end{eqnarray}
and obtain an analogous formula of (\ref{Gradient}) for the gradient $\nabla(\Phi\circ\mu_{M})(0,0):=(\tilde{\Omega},\tilde{v})$ with
\begin{align}
\begin{split}
\tilde{\Omega}&=\sum\limits_{s,r\in \frac{1}{\lambda}\Z^n}c_s^fc_r^g \pi_{\frak{k}} \left( \nabla F(Ar+t-s)(Ar+t)^{\top}\right) \\
\tilde{v}&=\sum\limits_{s,r\in \frac{1}{\lambda}\Z^n}c_s^fc_r^g\nabla F(Ar+t-s).\end{split}\label{Fkt-splin-gradient}
\end{align}
Similarly, formula (\ref{SimpleHessian}) for the Hessian operator  $Hess_{\Phi\circ\mu_M}(0)(\Omega,v)=(\hat{\Omega},\hat{v})$ is replaced by  
\begin{align}
 \hat{\Omega}
&=\pi_{\frak{k}}\left[ \frac{1}{2}\sum\limits_{s,r\in \frac{1}{\lambda}\Z^n}c_s^fc_r^g\Big(   \Omega^{\top}\nabla F(Ar+t-s)(Ar+t)^{\top}  \right.
\nonumber\\
&+\nabla F(Ar+t-s)(Ar+t)^{\top}\Omega^{\top}+2v(Ar+t)^{\top}\mbox{H}_{F}(Ar+t-s) \nonumber\\
&+ 2\mbox{H}_{F}(Ar+t-s)\Omega(Ar+t)(Ar+t)^{\top}\Big)  \bigg]\label{Fkt-splin-hessian}\\
\hat{v}&= \sum\limits_{s,r\in \frac{1}{\lambda}\Z^n}c_s^fc_r^g\mbox{H}_{F}(Ar+t-s)(\Omega (Ar+t)+v) \nonumber
\end{align}

Using the same calculations as in section 3 we end up with two  new \textbf{Spline-based Newton Registration-Algorithm} (SB), one is acting on $SE(n)$  and one on $SA(n)$. Both algorithms have almost the same steps as their continuous counterparts, the QMC-Newton on $SE(n)$ and on $SA(n)$ respectively, which is the reason why we don't present them here once again. The only difference in the algorithms mentioned before is the calculation of the coefficients $\alpha,\ldots,\epsilon$ which is done in step 2 of each algorithm. (See Table 3.1 or Table 3.2)

In the case of the spline-based algorithms, these coefficients have the following form:
\begin{align}
\alpha_i&=\sum\limits_{s,r\in \frac{1}{\lambda}\Z^n}c_s^fc_r^g\frac{\partial F}{\partial x_i}(Ar+t-s)\nonumber\\  
\beta_{i,j}&=\sum\limits_{s,r\in \frac{1}{\lambda}\Z^n}c_s^fc_r^g \frac{\partial F}{\partial x_i}(Ar+t-s)(Ar+t)_j \nonumber\\
\gamma_{i,j,k}&=\sum\limits_{s,r\in \frac{1}{\lambda}\Z^n}c_s^fc_r^g(Ar+t)_i\frac{\partial^2 F}{\partial x_j\partial x_k}(Ar+t-s) \nonumber\\
\delta_{i,j,k,l}&= \sum\limits_{s,r\in \frac{1}{\lambda}\Z^n}c_s^fc_r^g(Ar+t)_i(Ar+t)_j\frac{\partial^2 F}{\partial x_k\partial x_l}(Ar+t-s)\nonumber\\
\epsilon_{i,j}& =\sum\limits_{r,s\in \frac{1}{\lambda}\Z^n}c_s^fc_r^g\frac{\partial^2 F}{\partial x_i\partial x_j}(Ar+t-s). \nonumber
\end{align}

One may argue that the above presented SB-Newton algorithm requires a high degree of differentiability of the images, as it is assumed in the QMC-Newton algorithm. Exactly as in Theorem 3.4 the only condition on the images $f$ and $g$ to show the local quadratic convergence of the SB-Newton algorithms (Table 31 or Table 3.2) is that the cost function $\Phi$ in (\ref{Fkt-splin}) is in $C^3(G,\R)$. However, here, the construction of $\Phi$ via splines reduces the requirements of differentiability on $f$ and $g$. Explicitly, for $f\in \mathcal{S}_{\lambda}^p$ and $g\in \mathcal{S}_{\lambda}^q$ this condition is fullfilled if and only if $p+q\geqslant 3$. Thus, for $p=2$ and $q=1$ we have $f\in C^1(\R^n,\R)$, $g\in C(\R^n,\R)$ and get the local quadratic convergence of the SB-Newton algorithm in the same way as it was shown in the QMC-Newton algorithms (where $f\in C^3(\R^n,\R)$ is needed). Moreover, the SB-Newton algorithms need no evaluations of the derivatives of $f$ or $g$. They work directly on the given coefficients, given by the spline representation of the images.

%%%%%%%%%%%%%%%%%%%%%%%%%%%%%%%%%%%%%%%%%%%%%%%%%%%%%%%%%%%%%%%%%%%%%%

\subsection{Extension to Mutual Information}
Following the approach of Viola~\cite{Viola}, the calculation of the mutual information of two images $f$ and $g$ is done in two steps. Firstly, one has to calculate $f(x_l)-g(x_l)$ for a couple of supporting points $\left\lbrace x_l\right\rbrace _{l\in I}\subset\R^n$ to give an estimation of the joint density $\rho_{f,g}$. Afterwards, the integral $\int \rho_{f,g}\log\rho_{f,g}$ is approximated numerically - usually by a Monte Carlo method.

Let two images $f\in\mathcal{S}_{\lambda}^2$ and $g\in\mathcal{S}_{\lambda}^1$ be given. We estimate the joint density by using the coefficients $c_u^f$ and $c_u^g$ of $f$ and $g$:
\begin{eqnarray} \rho_{f,g}(x)\approx\sum\limits_{u\in\frac{1}{{\lambda}}\Z^n}\sigma B^3\left( \frac{x-(c_u^f-c_u^g)}{\sigma}\right) .\label{rho}
\end{eqnarray}
 Here, $\sigma\in\R$ controls the approximation error of $\rho_{f,g}$ (, cf. \cite{Duda} p. 88-95), and $B^3$ denotes the cubic B-Spline in one dimension (cf. Fig. 4.1). In image registration we are especially interested in the case in which $g$ is deformed by an element $M\in G$. Therefore, we have to replace $c_u^g$ in (\ref{rho}) by the coefficients $\hat{c}_u^g$ of $g(PM\bar{x})$. However, in general, $g(PM\bar{x})$ is not in $\mathcal{S}_{\lambda}^2$ for an arbitrary $M\in G$ and a additional projection step is needed to get $\hat{c}_u^g$. We make the same simplification as in the previous section and get
\begin{eqnarray}
\hat{c}_u^g\left( 
\begin{array}{cc}
 A & t\\
0 & 1
\end{array}
\right) 
\approx\sum\limits_{s,r\in \frac{1}{{\lambda}}\Z^n}c_s^g\gamma_{u,r} F(A^{-1}(r-t)+s).\label{bdach}
\end{eqnarray}
with suitable $\gamma_{u,r}\in\R$. The weights $\gamma_{u,r}$ are necessary, since the translations of the splines do not form an orthonormal system. They can be calculated explicitly and independently of $f$ and $g$. Substituting $c_u^g$ in (\ref{rho}) by (\ref{bdach}) provides an explicit formula for the objective function $\Phi:G\rightarrow \R$: 
\begin{eqnarray} 
\Phi(M):=\int\limits_{\R}\varrho\left(  \sum\limits_{u\in\frac{1}{{\lambda}}\Z^n}\sigma B^3\left( \frac{x-(c_u^f-\hat{c}_u^g(M))}{\sigma}\right)\right) dx,~~~\varrho(x):=-x\log x \label{entropy_approx}
\end{eqnarray}

In the sequel, we use the Simpson-rule to approximate the entropy $\int\! \rho_{f,g}\log\rho_{f,g}$. Following~\cite{Viola}, we have to search the maximum of $\Phi$. Since (\ref{entropy_approx}) is an three times differentiable function, we can use the same techniques used in section 3.2. We get:
\begin{align}
\Phi(M)=\frac{1}{p}&\sum\limits_{k\in\Z}\varrho\left(\sigma \sum\limits_{u\in\frac{1}{{\lambda}}\Z^n} B^3\left(\frac{k}{\sigma p}-\frac{1}{\sigma}(c_u^f-\hat{c}_u^g(M))\right)\right) \label{Zielfunktion_MI}
\intertext{with}
\nabla(\Phi\circ\mu_M)(0)=\frac{1}{p}&\sum\limits_{k\in\Z}\varrho'\left(\sigma \sum\limits_{u\in\frac{1}{{\lambda}}\Z^n} B^3\left(\frac{k}{\sigma p}-\frac{1}{\sigma}(c_u^f-\hat{c}_u^g(M))\right)\right) \nonumber\\
\cdot & \sum\limits_{u\in\frac{1}{{\lambda}}\Z^n}
{B^3}'\left(\frac{k}{\sigma p}-\frac{1}{\sigma}(c_u^f-\hat{c}_u^g(M))\right)\nabla(\hat{c}_u^g\circ\mu_M)(0)
 \nonumber
\end{align}
and
\begin{align}
&\mbox{Hess}_{\Phi\circ\mu_M}(0)(\Omega,v)\nonumber\\
&=\frac{1}{p}\sum\limits_{k\in\Z}\varrho'\left(\sigma \sum\limits_{u\in\frac{1}{{\lambda}}\Z^n} B^3\left(\frac{k}{\sigma p}-\frac{1}{\sigma}(c_u^f-\hat{c}_u^g(M))\right)\right) \nonumber\\
&\cdot\sum\limits_{u\in\frac{1}{{\lambda}}\Z^n}
{B^3}'\left(\frac{k}{\sigma p}-\frac{1}{\sigma}(c_u^f-\hat{c}_u^g(M))\right)\mbox{Hess }_     {\hat{c}_u^g\circ\mu_M}(0)(\Omega,v)
\nonumber\\
&+\frac{1}{p}\sum\limits_{k\in\Z}\varrho'\left(\sigma \sum\limits_{u\in\frac{1}{{\lambda}}\Z^n} B^3\left(\frac{k}{\sigma p}-\frac{1}{\sigma}(c_u^f-\hat{c}_u^g(M))\right)\right) \nonumber\\
&\cdot\sum\limits_{u\in\frac{1}{{\lambda}}\Z^n}
{B^3}''\left(\frac{k}{\sigma p}-\frac{1}{\sigma}(c_u^f-\hat{c}_u^g(M))\right)
\left\langle \nabla(\hat{c}_u^g\circ\mu_M)(0), (\Omega,v)\right\rangle 
\nabla(\hat{c}_u^g\circ\mu_M)(0) 
\nonumber\\
&+\frac{1}{p}\sum\limits_{k\in\Z}\varrho''\left(\sigma \sum\limits_{u\in\frac{1}{{\lambda}}\Z^n} B^3\left(\frac{k}{\sigma p}-\frac{1}{\sigma}(c_u^f-\hat{c}_u^g(M))\right)\right) \nonumber\\
&\cdot\sum\limits_{u\in\frac{1}{{\lambda}}\Z^n}
\left( 
{B^3}'\left(\frac{k}{\sigma p}-\frac{1}{\sigma}(c_u^f-\hat{c}_u^g(M))\right)
\right)^2 
\left\langle \nabla(\hat{c}_u^g\circ\mu_M)(0), (\Omega,v)\right\rangle 
\nabla(\hat{c}_u^g\circ\mu_M)(0).
\nonumber
\end{align}
Where ${B^3}'$ and ${B^3}''$ denote the first and second derivative of $B^3$ (and the same holds for $\varrho'$ and $\varrho''$).

Before writing the corresponding Newton step in components, we need some substitutions:
\begin{align}
 b_{u,k}=&B^3\left(\frac{k}{\sigma p}-\frac{1}{\sigma}(c_u^f-\hat{c}_u^g(M))\right)&\nonumber\\
{b_{u,k}}' =&{B^3}' \left(\frac{k}{\sigma p}-\frac{1}{\sigma}(c_u^f-\hat{c}_u^g(M))\right)&\nonumber\\
{b_{u,k}}''=&{B^3}''\left(\frac{k}{\sigma p}-\frac{1}{\sigma}(c_u^f-\hat{c}_u^g(M))\right)&\nonumber\\
\varrho_k'=& \varrho'\left(\sigma \sum\limits_{u\in\frac{1}{{\lambda}}\Z^n} B^3\left(\frac{k}{\sigma p}-\frac{1}{\sigma}(c_u^f-\hat{c}_u^g(M))\right)\right)\nonumber\\
\varrho_k''=& \varrho''\left(\sigma \sum\limits_{u\in\frac{1}{{\lambda}}\Z^n} B^3\left(\frac{k}{\sigma p}-\frac{1}{\sigma}(c_u^f-\hat{c}_u^g(M))\right)\right)\label{coefs_mi_1}\\
\intertext{and}
\zeta_{r,s,i}=&\sum\limits_{k,u}\left( \varrho_k'{b_{u,k}}''+\varrho_k''{b_{u,k}}'^2\right) (\tilde{v}_u)_i(\tilde{\Omega}_u)_{r,s}\nonumber\\
\eta_{r,i}=&\sum\limits_{k,u}\left( \varrho_k'{b_{u,k}}''+\varrho_k''{b_{u,k}}'^2\right) (\tilde{v}_u)_i(\tilde{v}_u)_{r}\nonumber\\
\vartheta_{r,s,i,j}=&\sum\limits_{k,u}\left( \varrho_k'{b_{u,k}}''+\varrho_k''{b_{u,k}}'^2\right) (\tilde{\Omega}_u)_{i,j}(\tilde{\Omega}_u)_{r,s}\nonumber\\
%\intertext{and}
\theta_{k,u,i}=& \varrho_k'{b_{u,k}}' (\tilde{v}_u)_i\label{coefs_mi_1_star}\\
\iota_{k,u,i,j}=& \varrho_k'{b_{u,k}}' (\tilde{\Omega}_u)_{i,j}.\nonumber
\end{align}

Here, $(\tilde{v}_u,\tilde{\Omega}_u)$ denotes the gradient of the function $\hat{c}_u^g\circ\mu_M$ in zero. Note that $\hat{c}_u^g\circ\mu_M$ is a function of the form (\ref{Fkt-splin}). Hence, the gradient and the Hessian in zero are already calculated in (\ref{Fkt-splin-gradient}) and (\ref{Fkt-splin-hessian}). Therefore, the $i$th component of the vector part of the Newton equation becomes
\begin{align}
 \sum\limits_{k,u}\varrho_{k}'{b_{u,k}}'\left(\mbox{Hess}_{\hat{c}_u^g\circ\mu_M}(0)(\Omega,v)\right)_{vector,i} 
+\sum\limits_{r,s}\zeta_{r,s,i}\Omega_{r,s}+\sum\limits_{r}\eta_{r,i}v_r
=-\sum\limits_{k,u}\theta_{k,u,i}\nonumber
\end{align}
and the $(i,j)$ component of the corresponding matrix-part is 
\begin{align}
 \sum\limits_{k,u}\varrho_{k}'{b_{u,k}}'\left(\mbox{Hess}_{\hat{c}_u^g\circ\mu_M}(0)(\Omega,v)\right)_{matrix,i,j} 
+\sum\limits_{r,s}\vartheta_{r,s,i,j}\Omega_{r,s}+\sum\limits_{r}\zeta_{i,j,r}v_r
=-\sum\limits_{k,u}\iota_{k,u,i,j}.\nonumber
\end{align}

To expand these two equations in full detail, we make the same calculations as in the algorithms before and set
\begin{align}
\hat{\alpha}_i&=\sum\limits_{u,v}\varrho_v'{b_{u,v}}'\sum\limits_{s,r\in \frac{1}{\lambda}\Z^n}c_s^g\gamma_{u,r}\frac{\partial F}{\partial x_i}(Ar+t-s),\nonumber\\ 
\hat{\beta}_{i,j}&=\sum\limits_{u,v}\varrho_v'{b_{u,v}}'\sum\limits_{s,r\in \frac{1}{\lambda}\Z^n}c_s^g\gamma_{u,r} \frac{\partial F}{\partial x_i}(Ar+t-s)(Ar+t)_j, \nonumber\\
\hat{\gamma}_{i,j,k}&=\sum\limits_{u,v}\varrho_v'{b_{u,v}}'\sum\limits_{s,r\in \frac{1}{\lambda}\Z^n}c_s^g\gamma_{u,r}(Ar+t)_i\frac{\partial^2 F}{\partial x_j \partial x_k}(Ar+t-s), \label{coefs_mi_2}\\
\hat{\delta}_{i,j,k,l}&=\sum\limits_{u,v}\varrho_v'{b_{u,v}}' \sum\limits_{s,r\in \frac{1}{\lambda}\Z^n}c_s^g\gamma_{u,r}(Ar+t)_i(Ar+t)_j\frac{\partial^2 F}{\partial x_l \partial x_k}(Ar+t-s),\nonumber\\
\hat{\epsilon}_{i,j}& =\sum\limits_{u,v}\varrho_v'{b_{u,v}}'\sum\limits_{r,s\in \frac{1}{\lambda}\Z^n}c_s^g\gamma_{u,r}\frac{\partial^2 F}{\partial x_i \partial x_j}(Ar+t-s). \nonumber
\end{align}

Now we have all necessary instruments to present the Newton equation. In the case of the rigid registration we end up with the following analogue to the linear system (\ref{begin_newton_SE(n)})-(\ref{end_newton_SE(n)}):

\begin{lemma}
Let $(\Omega,v)\in so(n)\times \R^n$ be the Newton-direction for the objective function (\ref{Zielfunktion_MI}) in a certain point $M\in SE(n)$. Then the components $\Omega_{k,l}$ $, 1\leqslant k,l\leqslant n,$ of $\Omega$ and $v_,k$ $1\leqslant k\leqslant n,$ of $v$ satisfy
\begin{align}
 \sum\limits_{k>l}(\hat{\gamma}_{l,k,i}-\hat{\gamma}_{k,l,i}+\zeta_{k,l,i}-\zeta_{l,k,i})\Omega_{k,l}+\sum\limits_k(\hat{\epsilon}_{i,k}+\eta_{k,i})v_k=-\sum\limits_{k,u}\theta_{k,u,i}\label{n_start}\end{align}
for all $1\leqslant i\leqslant n$ and
\begin{align}
 &\frac{1}{2}\sum\limits_{k>j}(\hat{\beta}_{i,k}+\hat{\beta}_{k,i})\Omega_{k,j}-\frac{1}{2}\sum\limits_{k<j}(\hat{\beta}_{i,k}+\hat{\beta}_{k,i})\Omega_{j,k}
-\frac{1}{2}\sum\limits_{k>i}(\hat{\beta}_{j,k}+\hat{\beta}_{k,j})\Omega_{k,i}\nonumber\\
&+\frac{1}{2}\sum\limits_{k<i}(\hat{\beta}_{j,k}+\hat{\beta}_{k,j})\Omega_{i,k}
-\sum\limits_{k>l}(\hat{\delta}_{i,k,l,j}-\hat{\delta}_{j,l,k,i}+\hat{\delta}_{i,l,k,j}-\hat{\delta}_{i,k,l,j}-\vartheta_{k,l,i,j}+\vartheta_{l,k,i,j})\Omega_{k,l}\nonumber\\
&~~~~~~~~~~~~~~~~~~~~~~~~~~~~~~~~~~~~~~~~~~~~~~~~~~-\sum\limits_{k}(\hat{\gamma}_{j,k,i}-\hat{\gamma}_{i,k,j}-\zeta_{i,j,k})v_k=-\sum\limits_{k,u}\iota_{k,u,i,j}.\label{n_end}
\end{align}
for all $1\leqslant i<j\leqslant m$.
\end{lemma}

\begin{figure}[t]
\begin{minipage}[m]{0.3\textwidth}
\fbox{
\parbox{12.5cm}{
\textbf{Table 4.1: Mutual Information-based Registration-Algorithm on SE(n)}\\[0.5cm]
Step 1.\\
Pick an initial guess $M_0\in SE(n)$ and set $m=0$.\\[0.5cm]
Step 2.\\
Calculate all coefficients in (\ref{coefs_mi_1}), (\ref{coefs_mi_1_star}) 
and (\ref{coefs_mi_2}). \\[0.5cm]
Step 3.\\
Solve the linear system described in (\ref{n_start}) and (\ref{n_end}) with the unknowns $v_i$, $1\leqslant i\leqslant n$ and $\Omega_{i,j}$, $1\leqslant i<j\leqslant n$.\\[0.5cm]
Step 4.\\
Construct the $n\times n$ matrix $\Omega$ with entries $\Omega_{i,j}$. If $j>i$ use the solution of step 3 or else set
 \begin{eqnarray}
      \Omega_{i,j}=\left\lbrace 
\begin{array}{cc}
 -\Omega_{j,i} & \mbox{for~} j<i\\
 0 & \mbox{for~} j=i.
\end{array}
\right. \nonumber
    \end{eqnarray}
Compute 
\begin{eqnarray}
 M_{m+1}:=\nu_{M_m}^{QR}(\Omega,v),\nonumber
\end{eqnarray}
where $\nu^{QR}$ is defined in (\ref{Karte-SE(n)}).\\[0.5cm]
Step 5.\\
Set $m=m+1$ and goto Step 2.
}}
\end{minipage}
\end{figure}

\begin{figure}[t]
\begin{minipage}[m]{0.3\textwidth}
\fbox{
\parbox{12.5cm}{
\textbf{Table 4.2: Mutual Information-based Registration-Algorithm on SA(n)}\\[0.5cm]
Step 1.\\
Pick an initial guess $M_0\in SA(n)$ and set $m=0$.\\[0.5cm]
Step 2.\\
Calculate all coefficients in (\ref{coefs_mi_1}), (\ref{coefs_mi_1_star}) and (\ref{coefs_mi_2}). \\[0.5cm]
Step 3.\\
Solve the linear system described in (\ref{f0_mi}) - (\ref{f4_mi}) with the unknowns $v_i$, $1\leqslant i\leqslant n$ and $\Omega_{i,j}$, $1\leqslant i,j\leqslant n$ and $(i,j)\neq (n,n)$.\\[0.5cm]
Step 4.\\
Construct the $n\times n$ matrix $\Omega$ with entries $\Omega_{i,j}$. If $(j,i)\neq (n,n)$ use the solution of step 3 or else set
\begin{eqnarray}
\Omega_{n,n}=-\sum\limits_{k \neq n} \Omega_{k,k}\nonumber
    \end{eqnarray}
and compute 
\begin{eqnarray}
 M_{m+1}:=\nu_{M_m}^{QR}(\Omega,v),\nonumber
\end{eqnarray}
where $\nu^{QR}$ is defined in (\ref{Karte-SA(n)}).\\[0.5cm]
Step 5.\\
Set $m=m+1$ and goto Step 2.
}}
\end{minipage}
\end{figure}

We finish this section with the corresponding lemma in the case of volume-preserving transformations.

\begin{lemma}
Let $(\Omega,v)\in sl(n)\times \R^n$ be the Newton-direction for the objective function (\ref{Zielfunktion_MI}) in a certain point $M\in SA(n)$. Then the components $\Omega_{k,l},$ $ 1\leqslant k,l\leqslant n,~(k,l)\neq (n,n)$ of $\Omega$ and $v_k$ $1\leqslant k\leqslant n$ of $v$ satisfy for each $1\leqslant i\leqslant n$ 
\begin{align}
 \sum\limits_{k\neq l}(\hat{\gamma}_{l,k,i}+\zeta_{k,l,i})\Omega_{k,l}+\sum\limits_{k\neq n}(\hat{\gamma}_{k,k,i}-\hat{\gamma}_{n,n,i}+\zeta_{k,k,i}-\zeta_{n,n,i})\Omega_{k,k}\nonumber\\
+\sum\limits_{k}(\hat{\epsilon}_{i,k}-\eta_{k,i})v_k=-\sum\limits_{k,l}\theta_{k,l,i}.\label{f0_mi}
\end{align}
 and for all $1\leqslant i,j\leqslant n$, $(i,j)\neq(n,n)$ the following equations: 
\begin{align}\begin{split}
 \sum\limits_k\hat{\beta}_{i,k}\Omega_{j,k}+\sum\limits_{k}\hat{\beta}_{k,j}\Omega_{k,i}
+\sum\limits_{(k,l)\neq (n,n)}(\hat{\delta}_{j,l,k,i}+\vartheta_{l,k,i,j})\Omega_{l,k}\\
-\sum\limits_{k\neq n}(\hat{\delta}_{j,k,k,i}-\vartheta_{k,k,i,j})\Omega_{k,k} 
+\sum\limits_k(\hat{\gamma}_{j,k,i}-\zeta_{i,j,k})v_k
=-\sum\limits_{k,l}\iota_{k,l,i,j}\end{split}\label{f1_mi}
\end{align}
for $i\neq n,~j\neq n,~i\neq j$,\\
\begin{align}\begin{split}
 \sum\limits_k\hat{\beta}_{n,k}\Omega_{j,k}+\sum\limits_{k\neq n}\hat{\beta}_{k,j}\Omega_{k,n}
+\sum\limits_{(k,l)\neq (n,n)}(\hat{\delta}_{j,l,k,n}+\vartheta_{l,k,n,j})\Omega_{l,k}\\
-\sum\limits_{k\neq n}(\hat{\delta}_{j,k,k,i}-\vartheta_{k,k,n,j}+\frac{1}{2}\hat{\beta}_{n,j})\Omega_{k,k} 
+\sum\limits_k(\hat{\gamma}_{j,k,n}-\zeta_{n,j,k})v_k
=-\sum\limits_{k,l}\iota_{k,l,n,j}\end{split}\label{f2_mi}
\end{align}
for $i=n,~j\neq n$,\\
\begin{align}\begin{split}
 \sum\limits_{k\neq n}\hat{\beta}_{i,k}\Omega_{n,k}+\sum\limits_{k}\hat{\beta}_{k,n}\Omega_{k,i}
+\sum\limits_{(k,l)\neq (n,n)}(\hat{\delta}_{n,l,k,i}+\vartheta_{l,k,i,n})\Omega_{l,k}\\
-\sum\limits_{k\neq n}(\hat{\delta}_{n,k,k,i}-\vartheta_{k,k,i,n}+\frac{1}{2}\hat{\beta}_{i,n})\Omega_{k,k} 
+\sum\limits_k(\hat{\gamma}_{n,k,i}-\zeta_{i,n,k})v_k
=-\sum\limits_{k,l}\iota_{k,l,i,n}\end{split}\label{f3_mi}
\end{align}
for $j=n,~i\neq n$ and\\
\begin{align}\begin{split}
\sum\limits_k\hat{\beta}_{i,k}\Omega_{j,k}+\sum\limits_{k}\hat{\beta}_{k,j}\Omega_{k,i}
+\sum\limits_k\left(\hat{\gamma}_{j,k,i}-\zeta_{i,j,k}-\frac{1}{n}\sum\limits_{l}\hat{\gamma}_{l,l,k}\right)v_k\\
+\sum\limits_{(k,l)\neq (n,n)}\left(\hat{\delta}_{j,l,k,i}+\vartheta_{l,k,i,j}-\frac{1}{2}\left( \hat{\beta}_{l,k}+\sum\limits_m\hat{\delta}_{k,m,l,m}\right)\right)\Omega_{l,k}\\
-\sum\limits_{k\neq n}\left(\hat{\delta}_{j,k,k,i}-\vartheta_{k,k,i,j}-\frac{1}{n}
\left( \hat{\beta}_{k,k}+\sum\limits_m\hat{\delta}_{k,m,k,m}\right)
\right)\Omega_{k,k} \\
=-\sum\limits_{k,l}\left(\iota_{k,l,i,j}-\frac{1}{n}\sum\limits_m\iota_{k,l,m,m}\right)\end{split}\label{f4_mi}
\end{align}
for $i=j,~i,j\neq n$.
\end{lemma}

With this lemmas at hand, one can easily check as bevor that the algorithms presented in Table 4.1 and Table 4.2 are locally quadraticaly convergent.

%%%%%%%%%%%%%%%%%%%%%%%%%%%%%%%%%%%%%%%%%%%%%%%%%%%%%%%%%%%%%%%%%%%%%%%%%

\section{Experimental Results}
\begin{figure}[t]
 \centering
\includegraphics[height=6cm,width=12.5cm]{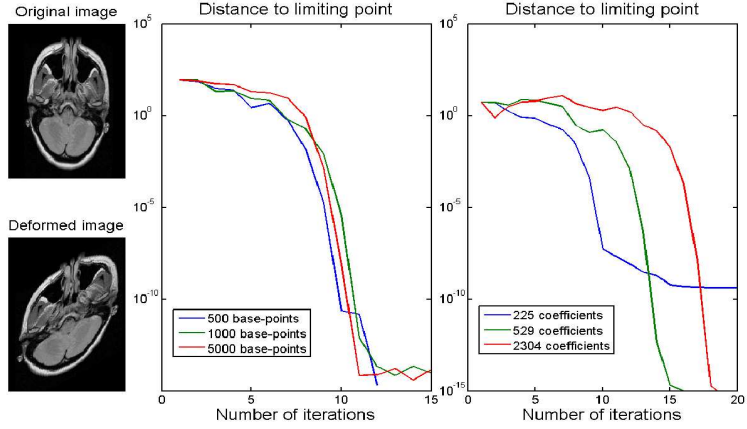}
\caption{\footnotesize Consideration of convergence speed for a $SA(2)$ problem. Middle: QMC-Newton algorithm for several number of base-points. Right: SB-Newton algorithm for several spline-approximations of the images.}\label{C_SA(2)}
\end{figure}

All computations in this section are performed using MATLAB R2008a on a 1.9 GHz laptop with 2 GB RAM. In our first example (Fig. \ref{C_SA(2)}), we demonstrate the local quadratic convergence rate of the described algorithms. We took a 2D cross-section ($250\times 250$ pixel) of a CT data set. The reference image of our artifical problem is the corresponding spline-approximation, with 289 coefficients. The template is identical to the reference. In this example we always start with the initial transformation 

\begin{eqnarray}
 M_0=\left( 
\begin{array}{rrr}
 1.2 & 0.5 & -87.5\\
0 & 0.8333 & 20.8\\
0 & 0 & 1
\end{array}
\right) .\label{initial_cond}
\end{eqnarray}

(Figure \ref{C_SA(2)} left buttom shows the original image transformed by this matrix.) The middle of Fig. \ref{C_SA(2)} plots the distance of the limiting point using the registration algorithm on $SA(2)$. To calculate the coefficients (\ref{coefs}) we use the Quasi Monte Carlo approximation with the 500, 1000 and 5000 first elements of the Halton sequence~\cite{Halton}. In all three cases we see a local quadratic convergence and in only 12 steps we achieve an accuracy $<10^{-12}$.
Since the objective function (\ref{appsum}) is a discretized and approximated versions of the correlation-measure (\ref{Zielfunktion}), there is a discrepancy between the exact and the detected transformation. That is, the limiting point of a particular iteration in Fig. (\ref{C_SA(2)}) is close to but not equal to identity matrix. 
For example, for 5000 base-points, the algorithm ends with 
\begin{eqnarray}
 M_{20}=\left( 
\begin{array}{rrr}
 1.009 & 0.001 & -2.3\\
0.008 & 0.991 & 0.24\\
0 & 0 & 1
\end{array}
\right) ,\nonumber
\end{eqnarray}
which is not very close to the optimal solution $M=I_3$. (However, this gap would not appear in a discretized version of the ``sum of squared difference''.) The graph on the right of Fig. \ref{C_SA(2)} shows the speed of convergence of the spline-based registration algorithm. Here, the reference image is the spline-approximation of the original image with 225, 529 and 2304 coefficients respectively. The template is again identical to the reference and the initial transformation is again (\ref{initial_cond}) for each experiment. Once more, we obtain a local quadratic convergence. In contrast to the Monte Carlo version of the registration algorithm, the limiting point of this SB-Newton registration is much closer to the identity matrix: In the case of 225 spline-coefficients the algorithm end with
\begin{eqnarray}
 M_{20}=\left( 
\begin{array}{rrr}
 1.0008 & -0.0004 & -0.043\\
-0.0002 & 0.9992 & 0.12\\
0 & 0 & 1
\end{array}
\right) .\nonumber
\end{eqnarray}
In case of 2304 coefficients we achieve a distance to the identity matrix of $3.7\cdot 10^{-4}$. 

\begin{table}[t]
 \begin{minipage}[l]{0.25\textwidth}
  \includegraphics[height=3cm]{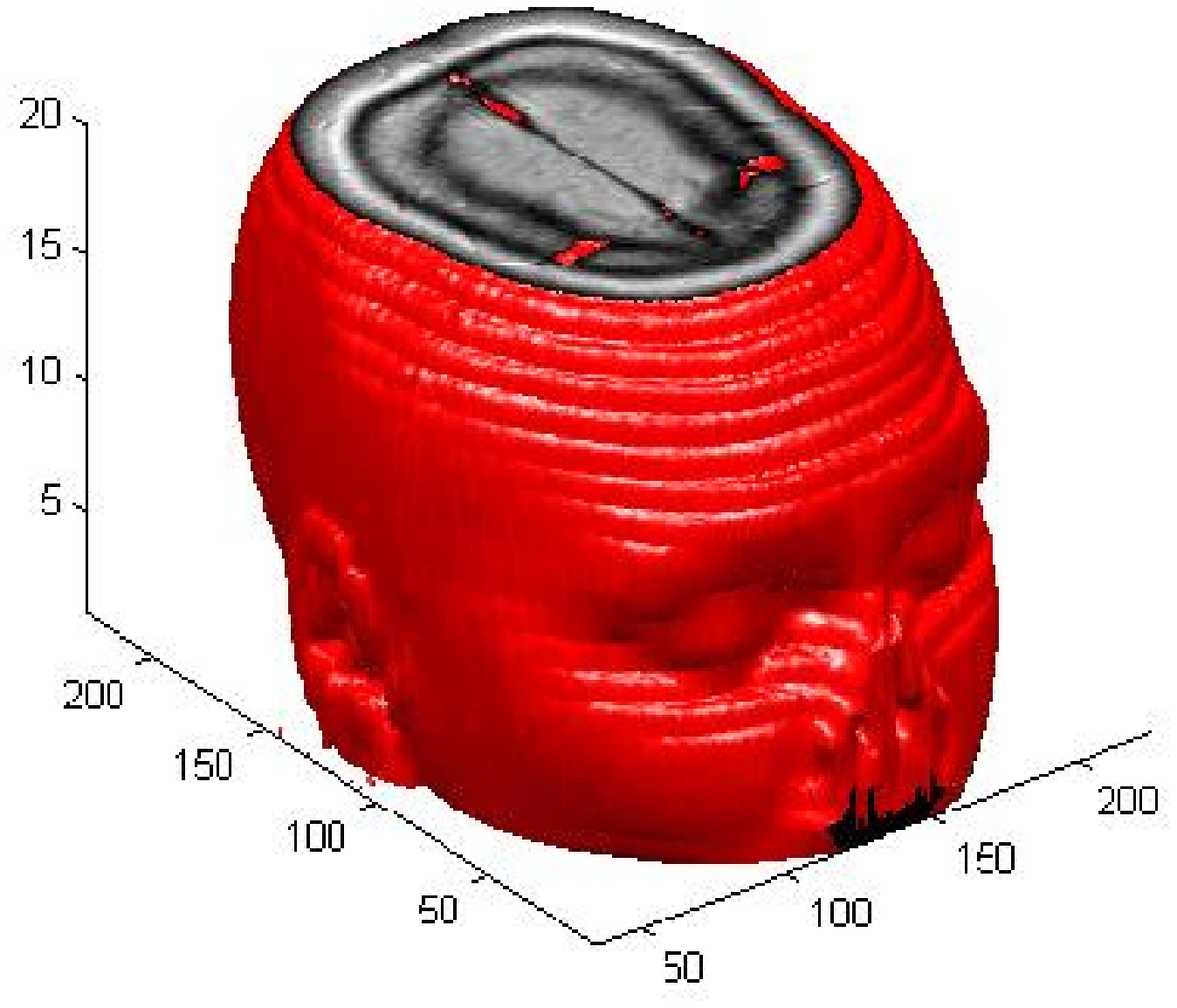}
  \includegraphics[height=3cm]{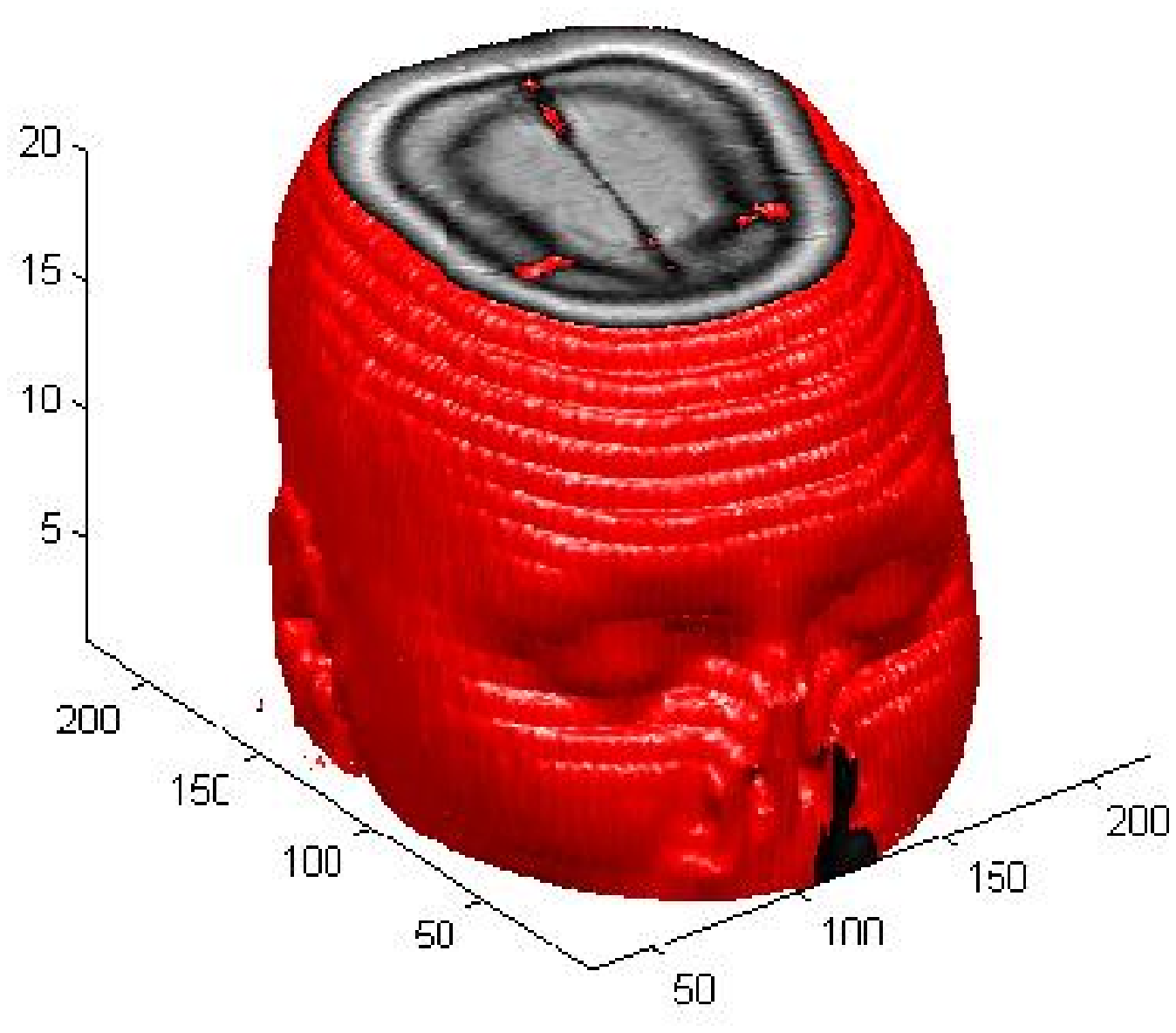}
 \end{minipage}
\begin{minipage}[mr]{0.73\textwidth}
\caption{\footnotesize Comparison of the SB-Newton with the QMC-Newton algorithm. We register the distance from the requested transformation to the result of the particular algorithm.\label{C_SE(3)}}~\\
{\footnotesize
 \begin{tabular}{|p{2cm}||p{2cm}|p{1cm}|p{1cm}|p{1cm}|}
\firsthline
Number of  & SB-Newton  & \multicolumn{3}{c|} {QMC-Newton Algorithm} \\\cline{3-5}
Coefficients & Algorithm & 2000 & 6000 & 10000\\\hline\hline
 $11\times 11\times 18$ & 0.0045 & 0.075 & 0.023 & 0.0056 \\
 & 0.45 & 9.5 & 2.3 & 1.5 \\\hline
 $16\times 16\times 18$ & 0.0016 & 0.013 & 0.017 & 0.014\\
& 0.15 & 4.5 & 1.7 & 2.1\\\hline
 $24\times 24\times 18$ & $8.8\cdot 10^{-5}$ & 0.042 & 0.0042 & 0.0036\\
& 0.47 & 5.6 & 0.49 & 0.7 \\\hline
 $50\times 50\times 18$ & $1.1\cdot 10^{-4}$ & 0.39 & 0.0026 & 0.0027\\
& 0.4 & 52 & 0.3 & 0.19\\\hline
 \end{tabular}}
\end{minipage}
\end{table}

In our next example, we examine the convergence of the $SE(3)$ algorithms. We consider a $250\times 250\times 20$ CT data set of a head (see Table 5.1). In contrast to the first example, we first transform the image and then make use of the spline-approximation to achieve the template (which is much closer to a natural registration problem). The used transformation is  
\begin{eqnarray}
 \left( 
\begin{array}{rrrr}
 0.9553 & -0.2955 & 0 &-32.1\\
-0.2955 & 0.9553 & 0 &43.5\\
 0 & 0 & 1 & 0\\
0 & 0 & 0 & 1
\end{array}
\right) \label{example2}
\end{eqnarray}
which is consistent with a rotation of 17.2 degrees around the central principal axis of inertia. In Table 5.1 we note the difference between the detected and the requested transformation: The first number in each box gives the distance of the detected to the real rotation matrix in the Frobenius norm, the second number the Euclidean distance of the detected to the real translation in pixel length. Note that a translation error smaller than 1, which is the size of one voxel, can be seen as perfect. We vary the number of coefficients used for the spline-approximation and the number of elements of the Halton sequence used for the Quasi Monte Carlo Method.

We recapitulate: For the algorithms defined in section 3.1.1 we need two approximations: A spline-approximation to smooth the image and a couple of base-point to approximate the integrals (\ref{coefs}). Also two approximations are made in the algorithms of section 4: The spline-based image smoothing and the approximation of the objective function (\ref{approx}). In case of the SB-Newton algorithm we achieve a overwhelming accuracy even for very strongly smoothed images. For a given smoothing level the operation on the spline function space seems to be superior to the Monte Carlo approximation. On the other hand, the speed of the Monte Carlo based algorithms is hardly influenced by the level of image smoothing. It depends nearly completely on the chosen number of base-points for the integral approximation. Therefore, a comparison of the algorithms focused on the rows of Table \ref{C_SE(3)} is limited. Comparing the columns of Table \ref{C_SE(3)} could lead to the impression that we achieve better results by rising the number of spline-coefficients. This is not true generally: For example, by reducing the smoothing local extrema appears, which leads to wrong results. This happens, for example, in the last row of Table \ref{C_SE(3)} for 2000 base-points. 

\begin{figure}[t]
\begin{minipage}[c]{1\textwidth}
\centering
\includegraphics[height=5.5cm, width=13cm]{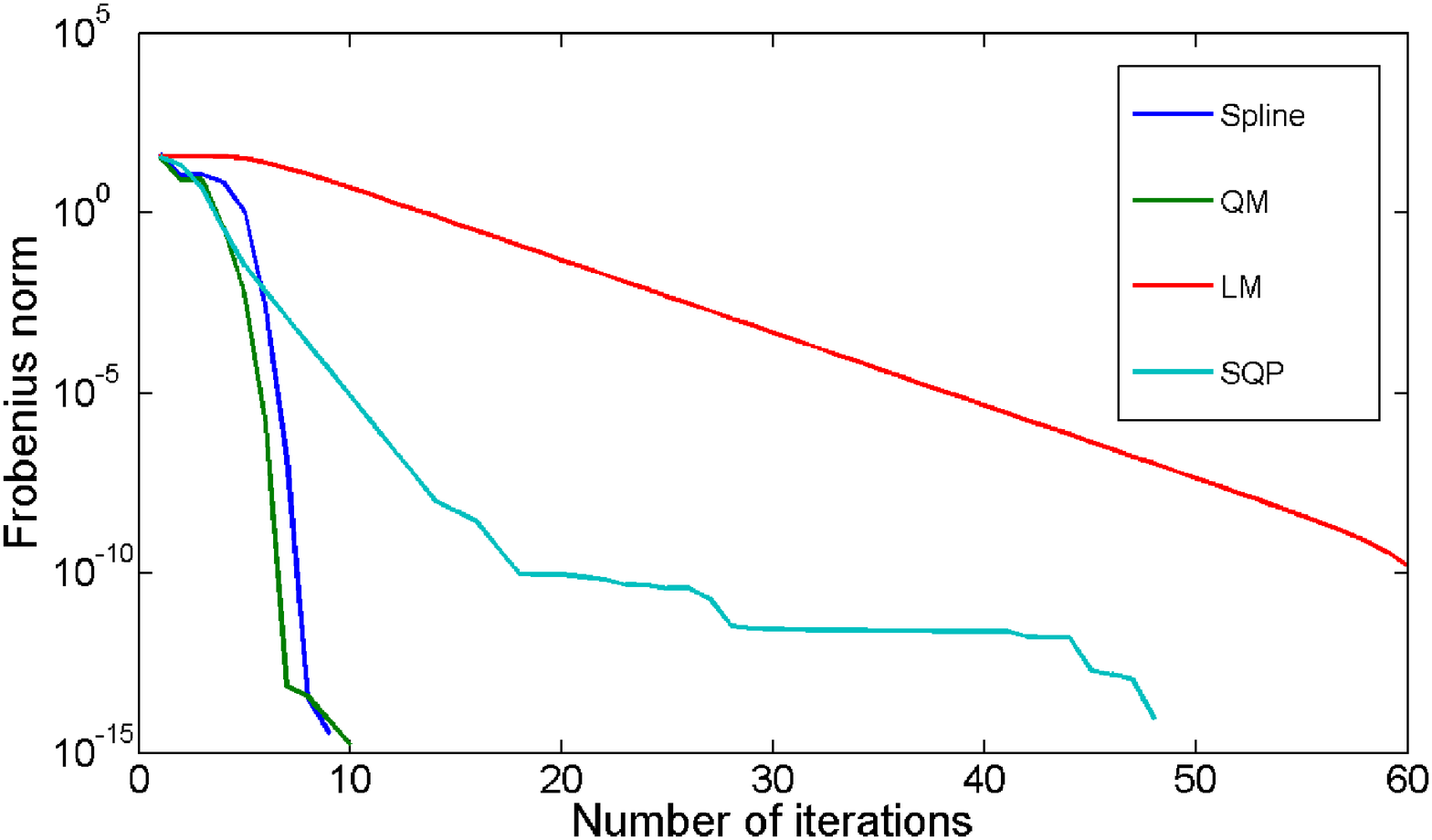}
\end{minipage}
\caption{\footnotesize Comparison of the SB-Newton, QMC-Newton, Levenberg-Marquardt (LM) and Sequential Quadratic Programming (SQP) algorithms. We choose a 2D-CT slide ($256\times 256$ pixel) as the reference and its translation along 20 pixel as the template (cf. Table 5.2). We plot the distance of the calculated transformation to the particular limiting point.\label{C_LM}}
\end{figure}

 We have already mentioned that the speed of a QMC-Newton step depends linearly on the number of base-points. In contrast the velocity of the spline-based version depends in a linear manner on the number of spline-coefficients of the reference. (To see this we make use of the fact that the function $F$ in (\ref{Fkt-splin}) has a compact support.) A spline-based Newton step on $24\times 24\times 18$ spline-coefficients take 5.3 seconds. Whereas a Quasi Monte Carlo-based Newton step for 6000 base-points takes only 1.9 seconds. However, this value depends extremely on the way of implementation, especially on the use of preimplemented MATLAB methods.

We proceed with a comparison to an established intensity-based registration algorithm. We choose the Levenberg-Marquardt-like (LM) algorithm described by  Th\'{e}venaz et al (cf. \cite{Levenberg}), since the same framework requirements are chosen to formulate the problem. Above all, they also make use of the spline-approximation of the images several times. In their approach, they use the substitution-rule in (\ref{ssd-diskret}) to get a good approximation of the derivative of the discretization (\ref{QM}) to achieve low computational costs. We also compare our algorithm with a classical, extrinsic Sequential Quadratic Programming (SQP) algorithm. (See e.g. \cite{Fletcher} chapter 12.4 for an introduction.) In this method, the objective function is approximated quadratically in every step, while the constraints are linearized. The resulting Newton-step is the performed in a vector-space of bigger dimension.

Just like in the beginning, we take the previous 2D-CT slide as reference and a translated and rotated version of of it for the template. In Fig. \ref{C_LM} we demonstrate the convergence-behavior of the SB-Newton, the QMC-Newton, the LM and the SQP-algorithm. We plot the logarithm of the distance of the actual guess $M_i$ to the particular limiting point $M_{SB}^{\star},~M_{QM}^{\star},~M_{LM}^{\star},~M_{SQP}^{\star}$. For each algorithm we use the same number of coefficients for the spline-approximation. Even though SB-Newton and QMC-Newton algorithms seem to dominate, we have to keep in mind that a difference in translation under $10^0$ pixel is still an excellent result,  which is attained by all four algorithms after 16 steps. But in contrast to LM we can choose a classical first derivative-based stopping criteria for the other algorithms: 
\begin{eqnarray}
 \|\nabla\Phi(M_i)\|<T\nonumber
\end{eqnarray}
where $T$ is an arbitrary threshhold. Because of the result in Fig. \ref{C_SA(2)} we can choose $T$ arbitrary small and, for sure, the stopping criteria will be achieved without running unnecessary many iteration steps. Therefore, no user-specification is necessary to get results of interest. 

\begin{table}[t]
\begin{minipage}[c]{1\textwidth}
{\footnotesize 
\begin{tabular}{|p{2cm}||p{2.2cm}|p{2.2cm}|p{2.2cm}|p{2.2cm}|}
\firsthline
\multicolumn{1}{|c||} {Requested}  & \multicolumn{4}{c|} {Detected Transformations} \\\cline{2-5}
  Transformation & SB-Newton & QMC-Newton & LM & SQP\\[0.1cm]\hline\hline
& & & &\\[-0.2cm]
$0^{\circ},~{0\choose 20}$    & $0^{\circ},~{0\choose 20}$ & $0.063^{\circ},~{0.09\choose 20.07}$ & $0.012^{\circ},~{0.08\choose 20.04}$ &
$0.063^{\circ},~{0.1\choose 20.07}$ 
 \\[0.2cm]\hline
& & & \\[-0.2cm]
$0^{\circ},~{40\choose 0}$  & $0.023^{\circ},~{39.44\choose 0.04}$ & $0.092^{\circ},~{38.94\choose 0.17}$ &  $0.608^{\circ},~{15.98\choose -0.84}$ &
$0.092^{\circ},~{38.94\choose 0.17}$
\\[0.2cm]\hline
& & & \\[-0.2cm]
$17.188^{\circ},~{0\choose 0}$  & $17.296^{\circ},~{1.25\choose -1.9}$ & $15.564^{\circ},~{-0.39\choose -0.02}$ &  $16.506^{\circ},~{0.17\choose -0.07}$ &
$15.564^{\circ},~{-0.39\choose -0.02}$  \\[0.2cm]\hline
 \end{tabular}}
\end{minipage}
\begin{minipage}[c]{1\textwidth}
\centering
\includegraphics[height=3cm]{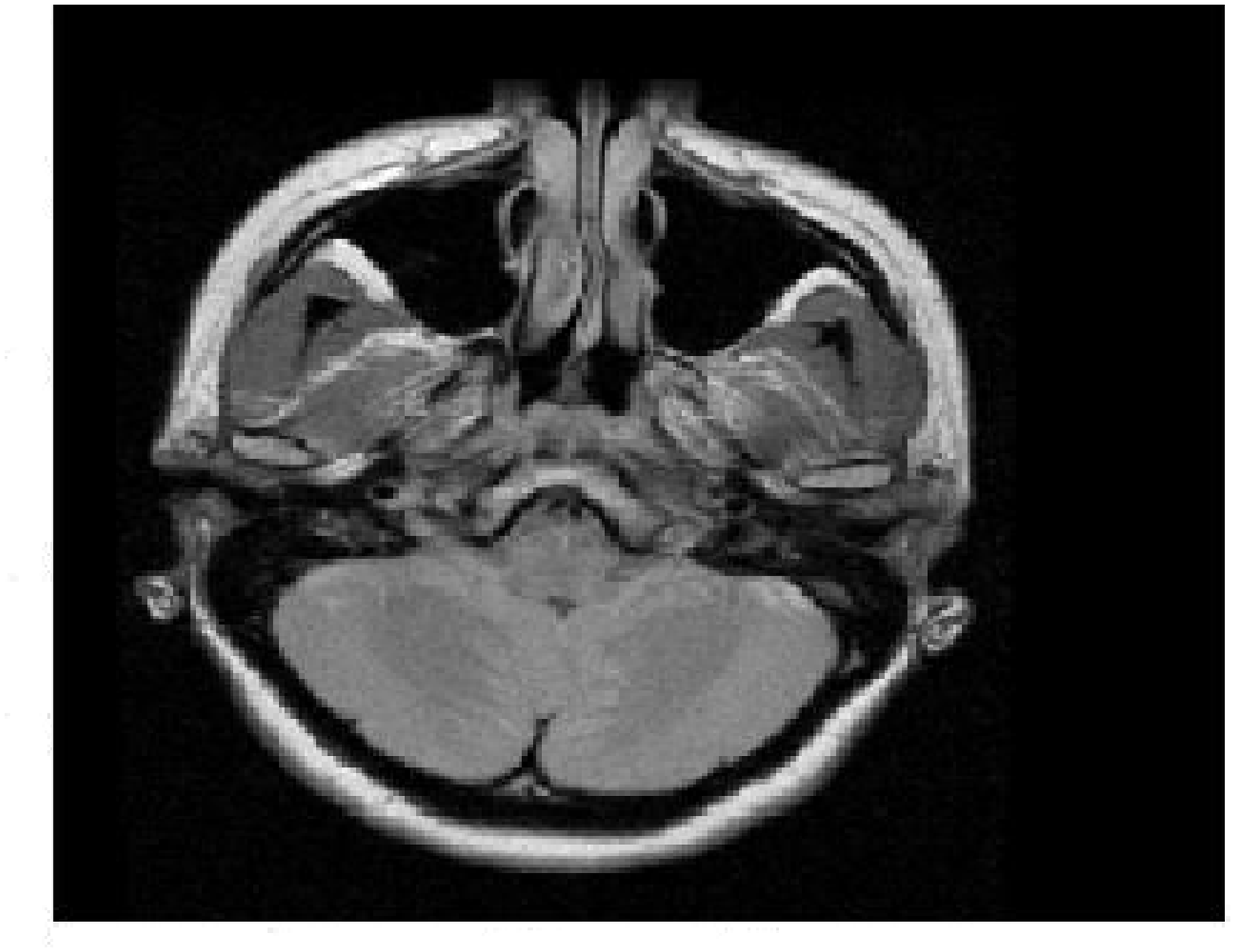}
\includegraphics[height=3cm]{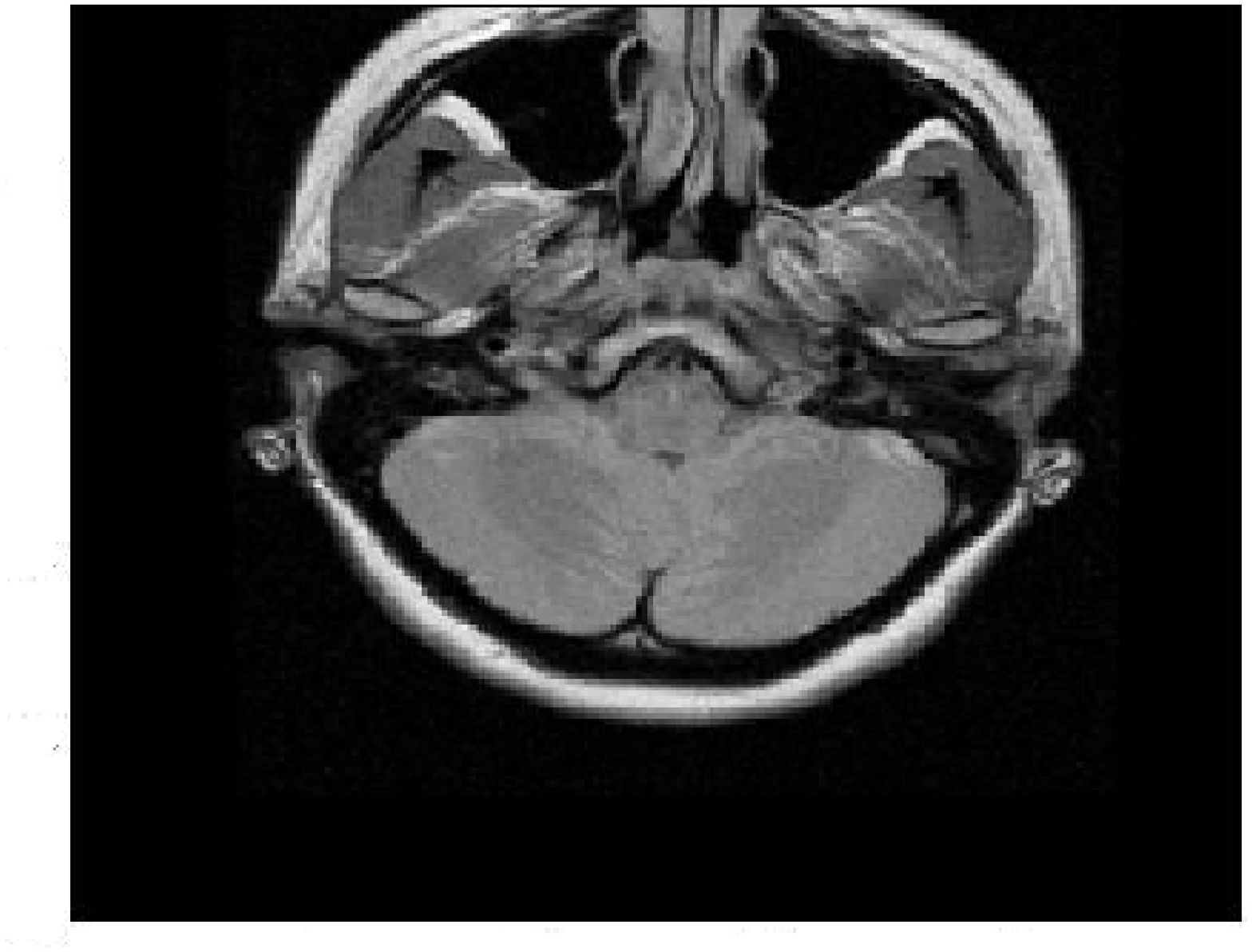}
\includegraphics[height=3cm]{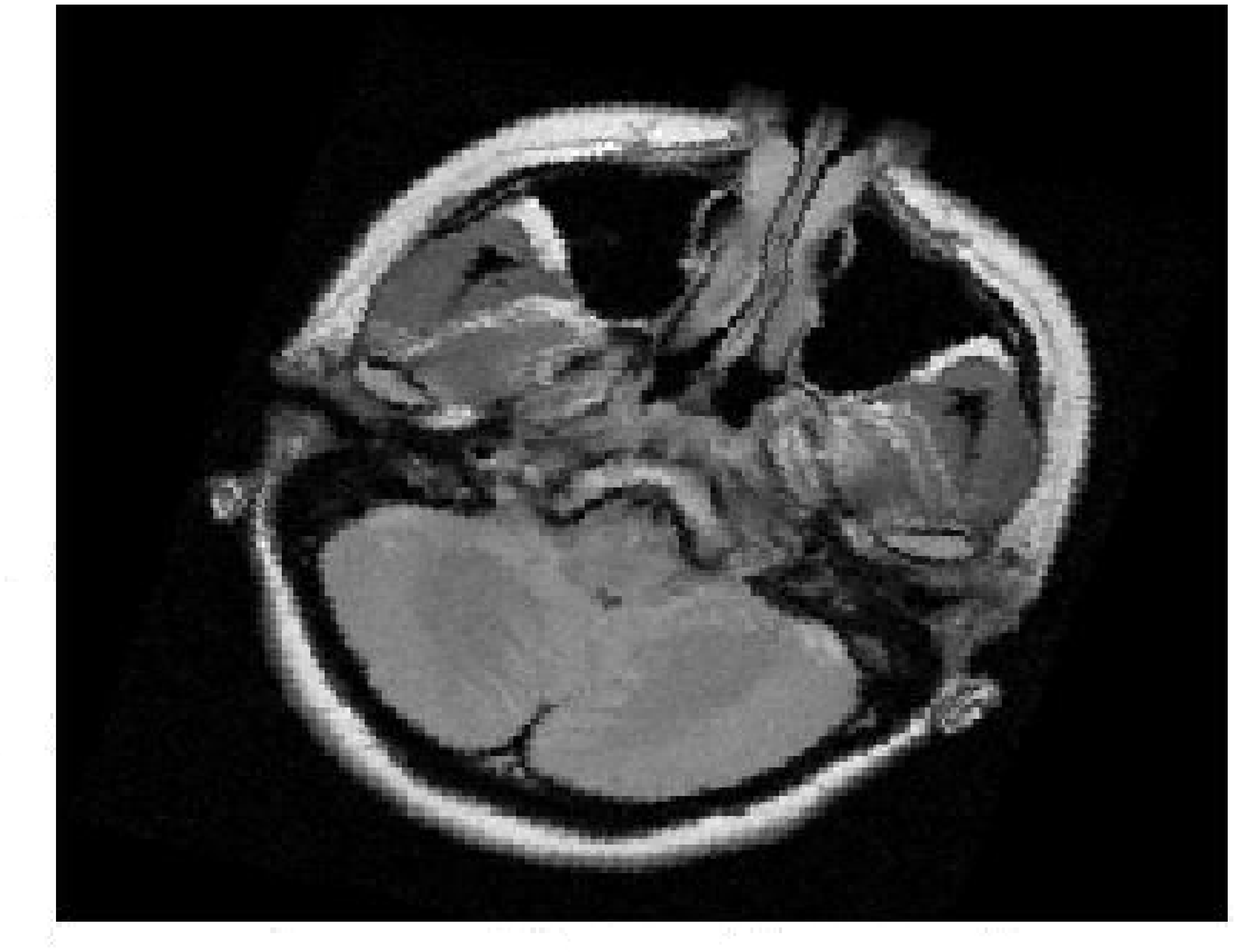}
\caption{\footnotesize Comparison of the SB-Newton, QMC-Newton, LM and SQP algorithms. The first column presents the requested transformation and the following columns the calculated transformations with the particular algorithms. The first number gives the rotation around the center of the image, the following vector gives the translation in pixel. The registration is done after a spline-approximation with $25^2$ coefficients for each algorithm. Below the table, we present the templates of the particular registration problem. The reference is equal to the one in the first experiment.\label{LMSPQ}}
\end{minipage}
\end{table}

Another result from Fig.~\ref{C_LM} is the superior convergence rate of the SB-Newton and the QMC-Newton algorithms compared to the SQP method. This stresses the point of view that intrinsic algorithms dominate extrinsic algorithms if the constraints form a differentiable manifold. By embedding the set of admissible points in a vector space, which is done in an extrinsic algorithm like the SQP, the dimension of the optimization problem may explodes. As an example, the SQP method searches for the optimum of a function $f:SE(3)\to\R$ in a space of 18 dimension (9 because of $A\in gl(3)$ instead of $A\in SO(3)$, 3 because of $t\in\R^3$ and 6 parameters are needed for the Lagrange multipliers), whereas the SB-Newton and QMC-Newton methods optimize over a space of 6 dimension (3 because of $A\in SO(3)$, 3 because of $t\in\R^3$). Hence, extrinsic algorithm may need unnecessarily many steps, a higher complexity and  finally, a projection step is needed in order to make sure that the result is an admissible point, since only the limiting point of an extrinsic algorithm is guarantee to be admissible.

Since the algorithms use different kinds of approximations, the particular limiting points diverge. We listed them in Table \ref{LMSPQ} for a few registration problems. Since we used the same cost function for the QMC-Newton and the SQP algorithm it is not surprising that the detected transformations are equal for both methods. It is notable that in the second row the detected translation in the case of the LM-algorithm is wrong. The main reason for this is that LM minimize the ``sum of squared difference''-measure while the QMC or SB-Newton algorithm maximize the correlation between two images. Of course, one can pass by such missleadings through reshaping the region of interest or  through starting at a coarser scale of spline-approximation, but this option is also applicable for the other algorithms. 

\begin{figure}[t]
%\begin{minipage}[m]{1.3\textwidth}
 \begin{tabular}{p{4cm}p{8.5cm}}
\begin{minipage}[l]{4cm}
\centering
\includegraphics[height=4cm,width=4cm]{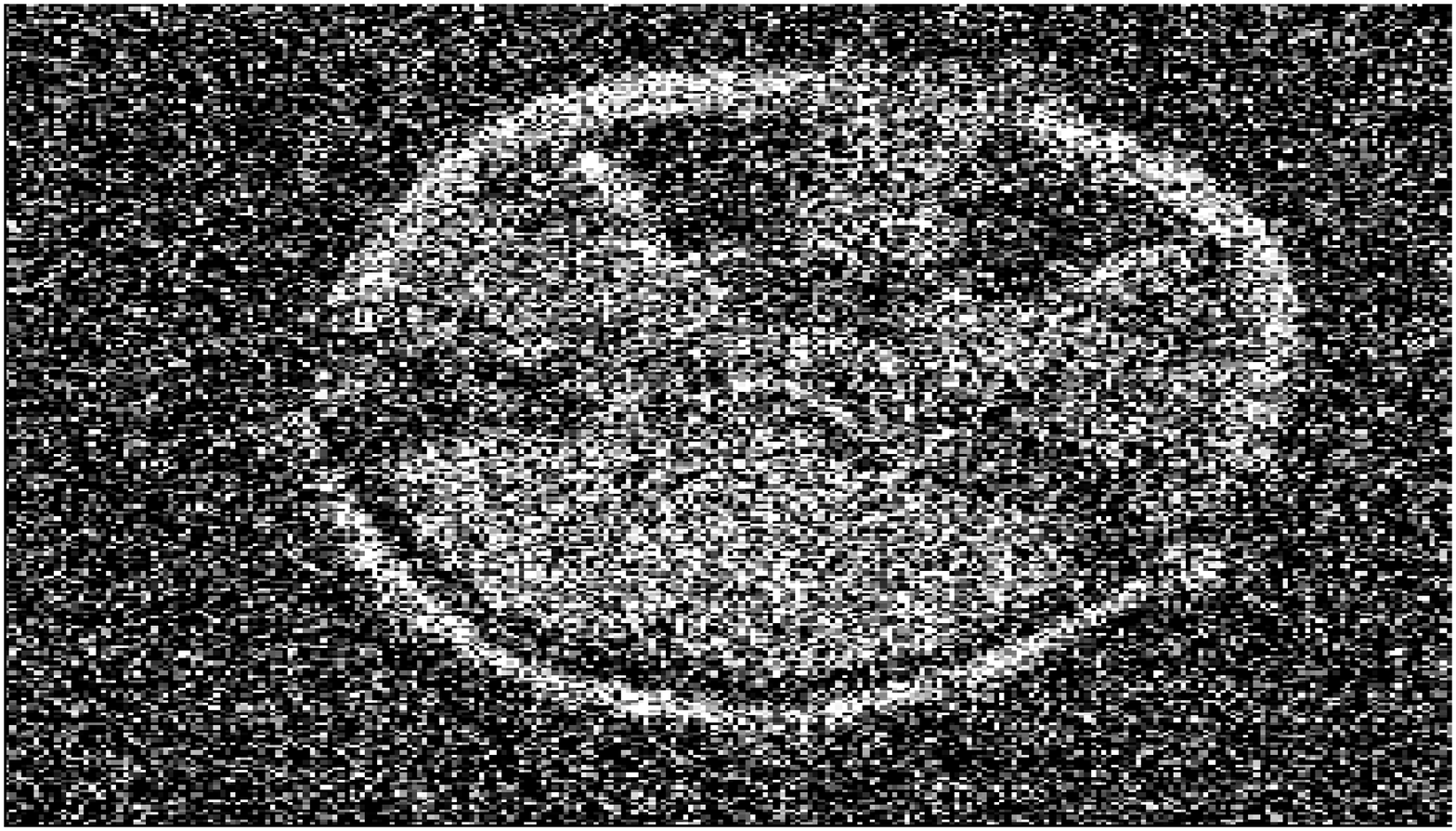}
\end{minipage}%
&
\begin{minipage}[c]{8.5cm}
\centering
\includegraphics[width=8.5cm]{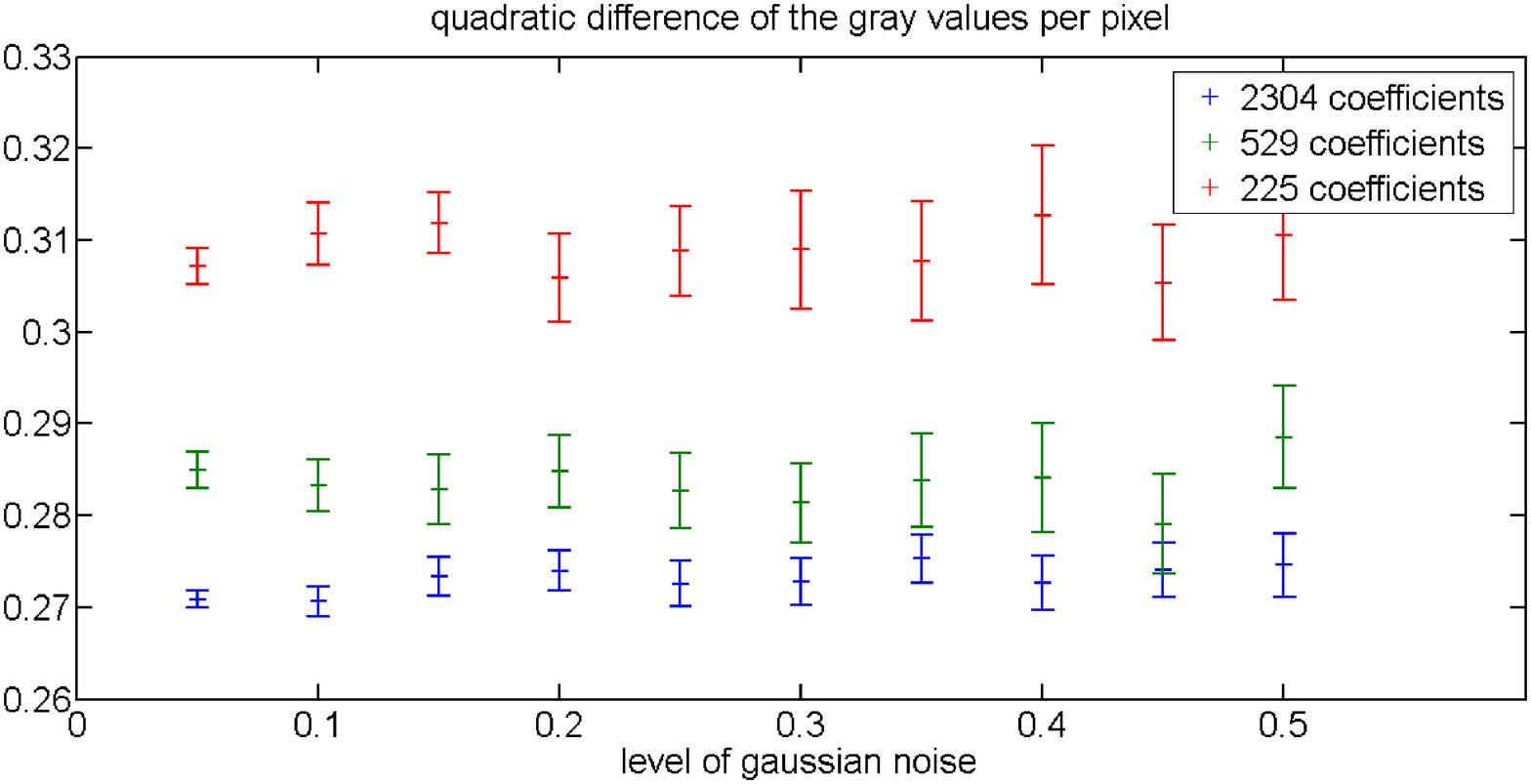}
\end{minipage}
\end{tabular}
%\end{minipage}
\caption{\footnotesize Left: Template image with $50\%$ gaussian noise. Right: For the gaussian noise from $5\%$ to $50\%$ we plot the mean and the $90\%$ confidence interval of the error (\ref{mass}), detected by the SB-algorithm.\label{Rauschen}}
\end{figure}

In the next experiment, the influence of noise to the result of the SB-Newton algorithms is studied. As before, we take the upper left image in Fig.~\ref{C_SA(2)} as the reference $f$ with $250\times 250$ pixel and gray value in the range $[0,932]$. To construct the template image $g$, we perform a rotation of the reference around its center with $11.5^{\circ}$ and add gaussian noise from $5\%$ to $50\%$. (I.e. the variance of the noise is between $0.05$ and $0.5$ after rescaling the range of the image to the interval $[0,1]$.) The left-hand side of Fig.~\ref{Rauschen} shows the template perturbed with the biggest noise level. For each noise level we project the reference and the template image to a spline function space and use the SB-algorithm to detect the transformation. For the detected and the exact transformation $M_{\mbox{\tiny detect}}$, $M_{\mbox{\tiny exact}}$, we measure the discrepancy as 
\begin{eqnarray}
 \frac{1}{250^2}\sum_i\big(f(PM_{\mbox{\tiny detect}}\tilde{x}_i)-f(PM_{\mbox{\tiny exact}}\tilde{x}_i)\big)^2,\label{mass}
\end{eqnarray}
where the sum is over all pixel of the image. This can be seen as the averaged quadratic difference of the gray value. For each noise level we consider three different cases of spline function spaces, with $15\times 15$, $23\times 23$ and $48\times 48$ coefficients. We repeat each experiment $50$ times and evaluate (\ref{mass}) for each detected transformation $M_{\mbox{\tiny detect}}$. On the right-hand side of Fig.~\ref{Rauschen} we plot the appropriate mean and the $90\%$ confidence interval. As in the previous experiments, we observe a systematical error caused by the spline-approximation. In comparison to this, the additional error caused by the gaussian noise is negligible small, even for big variances. This is not surprising, since the spline approximation of the image is performed with respect to the `sum of squared difference'-norm, known to be the unbiast estimator in the case of gaussian noise. Another result of this experiment is that the systematic error of the spline-approximation is very small (the range of the images is $[0,932]$) as we haves already pointed out in Table~\ref{C_SE(3)}. This underlines a known fact in image processing, namely that we lose less information of an image by a spline approximation than by of a standard interpolation method (cf.~\cite{640-4}). Moreover, if we want to detect the exact transformation perfectly, we have to incorporate the algorithms in a pyramidal approach in which we gradually increase the dimension of the spline function space. This procedure is quite natural and already implemented in various registration algorithms (see e.g.~\cite{Viergever-Uerbersicht} and the references therein). 

\begin{figure}[t]
\begin{minipage}[m]{1.3\textwidth}
\begin{minipage}[l]{0.25\textwidth}
\centering
\includegraphics[height=4cm]{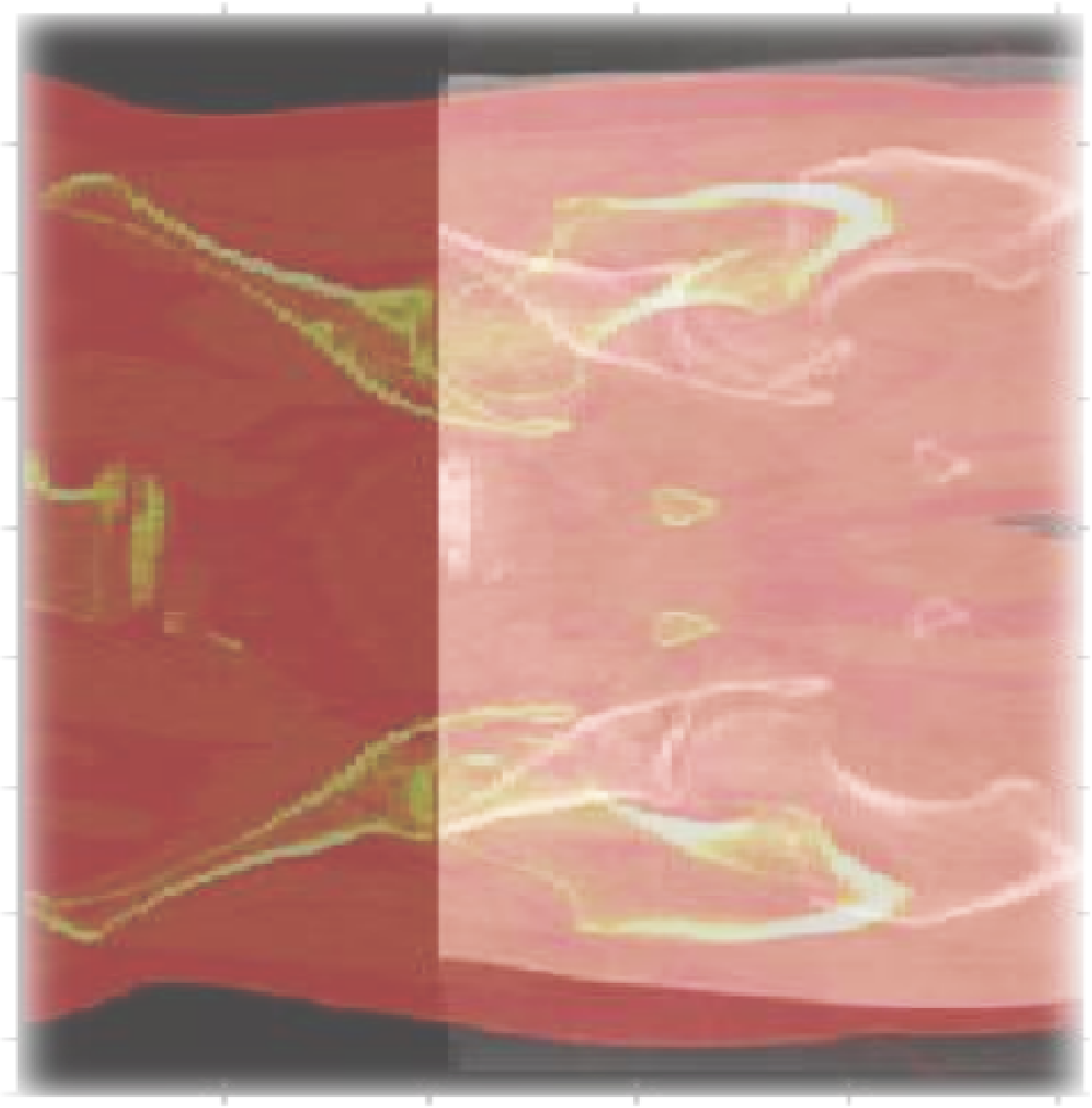}
\includegraphics[height=4cm]{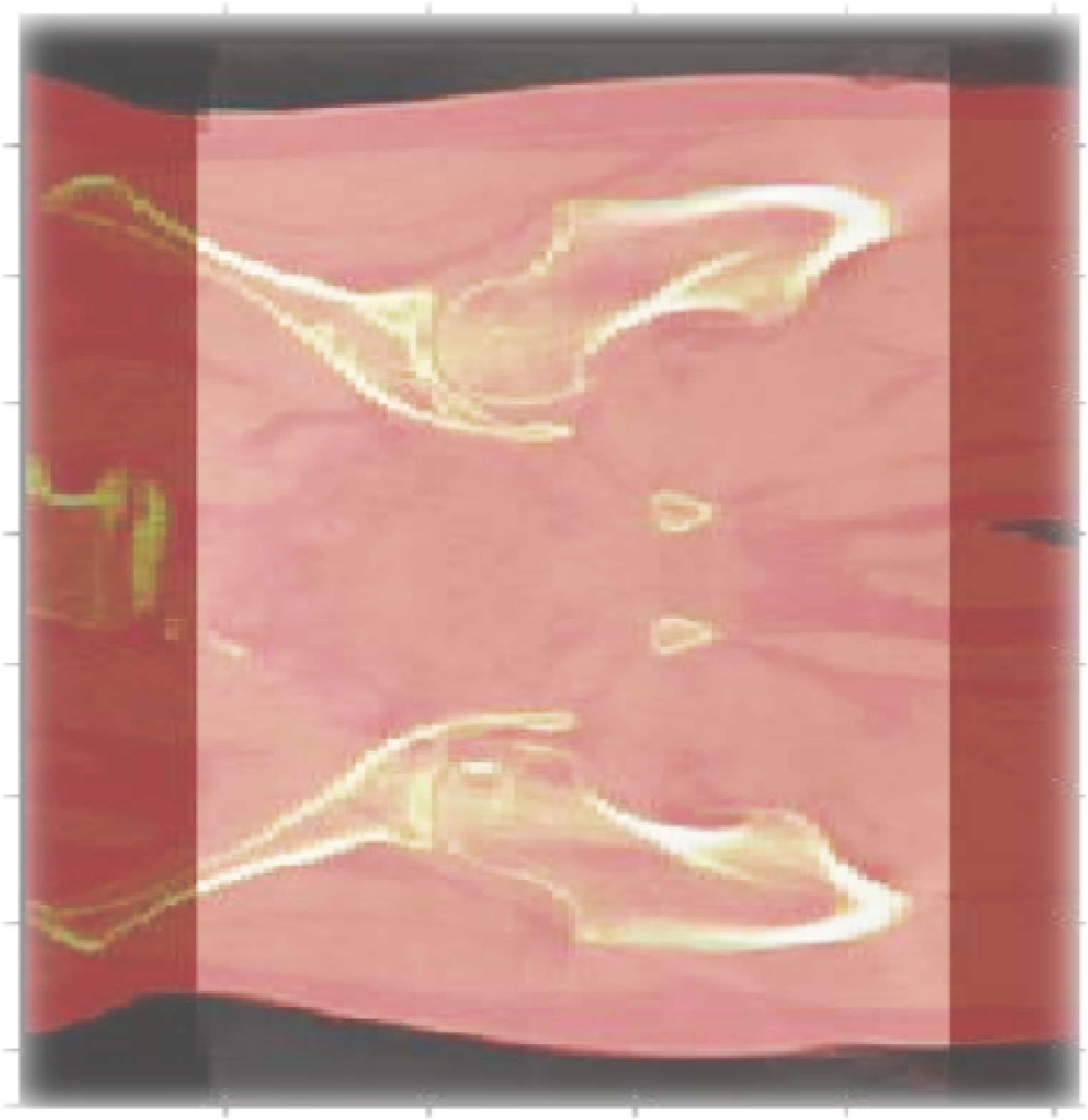}
\end{minipage}%
\begin{minipage}[m]{0.25\textwidth}
\centering
\includegraphics[height=4cm]{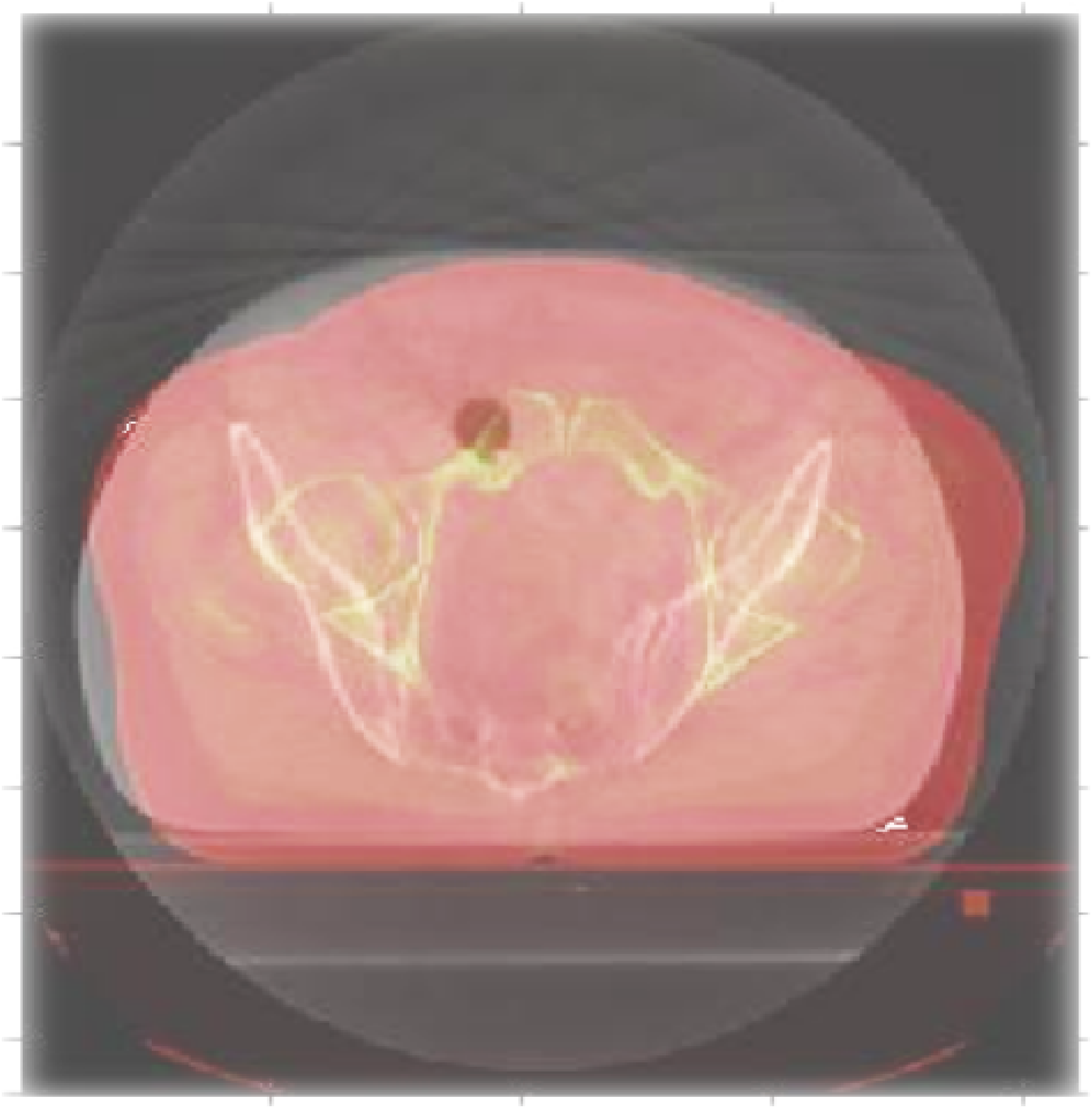}
\includegraphics[height=4cm]{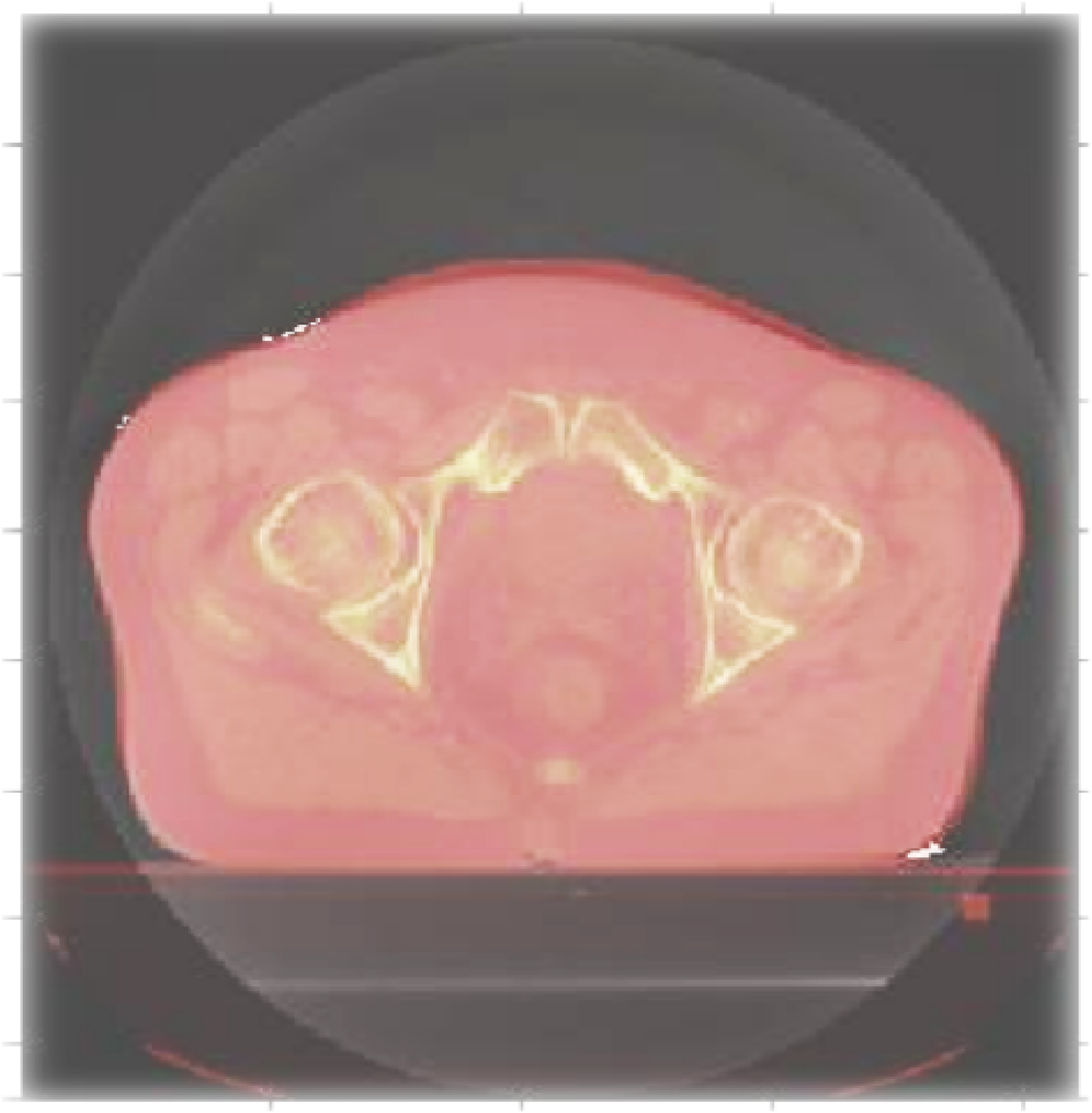}
\end{minipage}
\begin{minipage}[r]{0.25\textwidth}
\centering
\includegraphics[height=4cm]{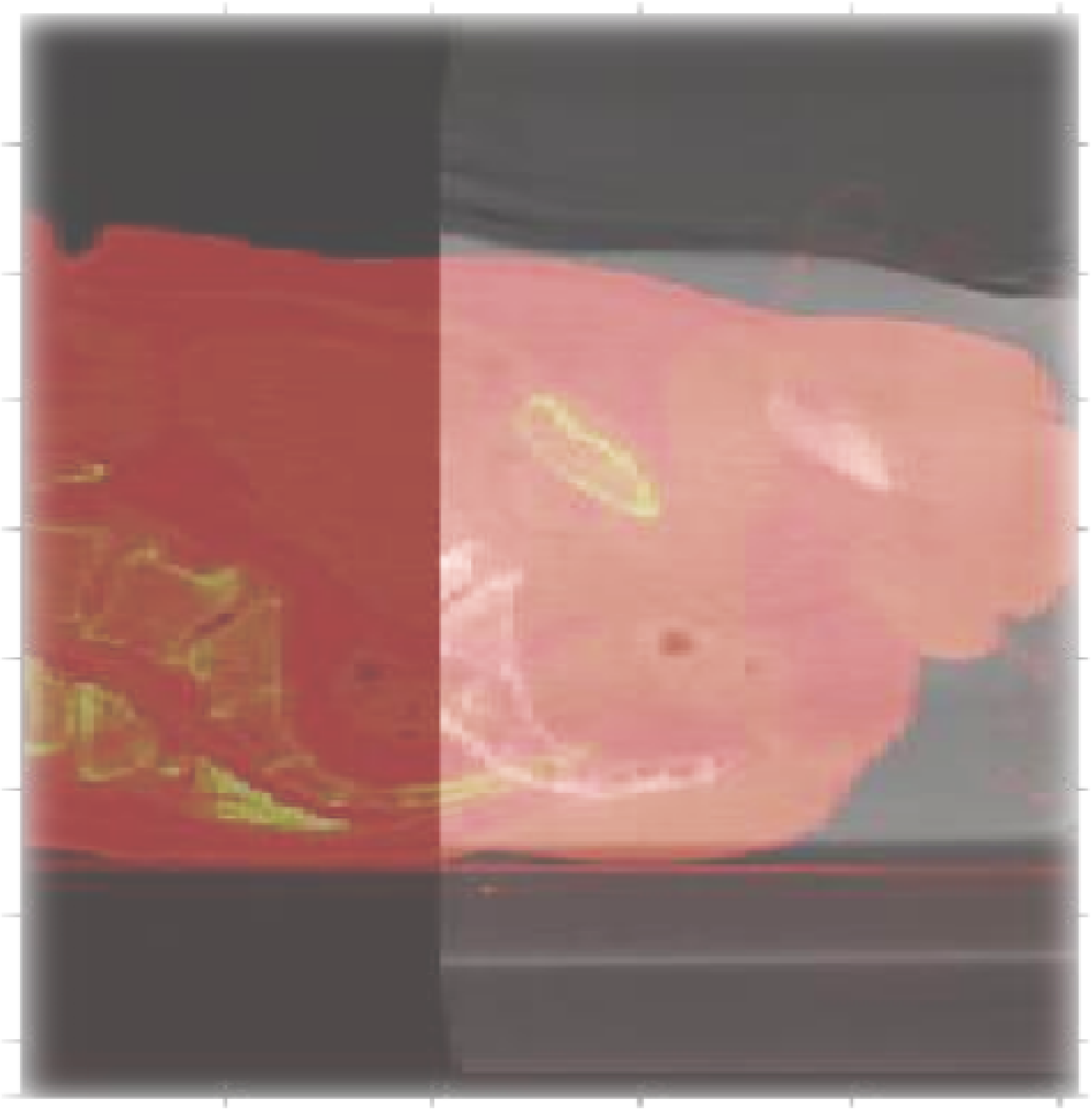}
\includegraphics[height=4cm]{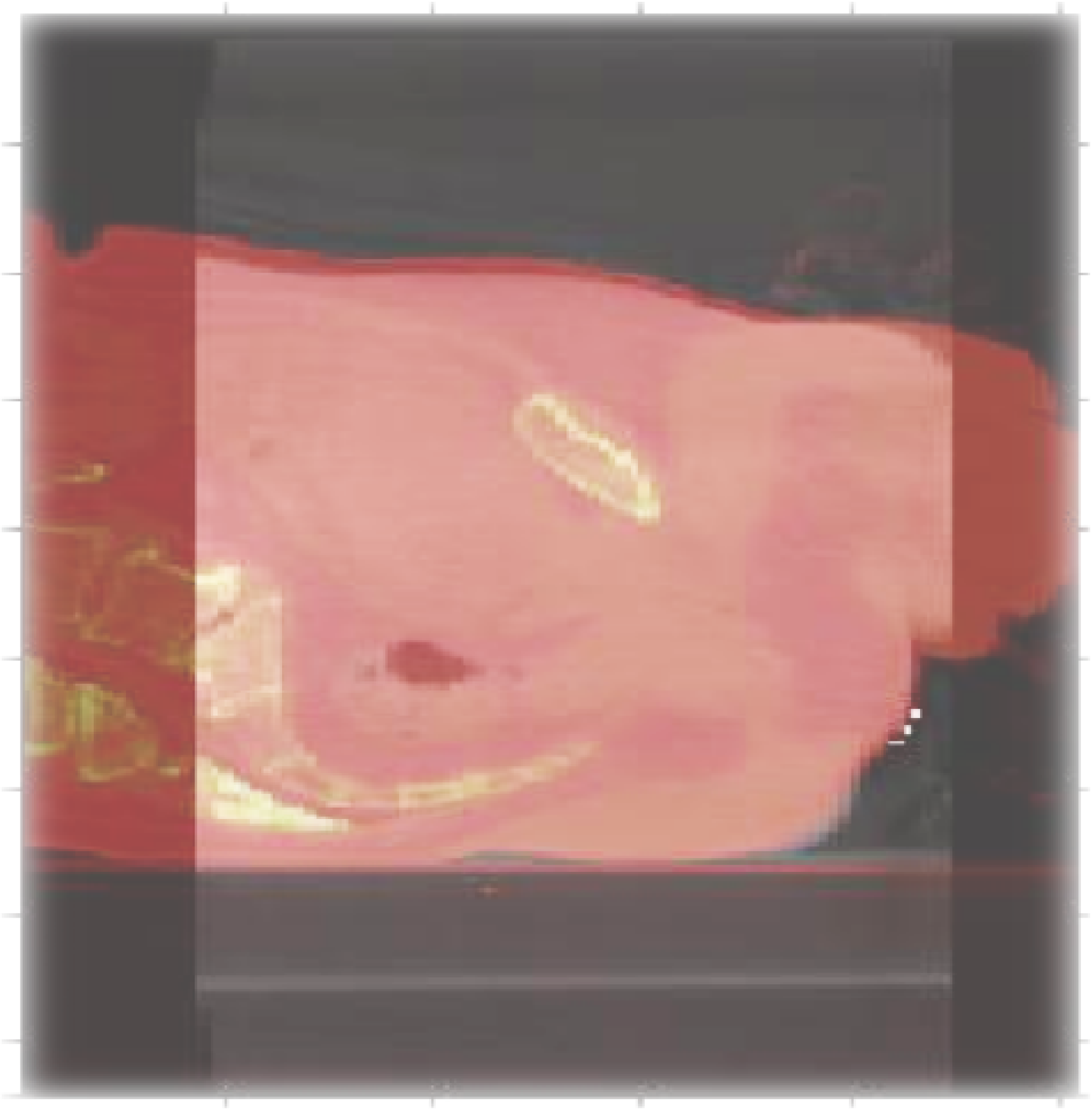}
\end{minipage}
\end{minipage}
\caption{\footnotesize The first row shows the initial position of the template with respect to the reference. The second row shows the astimated relative position after 12 quasi-Newton-Steps. We used a CBCT-image for the reference and a FBCT-image for the template.}
\label{C_CT}
\end{figure}

It turns out that comparing the numerical costs of the algorithms is a difficult task: In  \cite{Levenberg} the authors overcome the necessity of evaluating the gradient of one image in each iteration, but the numerical costs are of the order $N$, since they sum over all $N$ pixel on the image. In our QMC-approach we sum over a uniform distributed grid, which gives a certain (but hopefully negligible) error, which reduces the computational costs to $O((\log N)^2/N)$ in the case of 2D images. We want to point out that reducing the sum over all pixel to a sum over a grid with low discrepancy is also applicable for the LM-algorithms and leads to quite small approximation errors as in the QMC-Newton case. For the SB-Newton algorithm the computational costs are of the same order as the number of the spline-coefficients from the approximation. On the one hand, this number is often extremely small in comparison to the image size, since the results of the SB-Newton algorithm are more than sufficient even in a very course level of spline-approximation. On the other hand, for each spline coefficient, a series of B-splines has to be evaluated at different numbers (cf. \ref{Fkt-splin}), which is why the costs per spline-coefficient are high. Therefore, the SB-Newton algorithm becomes very slow for a fine level of spline-approximation.

Now, we demonstrate the algorithms with the help of some medically relevant pictures. In Fig 5.4, reference and template are FBCT and CBCT shots of the prostate area. These two datasets consist of $420\times 420\times 72$ and $512\times 512\times 101$ voxels respectively. We will not use a priori information about their relative positions to each other. Instead, we dispose the reference on the lower part of the template (see also the first row of Fig. \ref{C_CT}).

The beginning of the registration consists of comprising the data set to $26\times 26\times 25$ spline coefficients each. The second row of Fig. 5.4 shows the result of the registration after 12 SB-Newton steps. Any further steps only provide translations below the size of a voxel and rotations under 0.01 degree.

\begin{figure}[t]
\begin{minipage}[m]{1.3\textwidth}
\begin{minipage}[l]{0.25\textwidth}
\centering
\includegraphics[height=4cm]{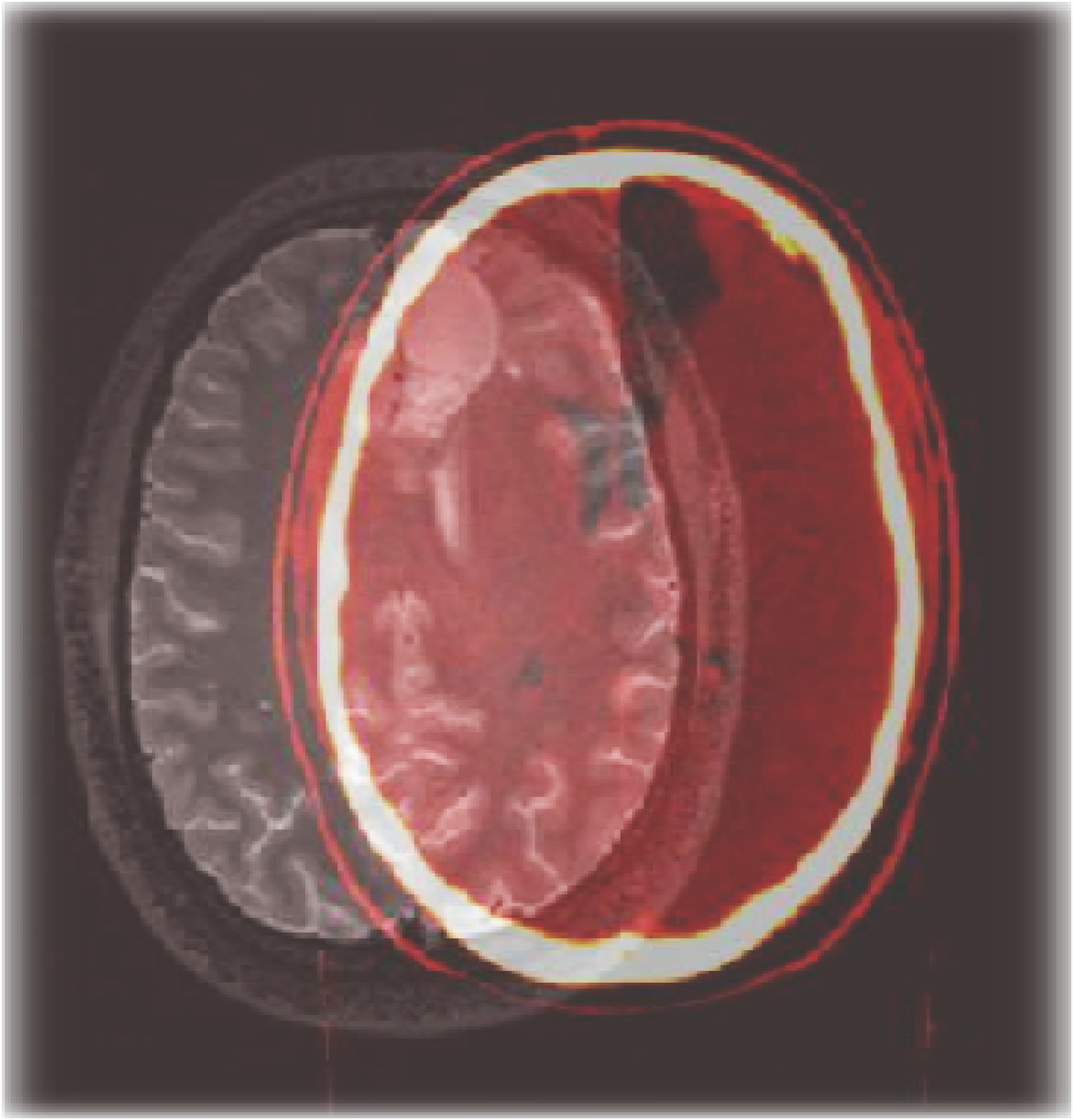}
\includegraphics[height=4cm]{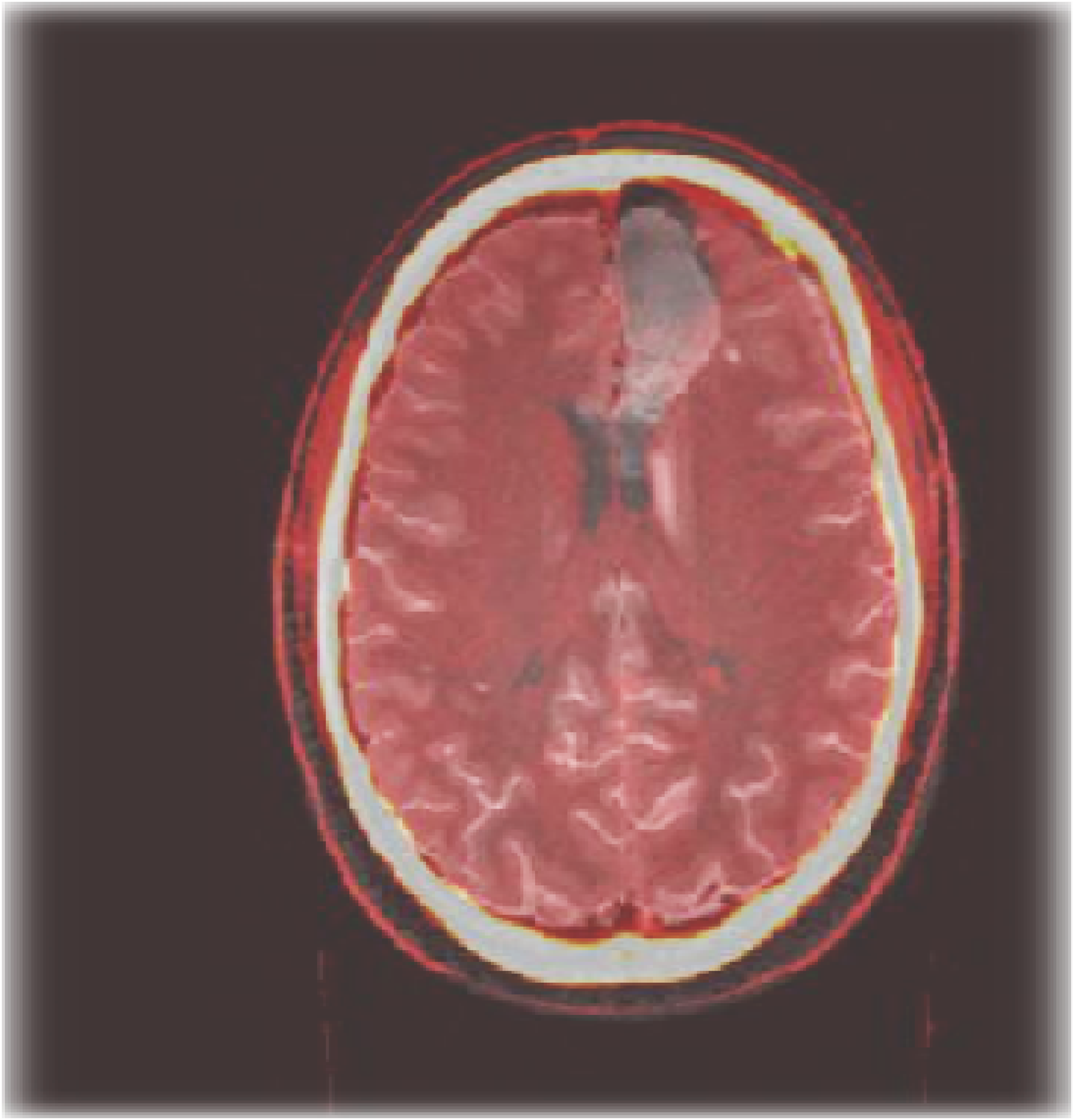}
\end{minipage}%
\begin{minipage}[m]{0.25\textwidth}
\centering
\includegraphics[height=4cm]{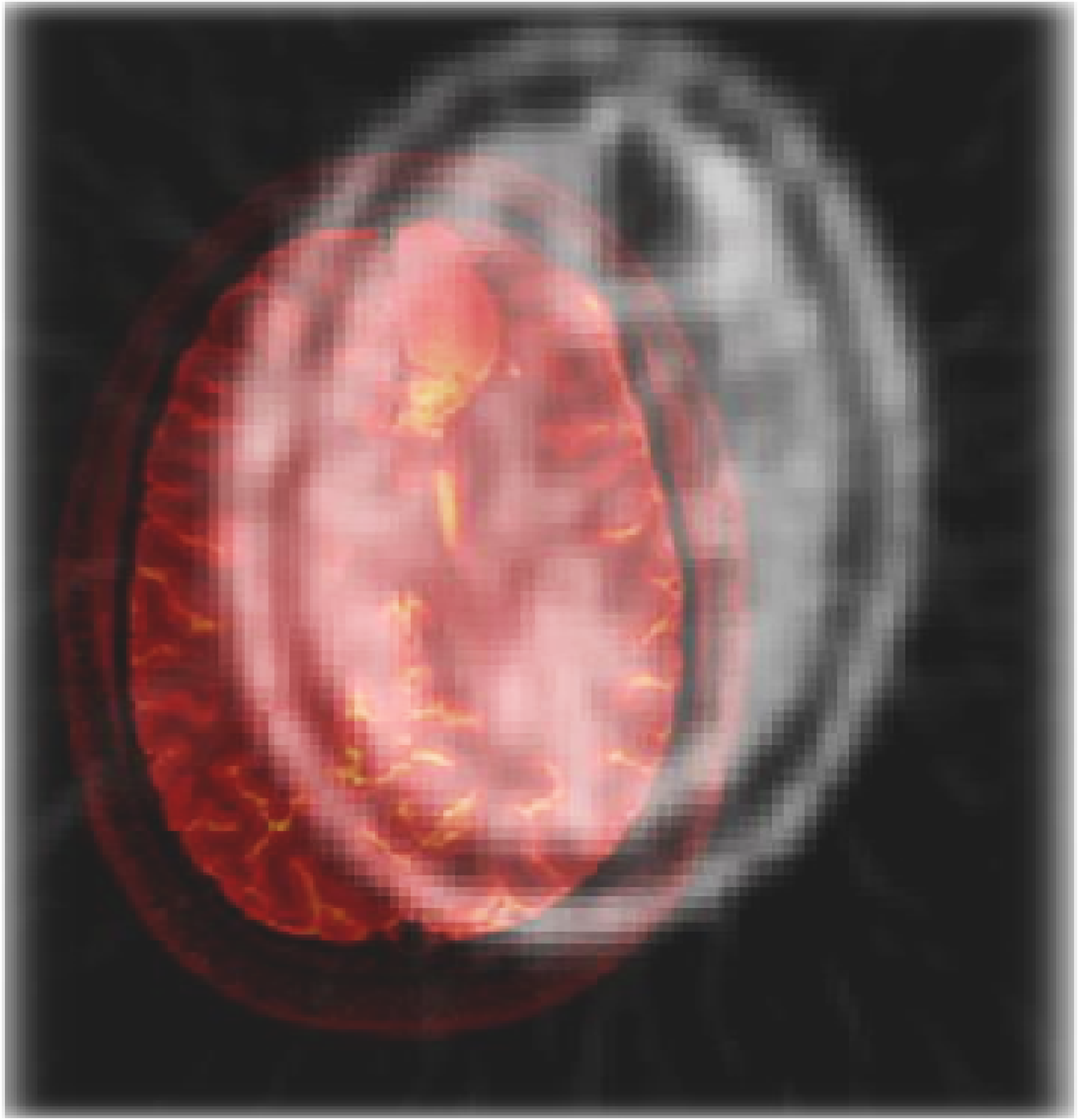}
\includegraphics[height=4cm]{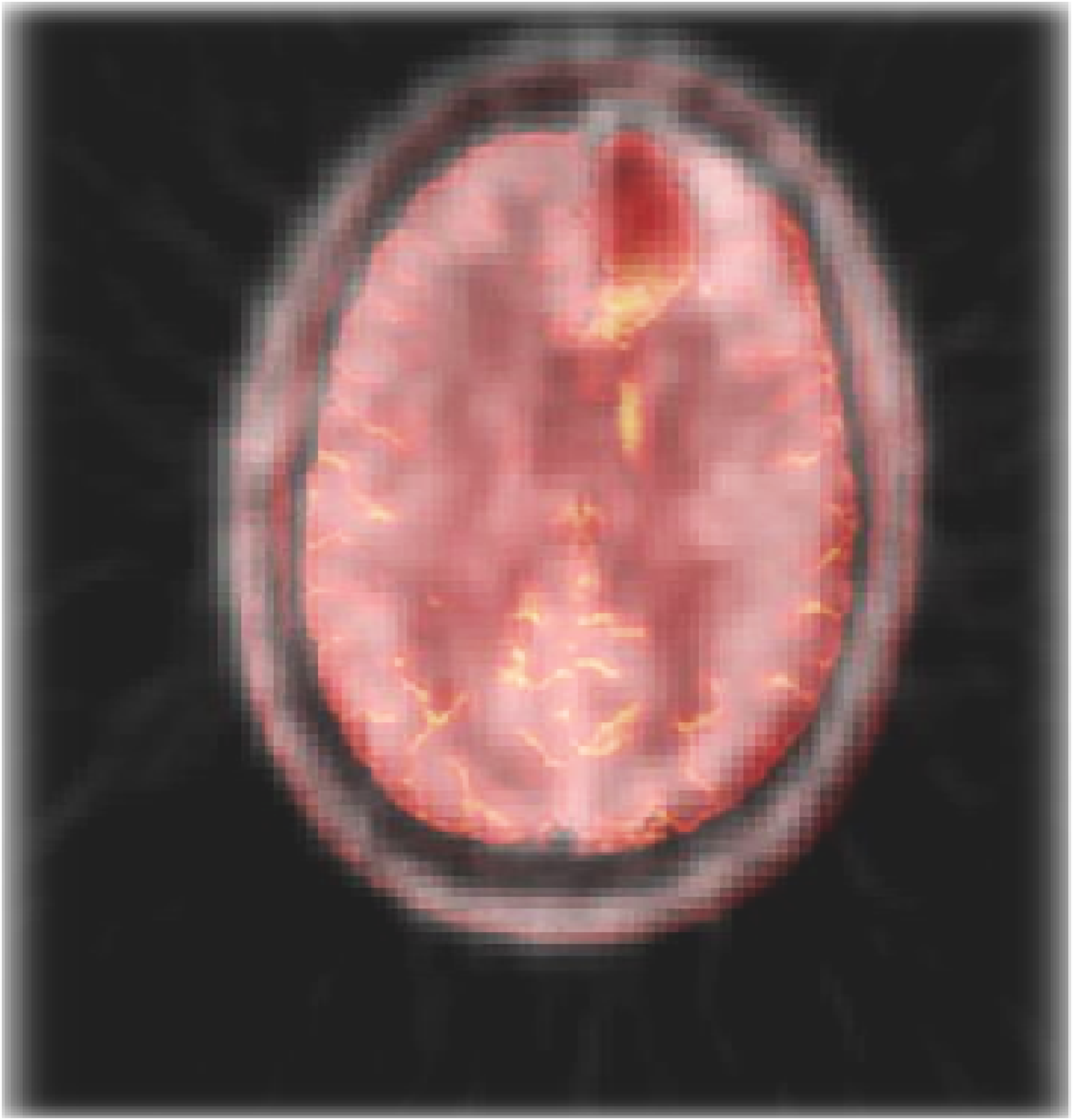}
\end{minipage}
\begin{minipage}[r]{0.25\textwidth}
\centering
\includegraphics[height=4cm]{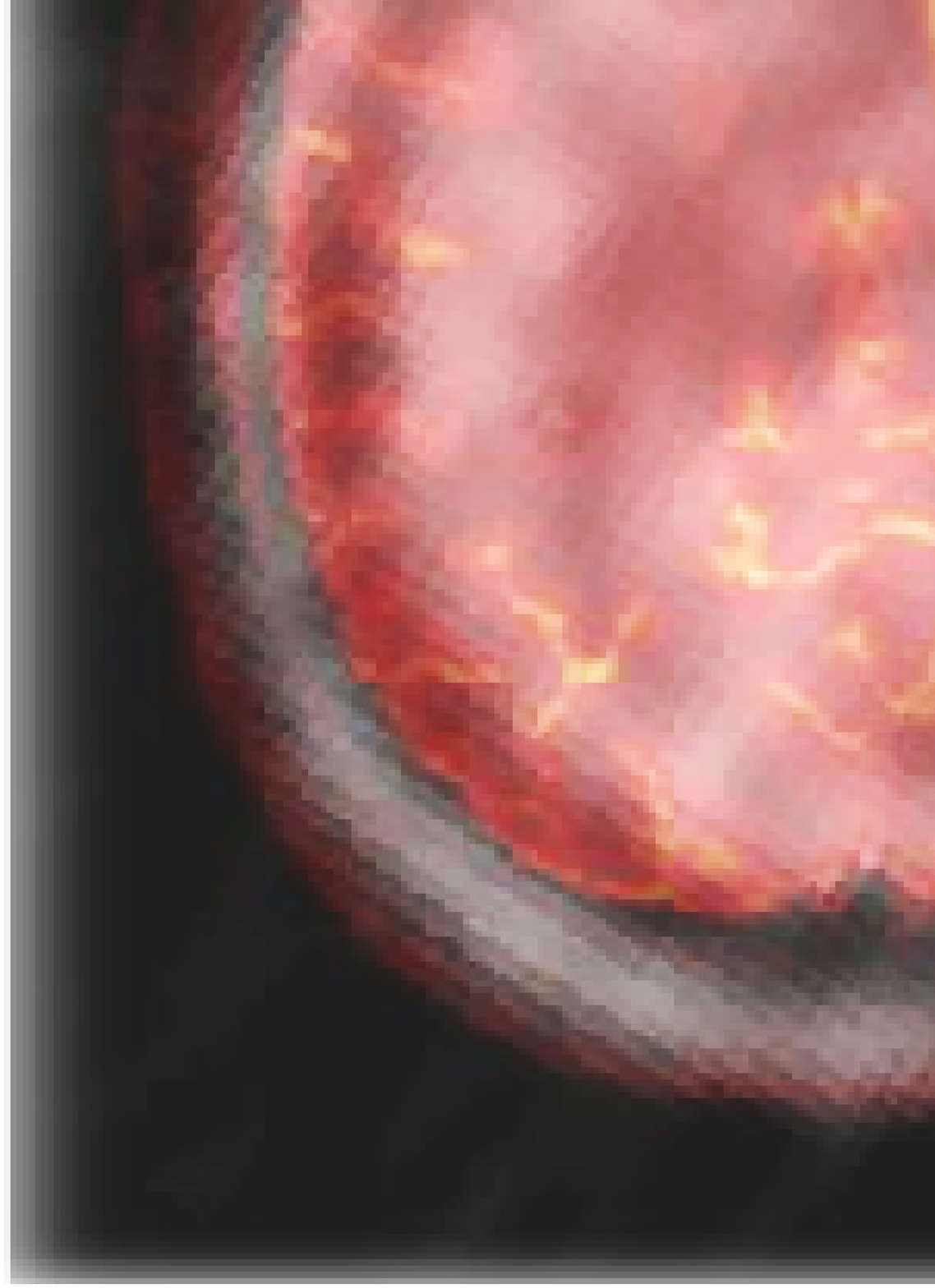}
\includegraphics[height=4cm]{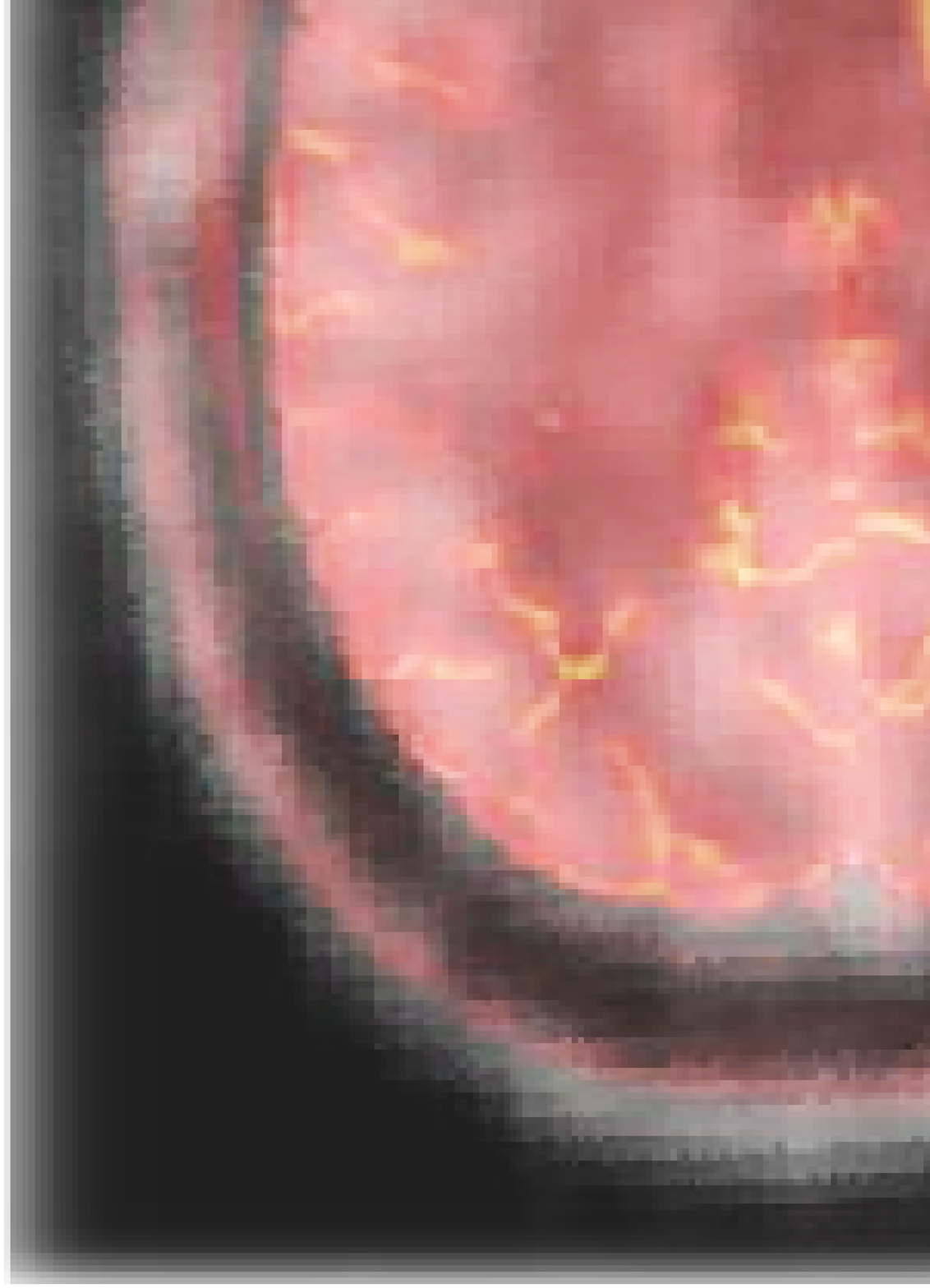}
\end{minipage}
\end{minipage}
\caption{\footnotesize The first row shows the initial situation of three registration problems. In the first one we compare a CT- with an NMR-image and in the second and third, we compare NMR- with PET-images. The second row presents the result of mutual information-based algorithm after ten quasi-Newton-steps.}
\label{C_PET}
\end{figure}

Due to the great similarity between the two recording methods the SB-Newton algorithm on SE(n) was used for the registration in the last example. With the help of the procedure described in section 4.3 one is also able to use the algorithm for data which have been recorded with very different modalities. We will show this in the next example by comparing a MR picture with a CT (each with $512\times 512$ Pixels) by using a PET- picture ($128\times 128$ Pixels) and used the ''Mutual Information-based Registration Algorithm on SE(n)``. The first row of Fig. \ref{C_PET} shows the origin of the correlating registration problems, and the second one reveals the result after 10 steps. We used a compression of  $64\times 64$ spline coefficients for each case mentioned. In the last column of Fig. \ref{C_PET} we more-over add an additional rotation of 30 degrees. Once more, we have reached a good match of the pictures within 10 steps. We already pointed out that this method only provides a local convergence. If the pictures differ in their origin - the algorithm won’t provide an acceptable result – even if tried using many steps. This happens if we rotate the PET- picture – as seen on third column in Fig. \ref{C_PET} –  more than 30 degrees.

\section{Conclusions}
\begin{sloppypar}
We developed a novel framework for the rigid or volume-preserving registration of two real valued functions. For this we used a modified $(\mu,\nu)-$Newton algorithm to solve the corresponding optimization problem on the manifolds $SE(n)$ or $SA(n)$. The local parameterizations of these manifolds are chosen in such a way to get a very efficient and easily implementable algorithm. Additionally, we proved the local quadratic convergence of this method under suitable generic conditions.
\end{sloppypar}

In order to apply this framework to the image registration task, we offered two strategies. The QMC-Newton compares two images in a sequence of points with low discrepancy. The appearing cost function could then be easily approximated by the Quasi Monte Carlo method. The SB-Newton strategy uses B-spline approximation of the images. Here, no image evaluations are necessary, the algorithms operate directly on the compressed (jpeg-like) data. Our numerical tests showed that both strategies preserve a high accuracy even in the case of high compressed data. Additionally, we confirmed the local quadratic convergence of both methods in numerical experiments. Comparatively, the QMC-Newton step is done in less computational time than a SB-Newton step - at least in our implementation. But it turned out that the second method has a higher accuracy in detecting the requested transformation.

\section{Appendix: Generic convergence conditions}
In Theorem 3.4 we proved the local quadratic convergence of the QMC-algorithm on the condition that the critical points of the cost functions are nondegenerate. In this appendix we will show that this condition is generically fulfilled in image registration,  i.e. on very mild conditions for the reference image $g\in\C^3(\R^n,\R)$, the set of all template images $f\in C^3(\R^n,\R)$ for which the algorithm converges locally quadratically is open and dense in $C^3(\R^n,\R)$ in terms of the strong topology\footnote{For a definition of the weak and strong topology of function spaces we refer to\cite{Hirsch}.}.

To begin with, we have to introduce some aspects of the transversality theory.
We refer to \cite{Hirsch} for a more detailed discussion. Let $M,N$ be manifolds and $A\subset N$ a submanifold of $N$. A differential map $f:M\rightarrow N$ is \textit{transverse} to $A$ and one writes $f\pitchfork A$ if
\begin{eqnarray}
 A_y+T_xf(M_x)=N_y\nonumber
\end{eqnarray}
whenever $f(x)=y\in A$. That is, the tangent space of $N$ at $y$ is spanned by the
tangent space of $A$ at $y$ and the tangent space of $M$ at $x$.  The following two theorems are well known in  transversal theory, their proofs can be found e.g. in \cite{Hirsch}.

\begin{theorem}
 Let $f:M\rightarrow N$ be a $C^r$ map, $r\geqslant 1$ and $A\subset N$ a $C^r$ submanifold. If $ f$ is transverse to A then $f^{-1}(A)$ is a submanifold of $M$. The codimension of $f^{-1}(A)$ is the same as the codimension of $A$ in $N$.
\end{theorem}

\begin{theorem}[Parametric Transversality]
 Let $V,M,N$ be $C^r$ manifolds without boundary and $A\subset N$ is a $C^r$ submanifold. Let $F:V\rightarrow C^r(M,N)$ satisfy the following conditions:
\begin{itemize}
 \item[(a)] the evaluation map $F^{ev}:V\times M \rightarrow N,(v,x)\mapsto F_v(x)$ is $C^r$.
\item[(b)] $F^{ev}$ is transverse to $A$.
\item[(c)] $r>\mbox{max} \left\lbrace 0, \mbox{dim }N +\mbox{dim }A -\mbox{dim }M \right\rbrace $. 
\end{itemize}
Then the set 
$$\pitchfork (F;A):=\left\lbrace v\in V~|~F_v\pitchfork A\right\rbrace $$
is residual and therefore dense. If $A$ is closed in $N$ and $F$ is continuous for the strong topology on $C^r(M,N)$, then $\pitchfork (F,A)$ is also open.
\end{theorem}

To give the degenerated critical points a geometric interpretation, we also have
to introduce the manifold $J^r(M,N)$ of $r$-jets of functions from $M$ to $N$. A $r$-jet from $M$ to $N$ is an equivalence-class $[x,f,U]_r$  of a triple in which $U\subset M$ is an open subset, $x\in M$ and $f:U\to N$ is an $C^r$ map. We say that two triples $[x,f,U]_r$ and $[x',f',U']_r$ are equivalent if $x=x'$, and $f$ and $f'$ have same derivatives in $x$ up to the order $r$. We use the notation $j_x^rf:=[x,f,U]_r$ to denote the $r$-jet of $f$ in $x$ and $j^rf:M\to J(M,n)$ defines the $r$-prolongation map $x\mapsto j^r_xf$. The following theorem
can be found in \cite{Hirsch}.

\begin{theorem}[Jet Transversality Theorem]
 Let $M$, $N$ be $C^{\infty}$ manifolds without boundary, and let $A\subset J^r(M,N)$ be a $C^{\infty}$ submanifold. Suppose $1\leqslant r<s\leqslant\infty$. Then $$\pitchfork^s(M,N;j^r,A):=\left\lbrace f\in C^s(M,N)~\big|~j^rf\pitchfork A \right\rbrace $$ is residual and thus dense in $C_S^s(M,N)$, and open if $A$ is closed.
\end{theorem}

Let us now come back to the Lie Groups $G=SA(n)$ and $G=SE(n)$ respectively. In the manifold $J^2(G,\R)$ the subset
\begin{eqnarray}
 \frak{U}=\left\lbrace j^2_xh~\big|~h\in C^2(G,\R),~x\in G,~\nabla h(x)=0,~\det \mbox{Hess}_h(x)=0  \right\rbrace \nonumber
\end{eqnarray}
contains all degenerated critical points. Furthermore, the set $\frak{U}$ is a finite union of submanifolds $\frak{U}=\frak{U}_1\cup\ldots\cup\frak{U}_m$ with
\begin{eqnarray}
 \dim\frak{U}_1\leqslant\ldots\leqslant\dim\frak{U}_m=\frac{3}{2}\dim G+\frac{1}{2}(\dim G)^2.\nonumber
\end{eqnarray}
Since we want to study the influence of the images $f,~g$ on the cost functions $\Psi$ and $\Phi$ of formula (\ref{Zielfunktion}) and (\ref{appsum}), we introduce the linear maps $\hat{\Psi},\hat{\Phi}:C^3(\R^n,\R)\to C^3(G,\R)$ which are defined by
\begin{align}
 \hat{\Psi}(f)(A,t)&:=\frac{1}{N}\sum\limits_{i=1}^Ng(x_i)f(Ax_i+t)\nonumber\\
\hat{\Phi}(f)(A,t)&:=\int_{\R^n}g(x)f(Ax+t)dx\nonumber.
\end{align}
For a given function $f\in C^3(\R^n,\R)$ the cost function $\hat{\Psi}(f)$ has no degenerated critical points if, and only if $j^2\hat{\Psi}(f)$ misses $\frak{U}$. Since 
\begin{align}
\dim J^2(G,\R)=\frac{5}{2}\dim G+\frac{1}{2}(\dim G)^2+1 \nonumber
\intertext{we get}
\dim G + \dim \frak{U}_i<\dim J^2(G,\R),~~~i=1,\ldots,m.\nonumber
\end{align}
By means of Theorem 7.1 we get that  $j^2\hat{\Psi}(f)$ misses $\frak{U}$ if, and only if  $j^2\hat{\Psi}(f)$ is transverse to every submanifold $\frak{U}_i$ $i=1,\ldots,m$. The same statement applies to the function $\hat{\Phi}$.

 \begin{theorem}
Suppose that the template image $g\in C^3(\R^n,\R)$ has compact support
and satisfy the following conditions: 
\begin{itemize}
 \item [a)] In the optimization problem (\ref{Zielfunktion}), $g$ is not identical to zero.
% \item [b)] $g(x_i)\neq 0$ for at least $k=\frac{1}{2}(n+1)(n+2)$ elements of the sequence $\left\lbrace x_i\right\rbrace _{i=1}^N$ in the optimization problem (\ref{appsum}), such that these $k$ elements do not lie on a quadric hypersurface.
\item [b)] In the optimization problem (\ref{appsum}), there exists $k=\frac{1}{2}(n+1)(n+2)$ elements of the sequence $\left\lbrace x_i\right\rbrace _{i=1}^N$  which do not do not lie on a quadric hypersurface and $g(x_i)\neq 0$ for all these $k$  elements.
\end{itemize}
Then the cost functions $\Phi$ and $\Psi$ have no degenerate critical points, for a generic set of reference images $f\in C^3_S(\R^n,\R)$.
\end{theorem}

\begin{proof}
 Following the argumentation above, we define
\begin{eqnarray}
 A_{\Psi,i}:=\pitchfork^3(\R^n,\R;j^2\hat{\Psi},\frak{U}_i):=\left\lbrace f\in C^3(\R^n,\R)~\big|~j^2\hat{\Psi}(f)\pitchfork\frak{U}_i \right\rbrace \nonumber
\end{eqnarray}
and $A_{\Phi,i}:=\pitchfork^3(\R^n,\R;j^2\hat{\Phi},\frak{U}_i)$ respectively. We have to show that $A_{\Psi,i}$ and $A_{\Phi,i}$ are open and dense in $C^3_S(\R^n,\R)$.

Using the Jet Transversality Theorem provides for  the openess of $\pitchfork^3(G,\R;j^2,\frak{U}_i)$. Since $\hat{\Psi}$ and $\hat{\Phi}$ are continuous maps, we conclude that the sets $A_{\Psi,i}$ and $A_{\Phi,i}$ are open. Hence, we only have to show the densness of both sets.
Therefore, we define $$A_{\Psi,i,r}:=\pitchfork^3(\R^n,\R;j^2\hat{\Psi}|_{K_r(0)},\frak{U}_i)\mbox{~~~and~~~} A_{\Phi,i,r}:=\pitchfork^3(\R^n,\R;j^2\hat{\Phi}|_{K_r(0)},\frak{U}_i)$$ respectively, while $K_r(0)$ is an open ball with radius $r$ in $G$ with respect to the Frobenius norm. Again, we can use the Jet Transversality Theorem and the continuity of $\hat{\Phi}$ and $\hat{\Psi}$ to show that $A_{\Phi,i,r}$ and $A_{\Psi,i,r}$ are  countable intersections of open sets. Due to the
category theorem of Baire and the fact that $A_{\Psi,i}=\bigcap_{r=1}^{\infty}A_{\Psi,i,r}$, it is enough to show the denseness of $A_{\Psi,i,r}$ with respect to the strong topology. An analogous argumentation for $\hat{\Phi}$ shows that the denseness of $A_{\Phi,i,r}$ is sufficient to complete the proof.

Now, let $K,L\subset\R^n$ be two compact subset with 
\begin{eqnarray}
 \left\lbrace y\in\R^n~\big|~y=Ax+t,~x\in \mbox{supp}(g),~(A,t)\in K_r(0)\cap G\right\rbrace \subset K~~~\mbox{and}~~~K\subset\mathring{L} \nonumber
\end{eqnarray}
and $f\in C^3(\R^n,R)$. Suppose there is a sequence $\left\lbrace f_k\right\rbrace_{k\in\N}\subset A_{\Psi,i,r}$ which converges to $f$ with respect to the weak topologie. Then we can construct a sequence $\left\lbrace h_k\right\rbrace_{k\in\N}\subset A_{\Psi,i,r}$ with $h_k(x)=f_k(x)$ for all $x\in K$ and $h_k(x)=f(x)$ for all $x\in \R^n\setminus L$ such that $h_k\to f$ with respect to the strong topology. Since $f_k(x)=h_k(x)$ for all $x\in K$ implies $h_k\in A_{\Psi,i,r}$ we get that the denseness of $A_{\Psi,i,r}$ with respect to the weak topology implies the denseness with respect to the strong topology. Since the same statement is true for the set $A_{\Phi,i,r}$ we only have to show that both sets are dense in $C^3_W(\R^n,\R)$.

% and $f\in A_{\Psi,i,r}$.  Then we get $h\in A_{\Psi,i,r}$ for every $h\in C^3(\R^n,\R)$ which fulfils $f(x)=h(x)$ for all $x\in K$. Hence, the denseness of $A_{\Psi,i,r}$ with respect to the weak topology implies the denseness with respect to the strong topology. Since the same statement is true for the set $A_{\Phi,i,r}$ we only have to show that both sets are dense in $C^3_W(\R^n,\R)$.

Now, let us complete the proof for the function $\hat{\Psi}$. Consider the vector space $V=\R\times\R^n\times\mbox{Sym}(n)$, endowed with the standard scalar-product $\left\langle ~.~\right\rangle :V\times V\to \R,$
\begin{eqnarray}
 \left\langle (\alpha,a,A),(\beta,b,B)\right\rangle =\alpha\beta+a^{\top}b+tr(AB)\nonumber
\end{eqnarray}
and the map 
\begin{eqnarray}
 \varphi:\R^n\to V~,~~~x\mapsto(1,x,xx^{\top}).\nonumber
\end{eqnarray}
Since a quadric hypersurface $Q$ is defined via 
\begin{eqnarray}
Q_{\alpha,a,A}=\left\lbrace
x\in\R^n~|~x^{\top}Ax+a^{\top}x+\alpha=0 \right\rbrace\nonumber   \end{eqnarray}
 for an arbitrary element $(\alpha,a,A)\in V\setminus \left\lbrace  0\right\rbrace $ we get the equivalence
\begin{eqnarray}
x\in Q_{\alpha,a,A}~\Leftrightarrow~\left\langle (\alpha,a,A),\varphi(x)\right\rangle=0. \nonumber
\end{eqnarray}
Hence, using the condition that $k=\frac{1}{2}(n+1)(n+2)$ elements of the sequence $\left\lbrace x_i\right\rbrace_{i=1}^{N}$ do not lie on a quadric hypersurface and that $g(x_i)\neq 0$ for those elements, we get 
% With the natural identification $\mbox{Sym(n)}=\R^{\frac{1}{2}n(n+1)}$ the map $\varphi$ is a polynomial in each component. Hence, the map $\chi:\R^{n\times k}\to\R$, $k=\frac{1}{2}n(n+1)+n+1$ defined by
% \begin{eqnarray}
%  \chi:(x_1,\ldots,x_k)\mapsto\det\left(\varphi(x_1),\ldots,\varphi(x_n) \right)\nonumber 
% \end{eqnarray}
% is a polynomial in $n\cdot k$ variables. Therefore, the zero set of $\chi$ is an affine variety and the set $\R^{n\times k}\setminus \chi^{-}(0)$ is dense and open in $\R^{n\times k}$. Furthermore, the set of all tupels $(x_1,\ldots,x_N)$, $N\geqslant k$ which fulfils
\begin{eqnarray}
 \mbox{span}\left\lbrace g(x_1)\varphi(x_1),\ldots,g(x_N)\varphi(x_N)\right\rbrace =V.\nonumber
\end{eqnarray}
 and a short calculation shows that also 
\begin{eqnarray}
 \mbox{span}\left\lbrace g(x_1)\varphi(Ax_1+t),\ldots,g(x_N)\varphi(Ax_N+t)\right\rbrace =V.\nonumber
\end{eqnarray}
is valid for each $(A,t)\in G$.
% According to the conditions of the template image $g$, we can assume that $g(x_i)\neq 0$ for all $i=1,\ldots,N$ without loss of generality. 
Hence, the linear map generated by the matrix
\begin{eqnarray}
 M=\Big( g(x_1)\varphi(Ax_1+t),\ldots,g(x_N)\varphi(Ax_N+t)\Big) \nonumber
\end{eqnarray}
is surjective for each $(A,t)\in G$. 

Now, consider formula (\ref{disccoefs}) as a linear map of the form 
\begin{eqnarray}
 (j^2_{Ax_1+t}f,\ldots,j^2_{Ax_N+t}f)\mapsto(\alpha,\beta,\gamma,\epsilon,\delta).\label{formel_dif}
\end{eqnarray}
Let $m_1,\ldots,m_k$ denote the rows of $M$, i.e. $M^{\top}=(m_1^{\top},\ldots,m_k^{\top})$. Then, the $m_1,\ldots,m_k$ can be used to built the rows of the representation matrix of (\ref{formel_dif}) in the following way:
\begin{align}
 &\alpha_i=m_1\cdot (\frac{\partial f}{\partial x_i}(Ax_1+t),\ldots,\frac{\partial f}{\partial x_i}(Ax_N+t))^{\top},\nonumber\\
&\beta_{i,1}=m_2\cdot (\frac{\partial f}{\partial x_i}(Ax_1+t),\ldots,\frac{\partial f}{\partial x_i}(Ax_N+t))^{\top}\nonumber\\
&~~~~~~~~~~~~~~~~~~~~~~~\vdots\nonumber\\
&\delta_{n,n,i,j}=m_k\cdot (\frac{\partial^2 f}{\partial x_i\partial x_j}(Ax_1+t),\ldots,\frac{\partial f}{\partial x_i\partial x_j}(Ax_N+t))^{\top}\nonumber
\end{align}
 Therefore, the map (\ref{formel_dif}) is surjective and with the formulas (\ref{begin_newton_SE(n)}) and (\ref{end_newton_SE(n)}) for $G=SE(n)$, and (\ref{f0})-(\ref{f6}) for $G=SA(n)$ respectively, the map
\begin{eqnarray}
 (j^2_{Ax_1+t}f,\ldots,j^2_{Ax_N+t}f)\mapsto j^2_{(A,t)}\hat{\Psi}(f).\label{form_dif2}
\end{eqnarray}
is surjective for each chosen $(A,t)\in G\cap K_r(0)$. Moreover, the map (\ref{form_dif2}) is transverse to every submanifold of $J^2(G,\R)$. 

Now, let $f\in C^3(\R^n,\R)$ be arbitrarily chosen. To show the denseness of $A_{\Psi,i,r}$ in terms of the weak topology, it is enough to show that $A_{\Psi,i,r}\cap [f+P^{2N}(n,1)]$ is dense. Here, $P^{2N}(n,1)$ denotes the set of all polynomials $\R^n\to\R$ with a degree smaller than or equal to $2N$. In the case of the map 
\begin{eqnarray}
 F^{\mbox{ev}}:[f+P^{2N}(n,1)]\times G\to J^2(G,\R)\nonumber\\
(f+p,A,t)\mapsto j^2_{(A,t)}\hat{\Psi}(f+p)\nonumber
\end{eqnarray}
we already showed the surjectivity if $(A,t)\in G$ is fixed. Hence, using the Parametric Transversality Theorem we get that $F^{\mbox{ev}}(f+p,~.~)$ is transverse to $\frak{U_i}$ for a dense subset of $P^{2N}(n,1)$. Therefore, $A_{\Psi,i,r}$ is dense in $C^3_W(\R^n,\R)$ which completes the proof for the cost function $\Psi$.

Let us now consider the cost function $\Phi$. Following the argumentation above,
it is enough to prove that $A_{\Phi,i,r}$ is dense in $C^3_W(\R^n,\R)$. Due to the conditions for the template image $g$, the set of functions 
\begin{align}
 \mathcal{H}_{A,t}=
\left\lbrace \frac{\partial}{\partial x_i}g(x),
\frac{\partial}{\partial x_i}\big(g(x)(Ax+t)_k\big),
\frac{\partial^2}{\partial x_i \partial x_j}g(x), 
\frac{\partial^2}{\partial x_i \partial x_j}\big(g(x)(Ax+t)_k\big),\right. \nonumber\\
\left.  
\frac{\partial^2}{\partial x_i \partial x_j}\big(g(x)(Ax+t)_k (Ax+t)_l\big)~~~\Big|~~~ i,j,k,l=1,\ldots,n,~j\leqslant i,~l\leqslant k 
\right\rbrace \nonumber
\end{align}
 is linear independent. Otherwise, $g$ would fullfil a partial differential equation, and, since $g$ has compact support, this would imply that $g\equiv 0$, which contradicts the requirements for $g$. To simplify the notation, we take an arbitrary order of $\mathcal{H}_{A,t}$ and write $h_i(x)$ for its elements, $i=1,\ldots,\tilde{k}$, $\tilde{k}:=\frac{1}{4}n(n+1)(6+3n+n^2)$. Now, let $Q\subset\R^n$ be a cube containing $\mbox{supp}(g)$. Then, define a set of orthonormal polynomials $\left(  b_j(x)\right) _{j\in\N} $ with respect to the $L_2$-norm on $Q$, which is in ascend order with respect to the
degree. Moreover, define $\mathcal{P}_m:=\left\langle b_1(x),\ldots,b_m(x)\right\rangle $. Thus, the best $L_2$ approximation of $h_i(x)\in \mathcal{H}_{A,t}$ with polynomials up to a certain degree is given by
\begin{align}
 h_{i,m}=\sum\limits_{j=1}^ma_{i,j}b_j(x)\nonumber
\intertext{with}
a_{i,j}=\int_Qb_j(x)h_i(x)dx.\nonumber
\end{align}
% Therefore, if we shrink the linear map $f\mapsto(\alpha,\beta,\gamma,\epsilon,\delta)$ of formula (\ref{coefs})
% to the subspace $\mathcal{P}_m:=\left\langle b_1(x),\ldots,b_m(x)\right\rangle $ we get the representation matrix

% Therefore, if we shrink the linear map $f\mapsto(h_{1,m},\ldots,h_{\tilde{k},m})$ to the subspace $\mathcal{P}_m:=\left\langle b_1(x),\ldots,b_m(x)\right\rangle $, we get the representation matrix
Now, consider the matrix
\begin{align}
\tilde{M}= \left( 
\begin{array}{ccc}
 a_{1,1} & \cdots & a_{1,m}\\
\vdots & & \vdots\\
a_{\tilde{k},1}& \cdots & a_{\tilde{k},m}
\end{array}
\right). \nonumber
\end{align}
Since $\mathcal{H}_{A,t}$ is linear independent, there exists a $m\in\N$ such that all rows of $\tilde{M}$ 
are also linear independent. Due to the fact that $\left\langle\right.  \mathcal{H}_{A,t}\left. \right\rangle=\left\langle\right.  \mathcal{H}_{\tilde{A},\tilde{t}}\left. \right\rangle$ for all $(A,t),(\tilde{A},\tilde{t})\in G$, this number $m$ is independent of the choise of $(A,t)$. Since $\tilde{M}$ is the representation matrix of the map
\begin{eqnarray}
 \mathcal{P}_m\to\R^{\tilde{k}},~~~f\mapsto (\alpha,\beta,\gamma,\delta,\epsilon),\label{nn1}
\end{eqnarray}
we get the surjectivity of (\ref{nn1}) for each $(A,t)\in G$, if $m$ is big enough. Hence, using the formulas (\ref{begin_newton_SE(n)}) and (\ref{end_newton_SE(n)}) for $G=SE(n)$, and (\ref{f0})-(\ref{f6}) for $G=SA(n)$ respectively, one can easily verify that $j^2_{(A,t)}\hat{\Phi}\big|_{\mathcal{P}_m}$ is surjective for each chosen $(A,t)\in G\cap K_r(0)$.

Now, let $f\in C^3(\R^n,\R)$ be arbitrarily chosen. Like in the case of the cost function $\Psi$ before, we apply the Parametric Transversality Theorem to the map
\begin{eqnarray}
 F^{\mbox{ev}}:[f+\mathcal{P}_m]\times G\to J^2(G,\R)\nonumber\\
(f+p,A,t)\mapsto j^2_{(A,t)}\hat{\Phi}(f+p).\nonumber
\end{eqnarray}
Since $F^{\mbox{ev}}(~.~,A,t)$ is surjective for all $(A,t)\in G\cap K_r(0)$, the map $F^{\mbox{ev}}(f+p,~.~)$ is transverse to $\frak{U_i}$ for a dense subset of $\mathcal{P}_m$. Therefore, $A_{\Phi,i,r}$ is dense in $C^3_W(\R^n,\R)$ which completes the proof for the cost function $\Phi$.
\end{proof}

Finally, we want to note that the condition \textit{b)} of Theorem 7.4 is not very restrictive. Since the Region of Interest $Q$ is typically bounded, there exists a minimum value $m_g\in\R$ of $g$. Hence, we can consider the registration-problem
\begin{eqnarray}
 \min\limits_{(A,t)\in G}\int_{Q}(\tilde{g}(x)-\tilde{f}(Ax+t))^2dx\nonumber
\end{eqnarray}
 with $\tilde{g}(x)=g(x)+m_g+1$ and $\tilde{f}(x)=f(x)+m_g+1$ instead of (\ref{leastsquares}). For that new registration problem we have $\tilde{g}(x)\neq 0$ for all $x\in Q$.

With this result we can guarantee with a probability of one that the QMC-algorithm localy converges quadratically to a local maximum in the case of non-artificial generated images. However, Theorem 7.4 will not give any result for the idealistic case, in which the reference and template image are identical. Moreover, the case in which both images are elements of a spline function space, or smoothed by a Gaussian kernel, is not  applicable to this theorem. Even though these restrictions are more relevant  for applications, we believe that a further discussion in this direction is beyond the scope of this paper.

\section*{Acknowledgment}
This work has been supported by the Interdisciplinary Center for Clinical Research (IZKF) through the project F-37-N (Organ Tracking).

\bibliographystyle{siam}
\bibliography{640}
%\bibliography{JUL9.bbl}

\end{document}